\documentclass{amsart}

\newcommand{\Ci}{\mathscr{C}}
\newcommand{\Sone}{\mathbb{S}^1}

\newcommand{\N}{\mathbb{N}}
\newcommand{\R}{\mathbb{R}}
\newcommand{\C}{\mathbb{C}}
\newcommand{\Z}{\mathbb{Z}}

\newcommand{\T}{\mathbb{T}}

\newcommand{\De}{\mathscr{D}}

\newcommand{\ve}{\epsilon}

\newcommand{\Arg}{\operatorname{Arg}}

\newcommand{\vertiii}[1]{{\left\vert\kern-0.25ex\left\vert #1 
    \right\vert\kern-0.25ex\right\vert_{\mathscr{H}^m}}}
 
\newcommand{\esssup}{\operatorname*{ess\,sup}}
\newcommand{\essinf}{\operatorname*{ess\,inf}}

\usepackage{frcursive}

 \newcommand{\mathfrcbfslanted}[1]{\text{{\slshape\cursive \textcal{#1}}}}

 \newcommand{\ncal}{\mathfrcbfslanted{n}}

\newcounter{tmp}

\usepackage[initials]{amsrefs}
\usepackage{amsmath,amsthm,amscd}
\usepackage{enumerate}
\usepackage{enumitem}

\usepackage{amssymb}
\usepackage{mathrsfs}
\usepackage{graphicx}

\usepackage{color}
\usepackage{mathtools}

\theoremstyle{plain}
\newtheorem{thm}{Theorem}[section]
\newtheorem{cor}[thm]{Corollary}
\newtheorem{lem}[thm]{Lemma}
\newtheorem{prop}[thm]{Proposition}

\theoremstyle{definition}
\newtheorem{dfn}[thm]{Definition}

\theoremstyle{remark}
\newtheorem{remark}[thm]{Remark}
\newtheorem*{rmk}{Remark}

\numberwithin{equation}{section}

\title[Quenched limit theorems]{
Quenched limit theorems for random U(1) extensions of expanding maps
 }

\date{\today}

\author{Yong Moo Chung}
\address[Yong Moo Chung]{Department of Applied Mathematics, Hiroshima University, Higashi-Hiroshima,
739-8527, JAPAN}
\email{chung@amath.hiroshima-u.ac.jp }

\author{Yushi Nakano}
\address[Yushi Nakano]{Department of Mathematics, Tokai University, 4-1-1 Kitakaname, Hiratuka, Kanagawa, 259-1292, JAPAN}
\email{yushi.nakano@tsc.u-tokai.ac.jp}

\author{Jens Wittsten}
\address[Jens Wittsten]{Centre for Mathematical Sciences, Lund University, Box 118, SE221 00 Lund, Sweden, and Department of Engineering, University of Bor\r{a}s, SE-501
90 Bor\r{a}s, Sweden}
\email{jens.wittsten@math.lu.se}

\subjclass{Primary: 37H12, 37D30, 60F05; Secondary: 37C30}

\keywords{Central limit theorem, transfer operator, random dynamical system, partially hyperbolic map, Lyapunov spectrum}

\begin{document}

\begin{abstract}
In this paper we provide quenched central limit theorems,  
large deviation principles  and local central limit
theorems for random U(1) extensions of expanding maps on the torus. The results are obtained as special cases of corresponding theorems that we establish for abstract random dynamical systems. We do so by extending a recent spectral approach developed for quenched limit theorems 
for expanding and hyperbolic maps to be applicable also to partially hyperbolic dynamics. 
\end{abstract}

\maketitle

\section{Introduction}\label{section:introduction}

This paper concerns quenched limit theorems for random dynamical systems on a compact  smooth Riemannian manifold $M$. 
Given   a measure-preserving mapping $\sigma : \Omega \to \Omega$ on a probability space $(\Omega ,\mathcal F, \mathbb P)$, a \emph{random dynamical system} (abbreviated RDS henceforth) on $M$ over $\sigma$ is given as a measurable map $\mathcal T : \mathbb N_0\times \Omega \times M \to M$ 
satisfying $\mathcal T(0, \omega , \cdot ) = \mathrm{id}_M$ and 
$\mathcal T(n + m, \omega , x) = \mathcal T(n, \sigma ^m\omega , \mathcal T(m, \omega , x))$ for each $n, m \in \mathbb N_0$, $\omega \in \Omega$ and $x \in M$, 
with $\mathbb N_0=\{ 0,1,\ldots \}$. 
Here $\sigma \omega$ denotes the value $\sigma (\omega )$, and $\sigma$ is called a \emph{driving system}. 
A standard reference for random dynamical systems is the monograph by Arnold \cite{Arnold}. 
Here we merely recall that
if we denote $\mathcal T(n, \omega, \cdot )$ and $\mathcal T(1, \omega , \cdot )$ by $T^{(n)}_\omega$ and $T_\omega$, respectively, then
we have
\begin{equation}\label{eq:1102a}
\mathcal T(n, \omega , \cdot ) = T^{(n)}_\omega = T_{\sigma ^{n-1}\omega} \circ T_{\sigma ^{n-2}\omega} \circ   \cdots \circ T_{\omega}  .
\end{equation}
Conversely, given a measurable map $T: \Omega \times M \to M:
(\omega , x) \mapsto T_\omega(x)$, the map $\mathcal T : \mathbb N_0 \times \Omega \times M \to M$ defined by \eqref{eq:1102a} is an RDS over
$\sigma$. We call it the \emph{RDS induced by $T$ over $\sigma$}.

Given a measurable function $g: \Omega \times M\to \R : (\omega , x) \mapsto g_\omega (x)$, we consider the \emph{quenched} (i.e.~noisewise) Birkhoff sum $(S_ng )_\omega:  M\to \C$  of $g$, given by
\begin{equation}\label{eq:sng}
(S_n g)_\omega   = \sum _{j=0}^{n-1} g _{\sigma ^j\omega }\circ T^{(j)} _\omega 
\quad (n\geq 1, \; \omega \in \Omega).
\end{equation}
Our ultimate interest is the $\mathbb P$-almost sure asymptotic behavior of such Birkhoff sums as $n\to \infty$. Results on such behavior are known as {\it quenched limit theorems} (for the random process $\{ g _{\sigma ^n\omega }\circ T^{(n)}_\omega \} _{n\geq 0}$ in the random environment $\{ \sigma ^n\omega\}_{n\geq 0}$). 
Correspondingly, results on the asymptotic behavior of the expectation $\mathbb{E}_{\mathbb P}[S_n g]$ are referred to as \emph{annealed} (i.e.~averaged) limit theorems. 
For the history of annealed and quenched limit theorems for random dynamical systems,
we refer to \cites{ANV2015, DFGV2018, DFGV2019}.

In \cite{ANV2015}, Aimino et al.~gave a comprehensive study of   annealed limit theorems for abundant random dynamical systems (including random piecewise expanding maps) via   a  Nagaev-Guivarc'h type perturbative spectral approach for random dynamical systems (see Subsection \ref{ss:se} for a brief description of the approach). 
This was extended by  Dragi{\v{c}}evi{\'c} et al.~\cite{DFGV2018} to quenched schemes by using the newly-developed theory of \emph{Lyapunov spectrum} of random  operators (see e.g.~\cites{FLQ2013, GQ2014} for a presentation of this theory) and the detailed analysis of the \emph{leading}  Lyapunov exponent of the associated transfer operator.
This was recently further extended to random piecewise hyperbolic maps 
 \cite{DFGV2019}. 
In this paper we extend the technology in \cite{DFGV2018} to a class of partially expanding maps with a neutral direction, namely, to U(1) extensions of expanding maps on the circle.

U(1) extensions of expanding maps can be seen as toy models of piecewise hyperbolic flows such as geodesic flows on negatively curved manifolds or dispersive billiard flows (via suspension flows of hyperbolic maps, see \cite{hasselblatt2017,PP1990}), and has been intensively studied by several authors, see e.g.~\cite{Dolgopyat2002,Faure,FaureTsujii2,NW2015,NTW2016,BE2017,FW2017}. 
In particular, billiard flows are closely related to kinetic theory of gases, and the random small perturbation (of the ceiling function $\tau$ with noise parameter $\omega$ defined below) studied in this paper can be viewed as random small deformations of the particles in the gas (scatters); see the recent paper by Demers and Liverani \cite{demers2021projective}.

The difficulty in the analysis of these dynamics, compared with the expanding or hyperbolic maps studied in the previous works \cite{DFGV2018, DFGV2019}, is the existence of a  \emph{neutral direction}. In fact, for limit theorems to hold the dynamics typically needs to be rapid mixing, but partially expanding (hyperbolic) systems with a neutral direction are not even weakly mixing in general.  In \cite{NW2015}, the second and third author showed that under a generic condition, randomly perturbed U(1) extensions of expanding maps exhibit quenched exponential mixing, and the corresponding transfer operator cocycles have a Lyapunov spectral gap. In this paper we extend the Nagaev-Guivarc'h type perturbative spectral approach of  \cite{DFGV2018, DFGV2019} to random U(1) extensions of expanding maps, and estimates from \cite{NW2015} will play an essential role. To the best of our knowledge, the results in this paper are the first quenched limit theorems for random partially expanding (hyperbolic) systems with a neutral direction, see Theorems \ref{thm:decay}--\ref{thm:lclt5} below. These results are obtained as special cases of corresponding quenched limit theorems for abstract random dynamical systems, and we in fact reproduce the results of \cite{DFGV2018, DFGV2019} as a consequence of the abstract theorems  (see Theorems \ref{thm:decay2}, \ref{thm:clt2}, \ref{thm:ldp2} and \ref{thm:lclt2}). We hope that this analysis subsequently can be helpful in establishing quenched limit theorems for more complicated random dynamical systems.

\subsection{Random U(1) extensions of expanding maps on $\Sone$}\label{subsection:mr}
Let $e_0: \Sone \rightarrow  \Sone$ be a ${\Ci}^\infty$ orientation-preserving  diffeomorphism on the circle $\Sone =\mathbb R/ \mathbb Z$. Let $k\ge2$ be an integer and set $E_0(x)=ke_0(x)$ mod 1. Let
$\tau_0: \Sone \rightarrow \mathbb{R}$ be a ${\Ci}^\infty$ function. 
We consider the skew product $T_0:\T^2\to\T^2$ of class $\Ci^\infty$
given by
\begin{equation}\label{eq:unperturbedsystem}
T_0:\binom{x}{s}\mapsto
\left( \! \! \begin{array}{cc} E_0(x) & \\ s+ \tau_0(x) & \!\! \mod 1\end{array} \! \! \right)
\end{equation}
on the torus $\mathbb T^2=\mathbb R^2/ \mathbb Z^2$. The map
$E_0$ is assumed to be an {\it expanding map} on $\Sone$ in the sense that
$\min_x E_0'(x)>1$, and we then say that $T_0$ is the {\it U(1) extension} of the expanding map $E_0$ over $\tau _0$.
As mentioned above, this dynamical system can be seen as a toy model of (piecewise) hyperbolic flows, such as geodesic flows on negatively curved manifolds or dispersive billiard flows.

\begin{rmk}
We recall that if $\tau _0$ \emph{is cohomologous to a constant function} (i.e.~there is a smooth function $u$ such that $\tau_0 -u\circ E_0 + u $ is a constant function), then $T_0$ is not even weakly mixing  (\cite{Tsujii}; see also \cite{BE2017}) which implies that the corresponding transfer operator does not have a spectral gap (and several limit theorems thus fail to hold). This would prevent us from applying the Nagaev-Guivarc'h spectral method.
Therefore, in this paper we shall always assume that the function $\tau_0$ is not cohomologous to a constant function, which is known to be a generic condition \cite{Tsujii}. We do so by imposing the slightly stronger condition that $T_0$ is {\it partially captive} which is still a generic condition on $\tau_0$ once $E_0$ is fixed \cite{NTW2016}.
A precise description of the partial captivity condition is given  in Appendix \ref{app:B}.
\end{rmk}

Let $(\Omega ,  \mathcal F, \mathbb P)$ be  a probability space.
Let ${\Ci}^{\infty}(\mathbb{T} ^2,\mathbb{T} ^2)$ be the space of smooth  endomorphisms on $\mathbb{T} ^2$, endowed with the Borel $\sigma$-algebra.
Given a noise level $\epsilon \geq 0$, we let 
$T\equiv T(\epsilon ):\Omega \rightarrow {\Ci}^{\infty} (\mathbb{T} ^2,\mathbb{T} ^2) : \omega \mapsto T_\omega$  be a measurable map such that 
$T_\omega $ is for $\mathbb{P}$-almost every $\omega\in\Omega$ of the form  
\begin{equation}\label{eq:perturbedsystem}
T_\omega :\binom{x}{s}\mapsto
\left( \! \! \begin{array}{cc} E_\omega (x) & \!\!  \\ s+ \tau _\omega (x)  
& \!\! \mod 1
\end{array} \! \! \right),
\end{equation}
and $\omega \mapsto T_\omega (z)$ is a measurable mapping from $\Omega$ to $\T ^2$ for each $z=(x,s)\in \T ^2$.
Here $E_\omega :\mathbb{S} ^1 \to \mathbb{S} ^1$ 
is given by $E_\omega(x)=ke_\omega(x)$ mod 1 where $k\ge2$ is the same integer as before while $e_\omega:\mathbb{S} ^1 \to \mathbb{S} ^1$  is a ${\Ci}^{\infty}$ diffeomorphism 
and  $\tau  _\omega :\mathbb{S} ^1 \to \mathbb{R}$ is a ${\Ci}^{\infty}$ function,
$\mathbb{P}$-almost surely.
We also assume that  
\begin{equation}\label{convinC1}
\esssup_\omega d_{{\Ci}^{\infty}}(T_{\omega} ,T_0 ) \leq \ve ,
\end{equation}
where $T_0$ is the partially expanding map given by \eqref{eq:unperturbedsystem} and $d_{{\Ci}^{\infty}}$ is some choice of metric on ${\Ci}^{\infty} (\mathbb{T} ^2,\mathbb{T} ^2)$.
Note  that   $E_\omega $ is an expanding map for any sufficiently small $\epsilon >0$ and $\mathbb P$-almost every $\omega \in \Omega$ (see Remark \ref{rmk:rdependsonT0} for details), and  that  when $\epsilon =0$, $T_\omega  =T_0$ for $\mathbb P$-almost every $\omega \in \Omega$.

Let $\sigma : \Omega \to \Omega$ be a measurably invertible, ergodic, measure-preserving map on $(\Omega ,  \mathcal F, \mathbb P)$. 
Then $(\omega ,z) \mapsto T_\omega (z)$ 
is a measurable map from $\Omega \times \T ^2$ to $\T ^2$ (see \cite{CV77}*{Lemma 3.14}), so the map $\mathcal T$ induced by $T$ over $\sigma$ (recall \eqref{eq:1102a}) is an RDS on $\T^2$.
If $\mathcal B(\mathbb T^2)$ denotes the Borel algebra, it follows from \cite[Theorem 1.3]{NW2015} that when $\epsilon$ is sufficiently small,  there is a unique  function 
$ (\omega , A) \mapsto \mu _\omega (A)$ on $\Omega \times  \mathcal B(\mathbb T^2)$ 
such that
\begin{itemize}
\item[(i)] $\{ \mu_\omega \} _{\omega \in \Omega}$ is a measurable family of probability measures (i.e.~$\mu _\omega$ is $\mathbb P$-almost surely a probability measure on $\mathbb T^2$, and $\omega \mapsto \mu _\omega (A)$ is measurable for each $A\in \mathcal B(\mathbb T^2)$),
\item[(ii)] $\{ \mu_\omega \} _{\omega \in \Omega}$ is $\mathcal T$-invariant (i.e.~$\mu _\omega (T_\omega ^{-1} A) =\mu _{\sigma \omega} (A)$ for $\mathbb P$-almost every $\omega \in \Omega$ and all $A\in \mathcal B(\mathbb T^2)$),
\item[(iii)] $\mu _\omega $ is $\mathbb P$-almost surely absolutely continuous. 
\end{itemize}
We  define a probability measure $\mu$ on $\Omega \times \mathbb T^2$ by $\mu (\Gamma \times A)=\int _\Gamma \mu _\omega (A) \mathbb P(d\omega )$ for each measurable set $\Gamma \subset \Omega$ and $A\subset \mathbb T^2$.

Let $r$ be a positive constant and let $g$ be in $L^\infty (\Omega ,H^r (\mathbb T^2) )$, that is, $g$ is a measurable map from $\Omega$ to the Sobolev space $H^r (\mathbb T^2)$
such that 
$\esssup _\omega \Vert g_\omega \Vert  _{H^r}<\infty$, where $g_\omega =g(\omega )$.  For simplicity, we will assume  that
\begin{equation}\label{eq:0419b}
\mathbb E_{\mu _\omega }\left[ g_\omega \right]=0 \quad \text{for $\mathbb P$-almost every $\omega \in \Omega$.}
\end{equation}
(Note that if $g\in L^\infty (\Omega ,H^r (\mathbb T^2))$, then $\widetilde g_\omega = g_\omega - \mathbb E_{\mu _\omega } [g_\omega]$, called the centering, satisfies $\mathbb E_{\mu _\omega } [\widetilde g_\omega ]=0$.) 
We define the asymptotic variance $V$  for $g$ by
\begin{equation}\label{eq:V}
 V  =\mathbb E_{\mathbb P}[v], \quad v_\omega  =
 \mathbb E_{\mu _\omega }\left[ g_{ \omega } ^2  + 2 \sum _{n=1}^\infty  g_{ \omega } \cdot g_{\sigma ^n\omega }\circ T_{\omega } ^{(n)} \right].
\end{equation}
We say that $g$ is \emph{non-degenerate} if $V >0$, which was shown in \cite{DFGV2018b} to be equivalent
to the non-coboundary condition of $g$ (i.e.~non-existence of $\phi \in L^2 (\Omega \times M)$ such that $g_\omega  =\phi _\omega  - \phi _{\sigma \omega }  \circ T_\omega$) when $T_\omega$ is a piecewise expanding map.

Exponential decay of quenched correlation functions for the RDS $\mathcal T$ with respect to $\{ \mu _\omega \} _{\omega \in \Omega} $ was  investigated in \cite{NW2015}  via a spectral approach. 
As a by-product of key estimates in the analysis of this paper we obtain the following improved version.
\begingroup
\setcounter{tmp}{\value{thm}}
\setcounter{thm}{0}
\renewcommand\thethm{\Alph{thm}}
\begin{thm}\label{thm:decay}
Let $T_0$ as well as $\sigma$, $T$ and $g$ be as above. In particular, assume that $T_0$ is partially captive. 
Then, for any sufficiently small $\epsilon \geq 0$,  quenched
correlation functions decay exponentially fast in the following sense: There are constants  $\rho \in (0,1)$, $m>1$ and $C_{g}>0$ such that if $r\ge m$ and $u \in H^r(\mathbb T^2)$ then
\[
\left\vert 
\mathbb E_{\mu_\omega} \big[g _{\sigma ^n\omega}\circ T_\omega ^{(n)} u \big] - \mathbb E_{\mu_\omega} \big[g  _{\sigma ^n\omega} \circ T_\omega ^{(n)} \big] \mathbb E_{\mu_\omega} \left[ u \right] \right\vert \leq C_{g}\rho ^n \Vert u\Vert _{H^r}
\]
for all $n\geq 1$ and $\mathbb P$-almost every $\omega \in \Omega$. 
\end{thm}
\endgroup

We remark that in our previous study, the coefficient $C_{g}$ above
depended on $\omega$ and we were unable to remove this dependency. 
By a more careful analysis we are now able to overcome this problem by using a different spectral scheme for a global transfer operator instead; a key step is to establish \eqref{eq:0212c2} in Section \ref{section:proof}.

Note also that Theorem \ref{thm:decay} implies exponential decay of annealed correlation functions and moreover that the quenched law of large numbers (or ergodicity) holds, see the discussion following Theorem \ref{thm:decay2} below.

The main results of the paper are the following limit theorems:
\begingroup
\setcounter{tmp}{\value{thm}}
\setcounter{thm}{1}
\renewcommand\thethm{\Alph{thm}}
\begin{thm}[Central limit theorem]\label{thm:clt5}
Assume that the conditions in Theorem \ref{thm:decay} are satisfied.
Assume also that $g$ is non-degenerate.
Then, for any sufficiently small $\epsilon \geq 0$
and $\mathbb P$-almost every $\omega \in \Omega$, 
$\frac{(S_ng)_\omega }{\sqrt n}$ converges in distribution to a normal distribution with mean $0$ and variance $V$ as $n\to \infty$, i.e.
for any $a\in \mathbb R$, we have
\[
\lim _{n\to \infty} \mu _\omega \left(  (S_ng ) _\omega \leq a\sqrt n\right) = \frac{1}{\sqrt{2\pi V}}\int _{-\infty }^a e^{-\frac{z^2}{2V}}dz .
\]
\end{thm}

 \begin{thm}[Large deviation principle]\label{thm:ldp5}
Assume that the conditions in Theorem \ref{thm:clt5} are satisfied.
Then, for any sufficiently small $\epsilon \geq 0$, 
 $\mathbb P$-almost every $\omega \in \Omega$ and any sufficiently small $\delta >0$, 
\begin{equation*}
\lim _{n\to \infty} \frac{1}{n} \log \mu _\omega ( (S_ng)_\omega >n\delta ) = - c (\delta ) ,
\end{equation*}
where $c$ is nonnegative, continuous, strictly convex on a neighborhood of $0$, vanishing only at $0$. 
\end{thm}

 \begin{thm}[Local central limit theorem]\label{thm:lclt5}
Assume that the conditions in Theorem \ref{thm:clt5} are satisfied.
Assume also that condition $(\mathrm{L})$ holds (see Definition \ref{def:L}). Then for $\mathbb P$-almost every $\omega \in \Omega$ and every bounded interval $[a , b] \subset \mathbb R$ we have
\[
\lim _{n\to \infty} \sup _{t\in \mathbb R} \sqrt n \left\vert \mu _\omega \left( t+a\leq (S_ng ) _\omega \leq t+b \right) - \frac{1}{\sqrt{2\pi V}} e^{-\frac{t^2}{2nV}} \vert b-a \vert \right\vert  =0 .
\]
\end{thm}

\endgroup
\setcounter{thm}{\thetmp}

We will prove Theorems A--D by establishing corresponding limit results for an abstract class of random dynamical systems. In Section \ref{sec:2} we present the abstract setting and state the general limit theorems (Theorems \ref{thm:decay2}, \ref{thm:clt2}, \ref{thm:ldp2} and \ref{thm:lclt2}), and also describe how random U(1) extensions of expanding maps on the circle is included within this class (see Theorem \ref{thm:v}). The verification of this inclusion is made in Section \ref{section:proof}. We then prove the general limit theorems in Section \ref{section:proofofabstract}.

\section{Limit theorems for abstract random dynamical systems}\label{sec:2}

\subsection{The abstract setting}\label{section:abstract}

Let $(M, \mathcal G)$ be a measurable space, where $\mathcal G$ is  a $\sigma$-algebra on $M$.
 We consider   a topological vector space $\mathscr D$ consisting of complex-valued  functions  on $M$, and a 
 map $T: M\to M$ preserving $\mathscr D$  (i.e.~$u \circ T\in \mathscr D$ for any $u \in \mathscr D$). 
Then, we define 
the (Perron-Frobenius) \emph{transfer operator} $\mathcal L_T$ of $T$ on the (continuous) dual space $\mathscr D^\prime$  of $\mathscr D$
by
\begin{equation}\label{eq:2}
\langle \mathcal L _T\varphi  , u \rangle=\langle \varphi ,  u\circ T\rangle,
\quad \varphi \in\De',\quad u \in \De ,
\end{equation}
where $\langle\phantom{i},\phantom{i}\rangle$ is the dual pairing.
When $T$ is a smooth map on  a compact smooth Riemannian manifold $M$ equipped with the normalized Lebesgue measure   $\mathrm{Leb}_M$, and 
 $\mathscr D= L^1(M) $, it follows from  the canonical form $L^1 (M) \ni u\mapsto \langle \varphi ,u \rangle = \int \varphi  u \,d\mathrm{Leb} _M$ of $\varphi \in L^\infty(M)  $ as an element of $\left( L^1(M) \right)^\prime$ that \eqref{eq:2} is equivalent to 
$
\int \mathcal L_T\varphi \cdot u \,d\mathrm{Leb} _M= \int  \varphi \cdot u\circ T \,d\mathrm{Leb}_M$ for each $\varphi \in L^\infty (M)$ and $ u\in L^1(M).
$
Hence, if $\det DT(x) \neq 0$ Lebesgue almost everywhere, then by a change of variables one can get
 \begin{equation}\label{eq:1}
\mathcal L _T \varphi (x) = \sum _{T(y) =x} \frac{\varphi (y)}{\vert \det DT(y) \vert } 
\end{equation}
for each $\varphi \in L^\infty (M)$ and Lebesgue almost every $x\in M$. 
Due to the relation \eqref{eq:2}, one may expect that several statistical properties of $T$ directly follow from properties of the spectrum of $\mathcal L$. 
Standard references for transfer operators 
 are the monographs by Baladi \cites{Baladi00,Baladibook2}.

In some literature, the transfer operator of $T$ is defined as the bounded operator on a space of functions on  a  smooth manifold  $M$
 given by \eqref{eq:1}. However,
in the last two decades it has been realized that it is  important to investigate the spectrum of the transfer operator on  Banach  spaces of \emph{distributions}
if one hopes to obtain information on
the Sinai-Ruelle-Bowen measures of the dynamics, which are the most relevant measures in smooth dynamical systems theory. 
We thus employ relation \eqref{eq:2} for the definition of transfer operators and interpret it in the distributional sense whenever appropriate.

 Let $(\Omega , \mathcal F, \mathbb P)$ and $\sigma$ be as in Subsection \ref{subsection:mr}.
 Let $\mathscr M $  be a measurable space consisting of  maps preserving $\mathscr D$, and $T: \Omega \to \mathscr M : \omega \mapsto T_\omega$ a measurable map such that $(\omega , x)\mapsto T_\omega (x)$ is a measurable map from $\Omega \times M$ to $M$. 
 We simply write $\mathcal L_{\omega }$ for the transfer operator $\mathcal L_{T_\omega} : \mathscr D^\prime \to \mathscr D^\prime $ of $T_\omega $.

\begin{dfn}\label{def:admissible}
We say that a Banach space $\mathscr B  \subset  \mathscr D^\prime $ with a norm $ \Vert \cdot \Vert $ is \emph{$T$-admissible} if $\mathcal L_\omega$  is bounded on $ \mathscr B$ for $\mathbb P$-almost every $\omega \in \Omega$, and  there exists a constant $C_0>0$ such that  
\begin{equation}\tag*{$\mathrm{(A1)}$}\label{(A1)}
 \Vert \mathcal L_\omega \varphi \Vert \leq C_0\Vert \varphi \Vert,  \quad\varphi \in \mathscr B, \quad \text{ $\mathbb P$-almost every $\omega$.}
\end{equation}
\end{dfn}
Let $\mathcal Z(\mathscr D) $ be a linear subspace of $\mathscr D$ given by  
$$
\mathcal Z(\mathscr D) = \{ u  \in \mathscr D \mid uv  \in \mathscr D \;  \text{for any $v \in \mathscr D$} \}.
$$
For each $u  \in \mathcal Z(\mathscr D)$ and $\varphi \in \mathscr D^\prime$, 
 define $u \varphi \in \mathscr D^\prime$ via duality by 
 \begin{equation}\label{eq:1104a}
 \langle u  \varphi ,  v \rangle = \langle  \varphi , uv    \rangle, \quad v \in \mathscr D.
 \end{equation}
Let $1_A$ be the indicator function of $A \subset M$. 

\begin{dfn}\label{def:associated}
We say that a vector space $\mathscr E  $ with a seminorm $\Vert \cdot \Vert _{\mathscr E}$ is
\emph{associated with 
 $(\mathscr  B, \Vert \cdot \Vert )$} if  $\mathscr E\subset \mathcal Z(\mathscr D)$, 
 $1_M\in \mathscr E$ and
  there exists a constant $C_1>0$  such that 
\begin{equation}\tag*{$\mathrm{(A2)}$}\label{(A2)}
\begin{aligned}  \max \left\{ \langle \varphi , u \rangle  , \Vert u \varphi  \Vert \right\} &\leq C_1 \Vert \varphi \Vert \Vert u\Vert _{\mathscr E},\quad \varphi\in\mathscr B,\ u\in\mathscr E,\\
\Vert uv\Vert _{\mathscr E} &\leq C_1 \Vert u\Vert _{\mathscr E} \Vert v\Vert _{\mathscr E},\quad u, v \in \mathscr E.
\end{aligned}
\end{equation}
We call each element  in $\mathscr E$ an \emph{observable}.
\end{dfn}

We now fix a seminormed vector space $(\mathscr E , \Vert \cdot \Vert _{\mathscr E})$ associated with a $T$-admissible Banach space $(\mathscr  B , \Vert \cdot \Vert )$. 
We define a cocycle $(n, \omega , \varphi ) \mapsto \mathcal L_\omega ^{(n)} \varphi$ 
by
\[
 \mathcal L_{\omega} ^{(n)} 
= \mathcal L_{\sigma ^{n-1}\omega } \circ \mathcal L_{ \sigma ^{n-2} \omega } \circ \cdots \circ \mathcal L_{\omega} , \quad n\geq 1,
\]
 and
 $\mathcal L_{\omega }^{(0)}=\mathrm{id} _{\mathscr B }$  for each $\omega\in \Omega$,
which  is called the 
 \emph{transfer operator cocycle} (of the RDS $\mathcal T $ induced by $T$ over $\sigma$), and denoted by $(\mathcal L ,\sigma)$ for simplicity.

For each topological space $\mathscr V$, we denote by
 $\mathcal B_{\mathscr V}$  the Borel $\sigma$-algebra of  $\mathscr V$. 
 In particular, $\mathcal B_{L(\mathscr B)}$ is the Borel $\sigma$-algebra generated by the norm topology on $L(\mathscr B)$, where $L(\mathscr B)$ is the space of all bounded operators on $\mathscr B$ equipped with the operator norm. 
When $\mathscr B$ is separable, we   also consider  
the $\sigma$-algebra $\mathcal S_{L(\mathscr B)}$ generated by the strong operator topology on $L(\mathscr B)$, see \cite[Appendix A]{GQ2014} for its definition and basic properties.
Here we merely recall that a map $\mathcal A:\Omega \to L(\mathscr B)$ is $(\mathcal F, \mathcal S_{L(\mathscr B)})$-measurable if and only if  $\mathcal A$ is strongly measurable, that is, for any $\varphi \in \mathscr B$, the map $\omega \mapsto \mathcal A(\omega )(\varphi )$ is  $(\mathcal F, \mathcal B_{\mathscr B})$-measurable.
Furthermore, as pointed out in \cite{GQ2014} (see also \cites{BY93,Bogenschutz00}),  the $(\mathcal F, \mathcal B_{L(\mathscr B)})$-measurability of $\mathcal L$ is a 
 strong  requirement (at least stronger than the $(\mathcal F, \mathcal S_{L(\mathscr B)})$-measurability). 
 Indeed, \cites{DFGV2018,DFGV2019} ensured  the $(\mathcal F, \mathcal B_{L(\mathscr B)})$-measurability of $\mathcal L$ by assuming that $\{ T_\omega \mid \omega \in \Omega \}$ is at most countable.
The main hypotheses which we shall place on an abstract RDS in this setting are the following two spectral conditions.

\begin{dfn}[Uniform spectral gap condition]
We say that $T$ satisfies the {\it uniform spectral gap condition} (UG) with respect to $\mathscr B$ and $\mathscr E$ if the following holds: There exists a unique $h\in L^\infty(\Omega , \mathscr B)$ and a measurable family of probability measures $\{ \mu_\omega \} _{\omega \in \Omega}$ on $M$ 
such that 
\[
\mathcal L_{\omega } h_{\omega} =h_{\sigma \omega}, \quad \langle h_{\omega }, u \rangle = \int u \,d\mu _\omega , \quad \essinf _{\omega \in \Omega } \Vert h_{\omega}\Vert \geq 1
\]
 for $\mathbb P$-almost every $\omega \in \Omega$ and any $u \in \mathscr E$.
Furthermore,
there is a constant $C_2>0$ and $\rho \in (0,1)$ such that if $\varphi \in \mathscr B$ satisfies $\langle \varphi ,1_M \rangle =0$, then
\begin{equation}\label{eq:0212c2}
\esssup _{\omega \in \Omega} \Vert \mathcal L_\omega^{(n)} \varphi \Vert \leq C_2\rho ^n \Vert \varphi \Vert \quad \text{for all $n\in \mathbb N$.}
\end{equation}
\end{dfn}

\begin{dfn}[Lasota-Yorke inequality condition]\label{dfn:LY}
We say that $T$ satisfies the {\it Lasota-Yorke inequality condition} (LY) with respect to $\mathscr B\subset\mathscr B_+$ if the following holds: either
\begin{enumerate}
\item[$\mathrm{(i)}$]  $(\Omega , \mathcal F, \mathbb P)$ is a Polish space, 
  $\sigma$ is a homeomorphism,
   and $\mathcal L$ is $(\mathcal F, \mathcal B_{L(\mathscr B)})$-measurable,
   or
\item[$\mathrm{(ii)}$]  $(\Omega , \mathcal F, \mathbb P)$ is a Lebesgue-Rokhlin probability space, $\mathscr B$ is separable,  and  $\mathcal L$ is $(\mathcal F, \mathcal S_{L(\mathscr B)})$-measurable.
\end{enumerate}
Furthermore, $(\mathscr B_+ , | \cdot |)$ is a Banach space with $\vert \cdot \vert \leq \Vert \cdot \Vert$ and $\mathscr B \subset  \mathscr B_+ \subset \mathscr D^\prime$, and it is required that there exist a positive integer $n_0$ and random variables $\alpha  , \beta $ with values in $ \mathbb R_+=\{ x>0\}$  such that 
 the inclusion $(\mathscr B, \Vert  \cdot  \Vert ) \hookrightarrow (\mathscr B_+, | \cdot |)$ is compact, $\mathcal L$ is $\mathbb P$-almost surely bounded  on $\mathscr B_+$, and 
\begin{equation}\label{eq:forProp35}
\Vert \mathcal L_\omega ^{(n_0)}\varphi \Vert  \leq  \alpha _\omega   \Vert  \varphi \Vert + \beta_\omega \vert \varphi \vert, \quad \varphi \in \mathscr B ,\quad \text{$\mathbb P$-almost every $\omega \in \Omega$},
\end{equation}
with
$
\mathbb E_{\mathbb P}\left[ \alpha\right] < 1$ 
and $ \mathbb E_{\mathbb P}\left[ \beta \right] <\infty .
$
\end{dfn}

\begin{rmk}
The expectation $\mathbb{E}_\mathbb{P}[u]$  of $u$ with respect to a probability measure $\mathbb P$ is $\mathbb{E}_\mathbb{P}[u]=\int u\, d\mathbb P$. 
It follows from \eqref{eq:2} and (UG) that 
\begin{equation}\label{eq:integrating}
\mathbb E_{\mu_\omega}[u\circ T_\omega]=\mathbb E_{\mu_{\sigma\omega}}[u]
\end{equation}
for any  $u\in \mathscr E$.
 Moreover, for several random dynamics  including random U(1) extensions of expanding maps in Subsection \ref{subsection:mr}, 
 one can also show that  $\{ \mu _\omega \} _{\omega \in \Omega}$ is $T$-invariant, i.e.
\[
\mu_\omega (T_\omega ^{-1} A) = \mu _{\sigma \omega } (A)
\]
for $\mathbb P$-almost every $\omega \in \Omega$ and any  $A\in \mathcal G$ (or equivalently, \eqref{eq:integrating} with $\{ 1_A \mid A\in \mathcal G\}$ in place of $\mathscr E$), although  in this paper we will not   use this property 
 (use \eqref{eq:integrating} instead)  to prove quenched limit theorems. 
Another remark on (UG) is  that for any positive $h_\omega\in (\mathscr C^0 (M) )^\prime$  with $ \langle h_{\omega }, 1_M \rangle = 1$, 
 there exists a unique probability measure $\mu _\omega$ such that $ \langle h_{\omega }, u \rangle = \int u \,d\mu _\omega$ by the Riesz-Markov-Kakutani representation theorem,  
and  that  $\essinf _{\omega \in \Omega } \Vert h_{\omega}\Vert \geq 1$ if $\Vert \cdot \Vert \geq \Vert \cdot \Vert _{L^1}$.

The setting of (i) in (LY) is required in order to apply the Froyland-Lloyd-Quas version of multiplicative ergodic theorem  (\cite{FLQ2013}), while the setting of (ii) comes from the Gonz\'alez-Tokman-Quas version of multiplicative ergodic theorem (\cite{GQ2014}), see Subsection \ref{ss:se}.
We also remark that it follows from \cite[Remark 13]{FLQ2013} that when $(\Omega , \mathcal F, \mathbb P)$ is a Polish space, the $(\mathcal F, \mathcal B_{L(\mathscr B)})$-measurability of $\mathcal L$ is equivalent to the $\mathbb P$-continuity of $\mathcal L$ (i.e.~there is a countable collection of Borel sets $\{ A_n\}_{n\in \mathbb N}$ such that $\bigcup _{n\in \mathbb N} A_n$ has $\mathbb P$-full measure and the restriction of $\mathcal L$ on each $A_n$ is continuous),  which is exactly the condition 
used  in \cites{DFGV2018, DFGV2019}.
\end{rmk}

\begin{rmk}\label{rmk:1103}
The spectral assumption (UG) may seem to be much stronger than (LY), and for some spaces $\mathscr D$, the inequality \eqref{eq:forProp35} of (LY) does indeed follow directly from (UG).
For example, consider the case when $M$ is  a   locally compact space and  $\mathscr D$ is the space of all compactly supported continuous functions on $M$ (with the standard locally convex topology), for which  any  element $\varphi\in  \mathscr D ' $ is   the difference $\varphi _+ - \varphi _-$ of two Radon measures $\varphi _+$ and $\varphi _-$ by  Riesz-Markov-Kakutani representation theorem.
Then
 \begin{equation}\label{eq:211022}
  \mathscr B_+:= \left\{ \varphi \in \mathscr D^\prime \mid \vert \varphi \vert <\infty \right\} \quad \text{with} \quad \vert \varphi \vert =\varphi _+ + \varphi _-
\end{equation}
is a Banach space, and   $\vert \langle \varphi , 1_M\rangle \vert \leq \vert \varphi \vert $. Let $\mathscr B$, $h$, $C_2$, $\rho$ be as in (UG), and define projections $\pi _\omega , \tilde \pi _\omega$ on $\mathscr B$ by
\[
\pi _\omega (\varphi ) = \langle \varphi , 1_M\rangle h_\omega, \quad \tilde \pi _\omega (\varphi )= \varphi - \pi _\omega (\varphi ).
\]
Then it follows from (A2) and (UG) that, with $H=\esssup _\omega \Vert h_\omega \Vert  <\infty$,
\begin{align*}
\esssup _{\omega \in \Omega } \Vert \mathcal L_\omega^{(n)}  \pi _\omega (\varphi )\Vert \leq  \esssup _{\omega\in\Omega}  \left\vert  \langle  \varphi  ,1_M\rangle \right\vert  \Vert h _{\sigma ^n\omega} \Vert 
 \leq   H\vert    \varphi   \vert  
\end{align*}
for each $\varphi \in \mathscr B$.
On the other hand,
since $
\langle \tilde \pi _\omega (\varphi ),1_M\rangle = 0$,
 by (A2) and (UG) we have  
\begin{align*}
\esssup _{\omega \in \Omega } \Vert \mathcal L_\omega^{(n)} \tilde \pi _\omega (\varphi )\Vert \leq C_2 \rho ^n \left\Vert \tilde \pi _\omega (\varphi ) \right\Vert  
 \leq C' \rho ^n  \Vert    \varphi   \Vert 
\end{align*}
with $C'=C_2 (1+H \Vert 1_M\Vert _{\mathscr E})$.
Hence, noticing that $\varphi = \pi _\omega (\varphi ) + \tilde \pi _\omega (\varphi )$, we get
\begin{align*}
\esssup _{\omega \in \Omega } \Vert \mathcal L_\omega^{(n)}  \varphi \Vert
\leq  C' \rho ^n \Vert    \varphi   \Vert  + H  \vert    \varphi   \vert ,
\end{align*}
that is, the inequality \eqref{eq:forProp35}.

However, the ``absolute value'' operation used in the example above does not immediately generalize for general $\mathscr D'$, so it is not clear to us if there is always a natural choice of weak norm $\vert \cdot\vert$ for which inequality \eqref{eq:forProp35} of (LY) follows by virtue of (UG) alone. Furthermore, 
 the compactness assumption of $(\mathscr B, \Vert  \cdot  \Vert ) \hookrightarrow (\mathscr B_+, | \cdot |)$ in (LY) is not true in general for the weak norm $\vert \cdot \vert$ defined in \eqref{eq:211022} 
(e.g.~the case when $\mathscr B=\mathscr D' =L^2(M)$, where $\mathscr B_+$  in \eqref{eq:211022}  automatically coincides with $\mathscr B$). 
\end{rmk}

The abstract setting considered here is quite natural. In addition to random U(1) extensions of expanding maps on the circle. dynamical systems satisfying conditions (UG) and (LY) can be found in \cites{DFGV2018, DFGV2019}:
\begin{itemize}\label{listofexamples}
\item For random piecewise expanding maps considered in \cite{DFGV2018}, 
the conditions can be verified with   $\Vert \cdot \Vert _{\mathscr E}
 =\Vert \cdot \Vert _{L^1} + \mathrm{var} (\cdot ) + \Vert \cdot \Vert _{ L^\infty }$, $\Vert \cdot \Vert =
 \Vert \cdot \Vert _{L^1} + \mathrm{var} (\cdot )$ and $\vert \cdot \vert = \Vert \cdot \Vert _{L^1}$, where $\mathrm{var} (\cdot )$ is the total variation.
\item
For random hyperbolic maps considered in \cite{DFGV2019},
the conditions are verified with  $\Vert \cdot \Vert _{\mathscr E} = \Vert \cdot \Vert _{\mathscr C^r}$   with $r>2$, 
$\Vert \cdot \Vert  =\Vert \cdot \Vert _{1,1}$ and $\vert \cdot \vert =\Vert \cdot \Vert _{0,2}$, where $\Vert \cdot \Vert _{p,q}$ is the Gou\"ezel-Liverani's scale of norms  given in (2.2) of \cite{GL2006}.
\end{itemize}

\begin{rmk}
For the above examples, the most difficult  condition to verify is \eqref{eq:0212c2} in (UG): note that the coefficient $C_2$ in \eqref{eq:0212c2} needs to be independent of $\omega$.
The techniques to show \eqref{eq:0212c2} in \cite{DFGV2018}, \cite{DFGV2019} and this paper are quite different. 
In particular, we will prove \eqref{eq:0212c2} for random U(1) extensions of expanding maps without using any abstract perturbation lemma, see Section \ref{section:proof} and compare  with  \cite[(10)]{DFGV2019}.

We also remark that for the analysis of random U(1) extensions of expanding maps, we  employ the setting (ii) in (LY). 
This is a difference from  \cite{DFGV2018} and the piecewise hyperbolic part of \cite{DFGV2019}, that used the setting (i)  and needed in practice the   restriction that  $\{ T_\omega \mid \omega \in \Omega \}$ is countable.
\end{rmk}

\subsection{Limit theorems}
\label{ss:results}

Let $\{ \mu _\omega \} _{\omega \in \Omega}$ be the measurable  family of probability measures on $M$ provided by (UG).
Define a probability measure $\mu$ on $\Omega \times M$ by $\mu (\Gamma \times A)=\int _\Gamma \mu _\omega (A) \mathbb P(d\omega )$ for each measurable set $\Gamma \subset \Omega$ and $A\subset M$. 
Fix also
 $g\in L^\infty (\Omega , \mathscr E )$ (i.e.~$g$ is a measurable map from $\Omega$ to $\mathscr E $ with $ \esssup _\omega \Vert g_\omega \Vert _{\mathscr E} <\infty$), and assume that the centering condition \eqref{eq:0419b} holds for $g$.
Our first main result concerns exponential decay of correlation functions.

\begin{thm}[Exponential decay of correlations]\label{thm:decay2}
Let $(\mathscr E , \Vert \cdot \Vert _{\mathscr E})$ be a seminormed vector space  associated with a $T$-admissible Banach space $(\mathscr  B , \Vert \cdot \Vert )$. 
Let $g\in L^\infty (\Omega , \mathscr E)$ satisfy \eqref{eq:0419b}.
Assume that $(\mathrm{UG})$ holds. Then, one can find a constant $C_g >0$ such that for all $u \in \mathscr E$, $n\geq 1$ and $\mathbb P$-almost every $\omega \in \Omega$,
\begin{equation}\label{eq:1114decay}
\left\vert 
\mathbb E_{\mu_\omega} \big[g _{\sigma ^n\omega}\circ T_\omega ^{(n)}  u \big] 
\right\vert \leq C_g \rho ^n \Vert u \Vert _{\mathscr E} .
\end{equation}
\end{thm}

\begin{proof} 
By (UG) we have $\int u \,d\mu_{\omega}=\langle h_\omega,u\rangle$ for $u\in\mathscr E$, so
$$
\mathbb E_{\mu_\omega} \big[g _{\sigma ^n\omega}\circ T_\omega ^{(n)}  u \big] =\int ug _{\sigma ^n\omega}\circ T_\omega ^{(n)}  \, d\mu_\omega 
=\langle uh_\omega,  g _{\sigma ^n\omega}\circ T_\omega ^{(n)} \rangle,
$$
where the right-hand side equals $\langle \mathcal L^{(n)}_\omega(uh_\omega),  g _{\sigma ^n\omega}\rangle$ by \eqref{eq:2}.  We would like to use \eqref{eq:0212c2} to estimate $\mathcal L^{(n)}_\omega(uh_\omega)$, but $\langle uh_\omega,1_M\rangle\neq0$ in general so this requires modification. Thus, we note that \eqref{eq:integrating} and the centering assumption \eqref{eq:0419b} gives
$$
\langle \mathcal L^{(n)}_\omega(h_\omega),  g _{\sigma ^n\omega}\rangle
=\int g_{\sigma^n\omega}\circ T^{(n)}_\omega\, d\mu_\omega
=\int g_{\sigma^n\omega}\, d\mu_{\sigma^n\omega}=\mathbb{E}_{\mu_{\sigma^n\omega}}[g_{\sigma^n\omega}]=0,
$$
which means that 
\begin{align*}
\mathbb E_{\mu_\omega} \big[g _{\sigma ^n\omega}\circ T_\omega ^{(n)}  u \big]
&=\langle \mathcal L^{(n)}_\omega(uh_\omega),  g _{\sigma ^n\omega}\rangle-\langle \mathcal L^{(n)}_\omega(h_\omega),  g _{\sigma ^n\omega}\rangle\mathbb{E}_{\mu_\omega}[u]\\
&=\langle \mathcal L^{(n)}_\omega((u-\mathbb{E}_{\mu_\omega}[u])h_\omega),  g _{\sigma ^n\omega}\rangle.
\end{align*}
Since $\langle (u-\mathbb{E}_{\mu_\omega}[u])h_\omega,1_M\rangle=\langle h_\omega,u\rangle-\mathbb{E}_{\mu_\omega}[u]\langle h_\omega,1_M\rangle=0$,
we can now use
\ref{(A2)} and \eqref{eq:0212c2} to obtain
\begin{align*}
\left\vert \mathbb E_{\mu_\omega} \big[g _{\sigma ^n\omega}\circ T_\omega ^{(n)}  u \big]\right\vert&\le C_1
\lVert \mathcal L^{(n)}_\omega((u-\mathbb{E}_{\mu_\omega}[u])h_\omega)\rVert \lVert  g _{\sigma ^n\omega}\rVert_\mathscr{E}\\
&\le C_1C_2\rho^n G(1+H)H\lVert u\rVert_\mathscr{E}
\end{align*}
where $G=\esssup _\omega \Vert g_\omega \Vert _{\mathscr E}$ and 
$H = \esssup _\omega \Vert h_\omega \Vert$. 
This completes the proof.
\end{proof}

Obviously, Theorem \ref{thm:decay2} implies exponential decay of annealed correlation functions:
\begin{equation}\label{eq:0419b2}
\left\vert 
\mathbb E_{\mu} \big[g \circ \mathsf T^n  u \big] 
\right\vert 
  \leq C_{g}\rho ^n \Vert u\Vert _{\mathscr E} 
\end{equation}
for every $u \in \mathscr E$ and  $n\geq 1$,
where $\mathsf T: \Omega \times M\to \Omega \times M$ is 
the skew-product map 
induced by $(T, \sigma )$, 
 i.e.~$\mathsf T(\omega , x)=(\sigma \omega , T_\omega (x))$ for $(\omega , x) \in \Omega \times M$.
 Furthermore, 
it easily follows from Theorem \ref{thm:decay2} that for $\mu$-almost every $(\omega , x)\in \Omega \times M$, 
\[
\lim _{n\to \infty} \frac{(S_ng)_\omega (x)}{n} = 0,
\]
 if 
for each $A\in \mathcal G$, $\tilde u\in \mathscr E$ and $\mathbb P$-almost every $\omega $  there is a sequence $( u_n)_{n\geq 1} \subset \mathscr E$ such that $\mathbb E_{\mu _\omega }\left[\tilde u(1_A-u_n) \right] \to 0$ as $n\to \infty$ (the examples of dynamical systems on page \pageref{listofexamples} indeed satisfy this condition).  
That is, the quenched strong law of large numbers
(or ergodicity)
   holds under a mild condition. 
   We refer to \cite{BDV04} for the relation between fast decay of correlation functions and various limit theorems for deterministic systems.

To state our other results,  we further assume that  $g$ is $\mathbb P$-almost surely a real-valued function.
As before, we define the asymptotic variance $V$  for $g$ by \eqref{eq:V}  (the existence and boundedness of  the sum in \eqref{eq:V} is ensured by \ref{(A2)} and Theorem \ref{thm:decay2}), and say that $g$ is non-degenerate if $V >0$.

\begin{thm}[Central limit theorem]\label{thm:clt2}
Assume that the conditions in Theorem \ref{thm:decay2} together with $(\mathrm{LY})$ hold.
 Assume also that $g$ is non-degenerate.
 Then, for $\mathbb P$-almost every $\omega \in \Omega$, $\frac{(S_ng)_\omega }{\sqrt n}$ converges in distribution to a normal distribution with mean $0$ and variance $V$ as $n\to \infty$, i.e.
for any $a\in \mathbb R$, we have
\[
\lim _{n\to \infty} \mu _\omega \left(  (S_ng ) _\omega \leq a\sqrt n\right) = \frac{1}{\sqrt{2\pi V}}\int _{-\infty }^a e^{-\frac{z^2}{2V}}dz .
\]
\end{thm}

 \begin{thm}[Large deviation principle]\label{thm:ldp2}
Assume that the conditions in Theorem \ref{thm:clt2} hold.
Then,  one can find  a nonnegative, continuous, strictly convex function  $c: (-\delta _0,\delta _0)\to\mathbb R$ with some $\delta _0>0$ such that $c$ vanishes at $0$ and that for $\mathbb P$-almost every $\omega \in \Omega$ and any $\delta \in (0,\delta _0)$, 

\begin{equation*}
\lim _{n\to \infty} \frac{1}{n} \log \mu _\omega ( (S_ng)_\omega >n\delta ) = - c (\delta ).
\end{equation*}
\end{thm}

 \begin{thm}[Local central limit theorem]\label{thm:lclt2}
Assume that the conditions in Theorem \ref{thm:clt2} hold. 
Assume also that the condition $\mathrm{ (L)}$ (see Definition \ref{def:L}) holds. 
Then, for $\mathbb P$-almost every $\omega \in \Omega$ and every 
 bounded interval $J \subset \mathbb R$, we have
\[
\lim _{n\to \infty} \sup _{z\in \mathbb R}  \left\vert  \sqrt{ n V} \mu _\omega \left( z + (S_ng ) _\omega \in J \right) - \frac{1}{\sqrt{2\pi }} e^{-\frac{z^2}{2nV }} \vert J \vert \right\vert  =0 ,
\]
where $\vert J\vert $ denotes the length of $J$.
\end{thm}

The proofs of Theorems \ref{thm:clt2}, \ref{thm:ldp2} and \ref{thm:lclt2} will be given in Section 3.
An excellent survey paper for limit theorems (including central limit theorems, large deviation principles and local central limit theorems) for deterministic dynamical systems is \cite{Gouezel2015}. 
For limit theorems of random dynamical systems, see \cites{ANV2015,DFGV2018, DFGV2019, DH2020} and reference therein.
These results 
were established by the Nagaev-Guivarc'h perturbative spectral method explained in \S \ref{ss:se},  which indicates that  a variety of other limit theorems (such as almost sure invariance principles, moderate deviations principles 
and Berry-Esseen theorems) 
also may hold for random dynamical systems in our abstract setting.

\begin{rmk}
Soon after we uploaded the first version of this paper to arXiv, Dragi\v cevi\' c and Sedro \cite{DS2021} also extended limit theorems via the Nagaev--Guivarc'h spectral approach of \cites{DFGV2018, DFGV2019} to expanding on average maps, by adapting Pesin theory to transfer operator cocycles (on   the space of bounded variation functions).
Therefore, although our abstract setting in this subsection is not applicable to expanding on average maps in general, an appropriate change of the setting  of this section  and its proof according to Dragi\v cevi\' c--Sedro's approach may enable us to include expanding on average maps as a new application of our abstract limit theorems.

Furthermore, Dragi\v cevi\' c and Hafouta \cite{DH2021} developed an abstract framework tailored for almost surely invariant principle, and in their work estimates appear which are similar to some of the estimates in this paper. 
However, because our setting is different and slightly more abstract we cannot apply the estimates of \cite{DH2021} directly.
\end{rmk}

\subsection{The Nagaev-Guivarc'h perturbative spectral method}\label{ss:se}

Here we briefly recall the idea of the Nagaev-Guivarc'h perturbative spectral method and the difficulties that have to be overcome in order for it to be used for quenched schemes.
This will heuristically  explain  why both (UG) and (LY) are needed, and provide preparation for 
 definition of condition (L).

 The   \emph{twisted transfer operator}   $\mathcal L_{\theta ,\omega }
 : \mathscr D^\prime \to \mathscr D^\prime$ (of $T_\omega$ associated with $g_\omega$) with $\theta \in \mathbb C$ and $\omega \in \Omega$  is given by 
 \begin{equation}\label{eq:1102f}
 \mathcal L_{\theta , \omega } \varphi = \mathcal L_{\omega} (e^{\theta g_\omega} \varphi ) ,\quad \varphi \in\De' .
 \end{equation}
 (Recall \eqref{eq:1104a} for the multiplication of $e^{\theta g_\omega} $
   and $\varphi$.) 
 By \ref{(A1)} and \ref{(A2)}, 
$\mathcal L_{\theta ,\omega}$ is  bounded on $\mathscr B$ for $\mathbb P$-almost every $\omega \in \Omega$ and 
\begin{equation}
\label{eq:A2theta}
 \esssup _{\omega \in \Omega} \Vert \mathcal L_{\theta ,\omega } \varphi - \mathcal L_\omega \varphi \Vert  \leq C_0\vert  \theta \vert  G e^{\vert  \theta \vert  C_1G}  \Vert \varphi \Vert ,  \quad \varphi \in \mathscr B,
\end{equation}
where $G=\esssup _{\omega }\Vert g_\omega\Vert _{\mathscr E}$.
In a  manner similar  to the one in the definition of $(\mathcal L, \sigma)$, we define a cocycle $\mathbb N_0\times \Omega \times \mathscr B \ni (n, \omega , \varphi )\mapsto \mathcal L_{\theta , \omega } ^{(n)} \varphi$ over $\sigma$, and denote it by $(\mathcal L_\theta , \sigma)$.
Furthermore, for any $\theta\in \mathbb C$, $n\in \mathbb N$, $\varphi \in \mathscr D'$  and $\mathbb P$-almost every $\omega \in \Omega$, we have
\begin{equation}\label{eq:1104}
\mathcal L_{\theta, \omega }^{(n)}\varphi =\mathcal L_\omega^{(n)}(e^{\theta (S_ng)_\omega}\varphi ).
\end{equation}
Indeed, it is obviously true for $n=1$, 
 and  if \eqref{eq:1104} holds for a fixed $n\geq 1$, 
  then
\begin{align*}
\langle\mathcal L_{\theta, \omega }^{(n+1)}\varphi,u\rangle &=\langle\mathcal L_{\sigma ^{n}\omega} (e^{\theta g_{\sigma ^{n}\omega}} \mathcal L_{\theta, \omega }^{(n)}\varphi ),u\rangle \\
&=\langle\mathcal L_{\theta, \omega }^{(n)}\varphi,e^{\theta g_{\sigma ^{n}\omega}}u\circ T_{\sigma ^{n}\omega}\rangle\\
&=\langle\mathcal L_\omega^{(n)}(e^{\theta (S_{n}g)_\omega}\varphi ),e^{\theta g_{\sigma ^{n}\omega}}u\circ T_{\sigma ^{n}\omega}\rangle\\
&=\langle e^{\theta (S_{n}g)_\omega}  \varphi , e^{\theta g_{\sigma ^{n}\omega} \circ T_\omega ^{(n)}}u\circ T_{\sigma ^{n}\omega}\circ T_\omega ^{(n)}\rangle\\
&=\langle   e^{\theta (S_{n+1}g)_\omega}\varphi , u\circ T_\omega ^{(n+1)}\rangle
=\langle   \mathcal L^{(n+1)}(e^{\theta (S_{n+1}g)_\omega}\varphi ), u \rangle
\end{align*}
 for each $u\in\mathscr D$, and we   get the claim.

For 
 a random process $\{ u_n\} _{n\geq 1} $ on a probability space $(M, \mathcal G, \nu )$, the moment generating function $\R \ni t \mapsto \mathbb E_{\nu } [e^{t u_n}]$    and the characteristic function $\R \ni t \mapsto  \mathbb E_{\nu } [e^{it u_n }]$ of $u_n$'s (around  $t=0$) play a fundamental role in the study of the asymptotic behavior of  $\{ u_n\} _{n\geq 1}$ (see e.g.~\cite{Durrett2019}).
Hence, to understand the asymptotic behavior of $(S_ng)_\omega$ with respect to $\mu _\omega$,  it is of great importance to investigate
\[
\mathbb E_{\mu _\omega } [e^{\theta (S_ng)_\omega }] = \langle h_{\omega } , e^{\theta (S_ng)_\omega } \rangle 
= \langle \mathcal L^{(n)}_{\theta  , \omega} h_{\omega } , 1_M \rangle 
\]
around $\theta =0$. (The first identity follows from (UG) and the second from \eqref{eq:1104}.) 
Therefore, it is natural to expect that several limit theorems should follow from $(\mathcal L_{\theta} , \sigma)$ having
nice spectral properties   around $\theta =0$. 
 This idea is called  
  \emph{Nagaev-Guivarc'h perturbative spectral method},
   and has been broadly applied to show limit theorems for abundant 
  deterministic dynamical systems (see \cites{HH2000, Gouezel2015, ANV2015, DFGV2018} and the references therein; this method was  originally applied to Markov chains \cite{HH2000}).

The difficulty in applying a (Nagaev-Guivarc'h  type) spectral method for quenched schemes is that one should consider a spectral analysis for operator \emph{cocycles}, and the usual notion of spectrum for a single operator is not useful. 
This  was overcome in  \cite{DFGV2018} by using the theory of Lyapunov spectrum. The notion of Lyapunov spectrum was also used in \cite{NW2015} to study decay of correlation functions for quenched schemes.

Let $\mathscr H$ be a Banach space and let $\mathcal A$ be a measurable mapping from $\Omega $ to the set of bounded operators on $\mathscr H$ (endowed with the operator norm) such that $\omega \mapsto \log ^+\Vert \mathcal A_\omega \Vert$ is $\mathbb P$-integrable, and consider the cocycle $(n,\omega , \varphi )\mapsto \mathcal A_\omega ^{(n)}\varphi = \mathcal A_{\sigma ^{n-1}\omega } \circ \cdots \circ \mathcal A_{\omega } $ over $\sigma$ (denoted  by $(\mathcal A,\sigma)$). 
Then, it follows from 
Kingman's sub-additive ergodic theorem  that the following limits (called the \emph{maximal Lyapunov exponent} and the \emph{index of compactness} of $\mathcal A$) exist 
 and are $\mathbb P$-almost surely independent of $\omega \in \Omega$:
\begin{equation}\label{eq:Lyapunovexponents}
 \Lambda (\mathcal A)= \lim _{n\to \infty} \frac{1}{n} \log \Vert \mathcal A^{(n)}_{\omega } \Vert ,
\quad  \kappa (\mathcal A)  = \lim _{n\to \infty} \frac{1}{n} \log \mathrm{ic}( \mathcal A^{(n)}_{ \omega } ) ,
\end{equation}
where
$\mathrm{ic}(\mathcal A_0)$ is the index of compactness  of a bounded operator $\mathcal A_0: \mathscr H\to \mathscr H$ (i.e.~$\mathrm{ic}(\mathcal A_0)$ is the infimum of $r>0$ such that the image of the unit ball of $\mathscr H$ by $\mathcal A_0$
can be covered with finitely many balls of radius  $r$).
The cocycle $(\mathcal  A , \sigma)$ is called \emph{quasi-compact} if $\Lambda (\mathcal A) > \kappa  (\mathcal A)$.
The notion of quasi-compactness for cocycles of linear operators was
introduced by Thieullen \cite{Thieullen1987}, and was key for establishing
the multiplicative ergodic theorem 
on Banach spaces (see \cites{FLQ2013,GQ2014} and reference therein).
We shall use the fact that quasi-compact cocycles admit an Oseledets decomposition in the following sense.

\begin{dfn}\label{dfn:OD}
We say that there is an {\it Oseledets decomposition} of $(\mathcal A , \sigma)$ on  $\mathscr H$ if there is a  real number  $ \widetilde{\Lambda} (\mathcal A)
 < \Lambda (\mathcal A)$, 
   a  measurable splitting of $\mathscr H$ into closed subspaces
$
\mathscr H=\mathscr F _{\omega}\oplus \mathscr R_{\omega}
$ with $
\dim{\mathscr F _{\omega}}<\infty$,  and a $\sigma$-invariant subset $\tilde{\Omega }\subset \Omega$ 
of full measure
such that   for each
$\omega \in \tilde{\Omega}$, the following holds:
\begin{enumerate}
\item[$\mathrm{(1)}$]  The splitting is semi-invariant:
\[
\mathcal A_\omega \mathscr F _{\omega } =\mathscr F_{\sigma \omega }\quad \text{and} \quad \mathcal A_\omega  \mathscr R _{ \omega } \subset \mathscr R _{ \sigma \omega }.
\]
\item[$\mathrm{(2)}$] If $\varphi \in \mathscr  F_{ \omega }\backslash \{ 0\}$, then
\[
\lim _{n\rightarrow \infty} \frac{1}{n} \log \lVert \mathcal A_{\omega}^{(n)} \varphi \rVert  =\Lambda (\mathcal A) .
\]
\item[$\mathrm{(3)}$] If $\varphi \in \mathscr R _{ \omega}$, then
\[
\limsup _{n\rightarrow \infty} \frac{1}{n} \log \lVert \mathcal A_\omega ^{(n)} \varphi \rVert \leq \widetilde{ \Lambda } (\mathcal A).
\]
\end{enumerate}
\end{dfn}

  It can be easily shown that (UG) implies that $(\mathcal L,\sigma)$ admits an Oseledets decomposition  $\mathscr B=\mathscr F_\omega\oplus \mathscr R_\omega$ with $\mathscr F_\omega =\{ \pi _\omega (f) \mid f\in \mathscr B\}$, $\mathscr R_\omega =\{ f-\pi _\omega (f) \mid f\in \mathscr B\}$, $\Lambda (\mathcal L)=0$ and $\widetilde \Lambda (\mathcal L)= \log \rho$, where  $\pi _\omega (f) = \langle h_\omega , f\rangle h_\omega$ and $\rho$ is the decay rate  in \eqref{eq:0212c2}.
 On the other hand, 
$\mathbb E_{\mathbb P}[\log^+\lVert\mathcal L\rVert]<\infty$ by (A1), where $\log ^+ \Vert \mathcal L\Vert (\omega )= \max \left\{\log  \Vert \mathcal L_\omega \Vert , 0\right\}$ for $\omega\in \Omega$.
Hence, using the multiplicative ergodic theorems in \cites{FLQ2013,GQ2014}, it can also be shown that (LY) implies that $(\mathcal L,\sigma)$ admits an Oseledets decomposition on $\mathscr B$ (\cite{FLQ2013} for the case (i) of (LY) and \cite{GQ2014} for the case (ii)). 
We consider the following condition (recall that $\mathcal L_\theta$ depends on $g$). 
\begin{dfn}[Oseledets decomposition condition]
We say that $(T, g)$ 
 satisfies the {\it Oseledets decomposition condition} (OD) if $\mathcal L_\theta$ admits an Oseledets decomposition on $\mathscr B$ for any $\theta\in\mathbb C$ with sufficiently small absolute value.
\end{dfn}

We note that \emph{the proofs of Theorems \ref{thm:clt2}, \ref{thm:ldp2}, \ref{thm:lclt2}  work   with} (OD) \emph{instead of} (LY).
However, it is unclear to us whether the  Oseledets decomposition of $(\mathcal L, \sigma )$ is stable under the twisting perturbation,
 i.e.~whether the Oseledets decomposition of $(\mathcal L,\sigma)$ implies (OD). 
On the other hand, the  stability of Lasota-Yorke inequality is easily obtained as in the following proposition, 
 which is the reason why (not only (UG) but also) (LY)  is assumed. 
 In particular, if $T$ satisfies (LY) then $(T, g)$ satisfies (OD) for \emph{arbitrary $g\in L^\infty (\Omega , \mathscr E)$}.

\begin{prop}\label{prop:1114a}
Assume that $T$ satisfies $(\mathrm{LY})$. Then the Lasota-Yorke inequality \eqref{eq:forProp35} holds with $\mathcal L_{\theta}$ instead of $\mathcal L$ for any $\theta\in\mathbb C$ with sufficiently small absolute value. 
\end{prop}

\begin{proof}[Proof]
Note that for any $n\geq 1$ and $\omega \in \Omega$,
\[
\mathcal L_{\theta , \omega} ^{(n)} - \mathcal L_\omega ^{(n)} = \sum _{j=0}^{n-1} \mathcal L_{\theta , \sigma ^{n-j}\omega} ^{(j)} (\mathcal L_{\theta , \sigma ^{n-j-1}\omega}  - \mathcal L_{\sigma ^{n-j-1} \omega} ) \mathcal L_{ \omega} ^{(n-j-1)} .
\]
Hence, it follows from \ref{(A1)} and \eqref{eq:A2theta} that there is a constant $C\equiv C(g, n)>0$ such that
\[
\Vert \mathcal L_{\theta , \omega} ^{(n)} \varphi- \mathcal L_\omega ^{(n)} \varphi \Vert \leq C\vert \theta \vert e^{C\vert \theta \vert} \Vert \varphi \Vert , \quad \varphi \in \mathscr B.
\]
This immediately leads to the implication that $\mathcal L_\theta$ satisfies \eqref{eq:forProp35}  (with $\alpha$ replaced by $\alpha  +C\vert \theta \vert e^{C\vert \theta \vert} $) for any $\theta$ satisfying $E_{\mathbb P}[\alpha ] +C\vert \theta \vert e^{C\vert \theta \vert} <1$.

We next show that in the case when $(\Omega , \mathcal F, \mathbb P)$ is a Polish space,  the  $(\mathcal F,\mathcal B_{L(\mathscr B)})$-measurability of $\mathcal L$ implies that of $\mathcal L_\theta$. 
Recall that, since $(\Omega , \mathcal F, \mathbb P)$ is a Polish space,  for any topological space $\mathscr V$ and any map $\mathcal A: \Omega \to \mathscr V$, the  $(\mathcal F,\mathcal B_{\mathscr V})$-measurability of $\mathcal A$ is equivalent to the $\mathbb P$-continuity of $\mathcal A$, see \cite[Remark 13]{FLQ2013}.
Assume that $\mathcal L: \Omega \to L(\mathscr B)$  is $\mathbb P$-continuous. 
Then, since $g: \Omega \to \mathscr E$ is also measurable,  one can find a countable collection of Borel sets $\{ A_n\}_{n\in \mathbb N}$ such that $\bigcup _{n\in \mathbb N} A_n$ has $\mathbb P$-full measure and both the restriction of $\mathcal L$ and $g$ on $A_n$   are continuous for each $n\in \mathbb N$. 
Notice that for each $n\in \mathbb N$, the restriction of $ \mathcal L_{\theta  }  $ on $A_n$  is the composition of a map  $A_n \times \mathscr E \to L(\mathscr B) : (\omega , g_0 ) \mapsto   \mathcal L_\omega (e^{\theta g_0} \cdot )$   and a continuous map  $A_n   \to A_n \times \mathscr E : \omega \mapsto (\omega , g_{\omega  })$. 
On the other hand, it is easy to see that for each $\omega \in A_n$ the map  $ g_0 \mapsto \mathcal L_\omega (e^{\theta g_0} \cdot )$  is a continuous map  from $\mathscr E$ to $L(\mathscr B)$ because
\[
\Vert  \mathcal L_\omega (e^{\theta g_0} \varphi ) -  \mathcal L_\omega (e^{\theta g_0'} \varphi )\Vert \leq C_0 \Vert e^{\theta g_0} \Vert _{\mathscr E} \vert \theta \vert \Vert g_0- g_0'\Vert _{\mathscr E} \times e^{\vert \theta \vert  C_1 \Vert g_0- g_0'\Vert _{\mathscr E}} 
\]
for each $\varphi \in \mathscr B$ with $\Vert  \varphi  \Vert =1$ and $g_0, g_0' \in \mathscr E$ (recall the calculation to obtain \eqref{eq:A2theta}). 
Moreover, for each $g_0\in \mathscr E$ the map  $ \omega \mapsto \mathcal L_\omega (e^{\theta g_0} \cdot )$ is a continuous map  from $A_n$ to $L(\mathscr B)$ due to the continuity  of the restriction of $\mathcal L$ on $A_n$ and (A2). 
In conclusion, $\mathcal L_\theta$ is continuous on each $A_n$, which implies the $(\mathcal F,\mathcal B_{L(\mathscr B)})$-measurability of $\mathcal L_\theta $.

Finally, we show that in the case when $(\Omega , \mathcal F, \mathbb P)$ is a Lebesgue-Rokhlin probability space and $\mathscr B$ is separable, 
the  $(\mathcal F,\mathcal S_{L(\mathscr B)})$-measurability of $\mathcal L$ implies that of $\mathcal L_\theta$. 
Recall that the  $(\mathcal F,\mathcal S_{L(\mathscr B)})$-measurability of a map $\mathcal A : \Omega \to L(\mathscr B)$ is equivalent to the strong measurability of $\mathcal A$ (i.e.~for any $\varphi \in \mathscr B$, the map $\omega \mapsto \mathcal A(\omega )\varphi$ is $(\mathcal F,\mathcal B_{\mathscr B})$-measurable), 
see \cite[Appendix A]{GQ2014}.
Assume the strong measurability of $\mathcal L$, and fix $\varphi \in \mathscr B$.
As in the previous paragraph, $\Omega \to \mathscr B: \omega \mapsto \mathcal L_{\theta ,\omega} \varphi$ is the composition of  $  \Omega \times \mathscr E \to \mathscr B : (\omega , g_0 ) \mapsto \mathcal L_\omega (e^{\theta g_0} \varphi )$ and  a measurable map  
$ \Omega \to \Omega \times \mathscr E : \omega  \mapsto (\omega , g_{\omega  })$. 
Furthermore, for each $\omega \in \Omega$ the map  $\mathscr E\ni g_0 \mapsto  \mathcal L_\omega (e^{\theta g_0} \varphi )$ is continuous, and for each $g_0\in \mathscr E$ the map $\Omega \ni \omega \mapsto  \mathcal L_\omega (e^{\theta g_0} \varphi )$ is measurable due to the strong measurability of $\mathcal L$.  
Hence, by \cite[Lemma 3.14]{CV77}, $(\omega , g_0 ) \mapsto \mathcal L_\omega (e^{\theta g_0} \varphi )$   is measurable, and thus $\mathcal L_\theta$ is 
$(\mathcal F,\mathcal S_{L(\mathscr B)})$-measurable.
\end{proof}

\subsection{Local central limit theorem}
After the preparation in Subsection \ref{ss:se}, we can now give sufficient  conditions under which the local central limit theorem  (LCLT)   in Theorem \ref{thm:lclt2} holds.

For recent progress on LCLT for dynamical systems (with a neutral direction), see \cites{DN2020,AT2020} and reference therein. 
As usual in the study of LCLT, we start this subsection from considering a dichotomy between  periodic (lattice, arithmetic) and aperiodic (non-lattice, non-arithmetic) cases.
We define $\mathcal U_g$ by
\[
\mathcal U_g=\{ t\in \mathbb R \mid \Lambda (\mathcal L_{it} ) =0\} 
\]
(recall that $\mathcal L_{it}$ depends on $g$).
In the case when $T_\omega$ is $\mathbb P$-almost surely a mixing piecewise expanding map and $\mathscr B$ is the space of  BV functions  (i.e.~real-valued functions whose total variation is bounded),   Dragi\v{c}evi\'{c} et al.~showed in \cite[Subsection 4.4]{DFGV2018} that $\mathcal U_g$ is a subgroup of $(\mathbb R, +)$. 
Furthermore, it follows from 
the remark in the proof of Lemma 4.13 of \cite{DFGV2018} 
that, by assuming that $g$ is non-degenerate (i.e.~$V  >0$),  $\mathcal U_g$ is one of the following two cases:
\begin{itemize}
\item $\mathcal U_g=a\mathbb Z$ for some $a>0$ (periodic);
\item $\mathcal U_g=\{ 0\} $ (aperiodic).
\end{itemize}
This dichotomy is known to hold for 
  wide classes of Markov processes or deterministic dynamical systems \cites{Durrett2019, HH2000, Iwata2008}. 
The conclusion in Theorem \ref{thm:lclt2} 
is a standard form of the LCLT in the aperiodic case, 
while the LCLT in the periodic case always needs some modification in the term $\vert J \vert$ (see \cites{Durrett2019, HK, DFGV2018} for a precise description of the LCLT in the periodic case).

Furthermore, Dragi\v{c}evi\'{c} et al.~also showed that (for their random dynamics and functional space)
the aperiodicity of $g$ is equivalent to 
the following useful spectral condition, see \cite[Subsection 4.3.2]{DFGV2018}. 
To express our respect to \cite{DFGV2019}, we keep using their terminology ``condition (L)'' for the condition. 
\begin{dfn} 
\label{def:L}
We say that $T$ satisfies \emph{condition} (L) for $g$ if 
 for every bounded  interval $J \subset  \mathbb R\setminus \{ 0 \}$, 
  there exist a real number $\kappa \in (0, 1)$ 
 and a random variable $C : \Omega \to (0, \infty )$ such that for $\mathbb P$-almost every $\omega \in \Omega$, 
\[
\Vert \mathcal L_{it, \omega }^{(n)} \Vert \leq C_\omega \kappa ^n 
\quad \text{for $t \in J$ and $n \geq 0$}.
\]
\end{dfn}
The deterministic version of condition (L) has been also considered and shown to be equivalent to the deterministic version of the aperiodicity condition, see e.g.~\cite{HH2000}.

Although condition (L) is useful for a functional analytic proof of the LCLT (and indeed we prove Theorem \ref{thm:lclt2} by assuming condition (L)), as pointed out in \cite{DFGV2019} it may not be easy to verify this condition directly.
Therefore, as in 
\cite{DFGV2019}, 
we employ a more tractable sufficient condition for condition (L) by Hafouta and Kifer \cite{HK}. 
\begin{dfn}\label{def:HK}
We say that $T$ satisfies the \emph{Hafouta-Kifer condition} (HK) for $g$ if the following holds:
\begin{itemize}
\item[(HK1)] $\mathcal F$ is the Borel $\sigma$-algebra on a metric space $\Omega$, $\mathbb P(U)>0$ for any open set $U$, and $\sigma$ is a homeomorphism with a periodic point $\omega _0$ with period  $\ell_0$  (i.e.~$\sigma ^{\ell_0}\omega _0 = \omega _0$).
Moreover for each $j\in  \{ 0, 1, \ldots , \ell_0 -1\}$, there exists a neighborhood of $\sigma ^j\omega _0$ on which the map $\omega \mapsto T_\omega$ is constant.
\item[(HK2)]  For any compact interval  $J \subset \mathbb R$,
the family of maps $\{ \omega \mapsto \mathcal L_{it , \omega } \} _{t\in J}$  is equicontinuous
at any point in the orbit of $\omega _0$. 
Furthermore, 
there exists a constant $B \geq 1$
such that for $\mathbb P$-almost every $\omega \in \Omega$,
\[
\Vert \mathcal L_{it, \omega }^{(n)} \Vert \leq B
\]
for any $n \in \mathbb N$ and $t \in J$.
\item[(HK3)] For any compact interval $J \subset \mathbb R\setminus \{ 0\}$,
there exist constants $ c > 0$ and $ b\in  (0, 1)$ such that
\[
\left\Vert \left ( \mathcal L_{it, \omega _0} ^{(\ell_0)} \right)^n \right\Vert \leq cb^n
\]
for any $n\in \mathbb N $ and $t \in J$.
\end{itemize}
\end{dfn}
In \cite[Lemma 2.10.4]{HK}, Hafouta and Kifer  showed (in a more abstract form) that condition (HK) implies condition (L).
 Note that (HK1) is rather a restriction on 
noise, than a hypothesis of dynamics to be verified. 
 See \cite[Section 9]{DFGV2019} where it is explained that the restriction of (HK1) is mild. 
Furthermore,  the first part of (HK2)  follows from a condition on the observable $g$ as follows.

\begin{prop}\label{thm:1103}
Assume that $\mathrm{(HK1)}$ holds and $g: \Omega \to \mathscr E$ 
 is 
 continuous at any point in the set $\{ \sigma ^j \omega_0 \mid 0\leq j\leq \ell_0 -1 \}$. Assume that $g$ satisfies the centering condition \eqref{eq:0419b}.
 Then  the first part of $\mathrm{(HK2)}$ holds.
\end{prop}

\begin{proof}
We first note that due to $\mathrm{(HK1)}$, 
 $\mathcal L_\omega = \mathcal L_{\omega _0}$ for all $\omega $  in a neighborhood of $\omega _0$. 
Thus, 
for each $\varphi \in \mathscr B$, it follows from (A1) and (A2) 
 that
\begin{equation}\label{eq:0422a}
\Vert  (\mathcal L_{it, \omega } - \mathcal L_{it , \omega _0} ) \varphi \Vert   \leq C_0 \Vert  e^{it g_ \omega } - e^{it g_{\omega _0}}  \Vert _{ \mathscr E} \Vert \varphi \Vert   .
\end{equation}
Since it holds 
\begin{equation}\label{eq:1104c}
e^{it g_ \omega } - e^{it g_{\omega _0}} = \sum _{k=0}^\infty \frac{(it)^k (g_\omega ^k - g_{\omega _0}^k)}{k!}
 = (g_\omega  - g_{\omega _0})\sum _{k=1}^\infty \frac{(it)^k \sum _{\ell =0}^{k-1} g_\omega ^\ell  g_{\omega _0}^{k-\ell }}{k!} ,
\end{equation}
by using  (A2) again, we get
\[
\Vert  e^{it g_ \omega } - e^{it g_{ \omega _0}}  \Vert _{\mathscr E} \leq \Vert  g_ \omega  -  g_{\omega _0}  \Vert _{ \mathscr E} \vert t\vert e^{\vert t\vert G}
\]
(recall that  $G=\esssup _\omega \Vert g_\omega \Vert _{ \mathscr E }$).
The first part of (HK2) immediately follows from this inequality and \eqref{eq:0422a}.
\end{proof}

 Notice that, as (HK1), the condition in 
  Proposition \ref{thm:1103} is rather a restriction on observables  than a condition to be verified. 
 See \cite[Remark 9.2]{DFGV2019} for the argument that the condition  is preserved under the centering  by virtue of (HK1), so  the centering condition on $g$  is not essential.

\subsection{Application to random U(1) extensions of expanding maps}

The following result shows that Theorems A--D follow from Theorems \ref{thm:decay2}--\ref{thm:lclt2}:

\begin{thm}\label{thm:v}
Let $T$ be the random U(1) extension of an expanding map on the circle with a noise level $\epsilon>0$ described in Subsection \ref{subsection:mr}. 
Assume that   $(\Omega , \mathcal F, \mathbb P)$ is a Lebesgue-Rokhlin probability space  and $T_0$ satisfies the partial captivity condition.
Then for any sufficiently small  $\epsilon>0$ and any sufficiently large integer $m$,
 the usual Sobolev space $H^m(\T^2)$ is admissible for $T\equiv T(\epsilon )$, and $H^r(\T^2)$ is an associated space of observables for all $r\ge m$. 
 Moreover, $T$ satisfies the uniform spectral gap condition $\mathrm{(UG)}$ with respect to $H^m(\T^2)$ and $H^r(\T^2)$ and
 the Lasota-Yorke inequality condition $\mathrm{(LY)}$   with respect to $H^m(\T^2)\subset H^{m-1}(\T^2)$.
 In particular, $T$ satisfies the conditions of Theorem \ref{thm:decay2} with $\mathscr B=H^m(\T^2)$ and $\mathscr E=H^r(\T^2)$ for such $\epsilon,m,r$.
\end{thm} 

The Sobolev spaces $H^m(\T^2)$ is the usual one -- a definition is given in Subsection \ref{subsection:preliminary} below.
We mention that the above regularity index $m $ only depends on the unperturbed map $T_0$, see Remark \ref{rmk:rdependsonT0}.

\begin{thm}\label{thm:v2}
Let $T$ be  the random U(1) extension of an expanding map on the circle 
given   in Theorem \ref{thm:v}. Assume that $g\in L^\infty (\Omega , \mathscr C^m(\mathbb T^2))$ satisfies the centering condition \eqref{eq:0419b}.
\begin{itemize}
\item[$\mathrm{(i)}$]
Assume  that 
  $g_\omega$ does not depend on the second variable 
for $\mathbb P$-almost every $\omega$.  \color{black}
Then 
the second part of $\mathrm{(HK2)}$ holds.
\item[$\mathrm{(ii)}$] Assume also  that 
the set $\{ \sigma ^j \omega_0 \mid 0\leq j\leq \ell_0 -1 \}$ is included  in the full measure set of  \eqref{convinC1} and that
$(S_{\ell_0}g) _{\omega _0}$ is not cohomologous to a constant with respect to $T_{\omega _0}^{(\ell_0)}$, that is, for 
 any 
 $c\in \mathbb R$ and  integrable function $\psi :\mathbb T^2\to  \mathbb R$, 
  it does not hold that
\[
(S_{\ell_0}g)_{\omega _0} = \psi \circ T_{\omega _0}^{(\ell_0)} -\psi  +c ,
\]
where $\ell _0$ is a positive integer and $\omega _0$ is a periodic point   with period $\ell _0$. \color{black} 
 Then  $\mathrm{(HK3)}$ holds for these $\ell _0$ and $\omega _0$.
 \end{itemize}
\end{thm}

\begin{rmk}
 A class of  examples satisfying the   condition   for $g$   in the item (ii) of Theorem \ref{thm:v2} is $\ell_0=1$ and $g_{\omega _0} (x,s) = \widetilde g(x) - \mathbb E_{\mu _{\omega _0}}[\widetilde g]$ with  some $\widetilde g \in \mathscr V_\varrho$,  $1<\varrho <k$, where $\mathscr V_\varrho$ is    the open and dense subset  of $\mathscr C^r (\mathbb S^1)$ given in Theorem \ref{conj:maingoal} in Appendix \ref{app:B}. 
Indeed, if
$
(S_{m_0}g)_{\omega _0} = \psi \circ T_{\omega _0} -\psi  +c 
$
with some $c\in \mathbb R$ and  $\psi \in L^1(\mathbb T^2)$
then $\widetilde g= \widetilde \psi \circ E_{\omega _0} -\widetilde  \psi +c+\mathbb E_{\mu _{\omega _0}}[\widetilde g]$ with $\widetilde  \psi  = \int \psi (\cdot , s ) ds$ due to the translation invariance of $\mathrm{Leb} _{\mathbb S^1}$. 
This contradicts the fact that no $\widetilde g\in \mathscr V_\varrho$  is cohomologous to a constant (see  Appendix \ref{app:B}).
Therefore, this observable 
satisfies the statement in item (ii) of Theorem \ref{thm:v2}. 
 \end{rmk}

As an immediate consequence, $T$ satisfies condition $\mathrm{(L)}$ for $g$ as long as the following, more easily verifiable, conditions are met:

\begin{cor}
Let $T$ be  the random U(1) extension of an expanding map on the circle 
given   in Theorem \ref{thm:v}. Assume $\mathrm{(HK1)}$. 
\begin{itemize}
\item[$\mathrm{(i)}$] Assume that $g\in L^\infty (\Omega , \mathscr C^m(\mathbb T^2))$ satisfies the centering condition \eqref{eq:0419b}.
\item[$\mathrm{(ii)}$] Assume that $g: \Omega \to \mathscr E$ 
 is 
 continuous at any point in the set $\{ \sigma ^j \omega_0 \mid 0\leq j\leq \ell_0 -1 \}$. 
\item[$\mathrm{(iii)}$]
Assume  that 
  $g_\omega$ does not depend on the second variable 
for $\mathbb P$-almost every $\omega$.  
\item[$\mathrm{(iv)}$] Assume that 
the set $\{ \sigma ^j \omega_0 \mid 0\leq j\leq \ell_0 -1 \}$ is included  in the full measure set of  \eqref{convinC1} and that
$(S_{\ell_0}g) _{\omega _0}$ is not cohomologous to a constant with respect to $T_{\omega _0}^{(\ell_0)}$.
 \end{itemize}
 Then $T$ satisfies condition $\mathrm{(L)}$ for $g$.
\end{cor}

\section{Hypothesis verification}\label{section:proof}

In this section we will prove  Theorem \ref{thm:v}, and thus show that the random U(1) extension of an expanding map in Section \ref{section:introduction} satisfies the conditions appearing in the statements of Theorems \ref{thm:decay2}--\ref{thm:lclt2}. We also prove Theorem \ref{thm:v2}.
Let therefore  $E$, $\tau$, $T$   be as in \eqref{eq:perturbedsystem}.  
In particular, $\tau$ is a random perturbation of a function $\tau _0$, where $\tau_0$ is guaranteed not to be cohomologous to a constant function by the assumption of partial captivity.
Throughout this subsection we assume that $(\Omega,\mathcal F,\mathbb P)$ is a Lebesgue-Rokhlin probability space.

\subsection*{Notation}
In Subsection \ref{ss:se}, we used the double-indexed operator $\mathcal L_{\theta , \omega}$ with $\theta \in \mathbb C$, $\omega \in \Omega$. 
However, this notation will not appear in this subsection, so another double-indexed operator $\mathcal L_{\nu , \omega}$ with $\nu \in \mathbb Z$, $\omega \in \Omega$ given in \eqref{eq:1115} should cause no confusion.
In particular, $\mathcal L_{0 , \omega}$ in this section always means $\mathcal L_{\nu , \omega}$ at $\nu =0$. 
Furthermore, we will use $C$ and $C_n$ ($n\geq 1$) as  positive constants which do not depend on $\omega \in \Omega$, $\nu \in \mathbb Z$ nor $\epsilon \geq 0$, and may change on occasion.

\subsection{Preliminaries}\label{subsection:preliminary}
We first note that $T_\omega $  preserves $L^2 (\mathbb T^2)  \equiv L^2 (\mathbb T^2, \mathrm{Leb} _{\mathbb T^2})$  $\mathbb P$-almost surely, so that we can define the transfer operator $\mathcal L_{\omega}$ of $T_\omega$ on $L^2 (\mathbb T^2) \cong (L^2 (\mathbb T^2) )^\prime$ by \eqref{eq:2}. 
As in \cite{NW2015} (originally in \cite{Faure}), we employ the following decomposition in Fourier modes,
\begin{equation}\label{orthogonaldecomp}
L^2(\T^2)=\bigoplus _{\nu \in 2\pi\mathbb{Z}} \mathcal{H} _{\nu} ,
\quad \mathcal{H} _{\nu}=\{ (x,s)\mapsto \psi(x)e^{i\nu s} : \psi\in L^2(\Sone) \}.
\end{equation}
The spaces $(\mathcal H_\nu,\lVert\phantom{i}\rVert_{L^2(\T^2)})$ and $L^2(\Sone)$ are isometrically isomorphic, and the decomposition \eqref{orthogonaldecomp} is preserved by the pullback operator $u\mapsto u \circ T _\omega$.

For given $\nu\in 2\pi \Z$ and fixed $\omega\in\Omega$, let $\mathcal M_{\nu , \omega }^{(n)}$ denote
the operator $u\mapsto u \circ T^{(n)} _\omega$ restricted to $\mathcal H_\nu$. It is straightforward to check that by identifying $\mathcal H_\nu$
with $L^2(\Sone)$ we can view $\mathcal M_{\nu , \omega }^{(n)}$ for $\mathbb{P}$-almost every $\omega$ as an operator on $L^2(\Sone)$ given by
\begin{equation}\label{restrictionoperator}
\mathcal M_{\nu , \omega }^{(n)} \psi(x)= \psi\big(E^{(n)} _ \omega (x)\big)e^{i\nu\tau^{(n)} _\omega (x)}.
\quad \psi \in L^2(\Sone),
\end{equation}
Here  
$E_\omega ^{(n)} = E_{\sigma ^{n-1}\omega } \circ E_{\sigma ^{n-2}\omega } \circ \cdots \circ E_\omega$
and 
$\tau^{(n)} _\omega$ is the quenched Birkhoff sum of $\tau$ given by
$$
\tau^{(n)} _\omega=\sum_{j=0}^{n-1}  \tau _{\sigma^j\omega}  \circ E_\omega^{(j)}\quad (n\ge1,\ \omega\in\Omega).
$$

Let $\mathcal L_{\nu , \omega }^{(n)} :L^2(\Sone) \to L^2(\Sone)$ be the adjoint operator of $\mathcal M_{\nu , \omega }^{(n)}$. 
It is straightforward to check that 
\begin{equation}\label{eq:1115}
\mathcal L_{\nu , \omega }^{(n)} \psi (x)=\sum_{E^{(n)} _ \omega (y)=x}\frac{e^{-i\nu\tau ^{(n)} _\omega (y)}}{\frac{d}{dy}E^{(n)} _\omega(y)}\psi(y),
\quad\psi\in L^2(\Sone).
\end{equation}
Note also that for $\varphi\in L^2(\T^2)$ we have 
\[
\mathcal L_{\omega } \varphi (x, s)= \sum _{\nu \in 2\pi \Z} e^{is\nu} \mathcal L  _{\nu ,\omega } \varphi _\nu (x) ,\quad\varphi(x,s)=\sum_{\nu\in 2\pi\Z} \varphi_\nu(x)e^{i\nu s},
\]
where $\varphi_\nu(x)=\int_{\Sone} e^{-is\nu}\varphi(x,s)\,ds$ and thus
\begin{equation}\label{eq:0811c}
( \mathcal L_{\omega }^{(n)} \varphi ) _\nu (x)  =\int _{\mathbb S^1} e^{-is\nu } \mathcal L_{\omega }^{(n)}  \varphi (x,s) ds = \mathcal L  _{\nu ,\omega }^{(n)} \varphi _\nu (x) ,\quad n\ge1 .
\end{equation}

For $m\geq 0$ we let $H^m(\Sone)\subset L^2(\Sone )$ denote
the Sobolev space of regularity index $m$, defined as the set of measurable functions $\psi $ satisfying
\[
\lVert \psi \rVert_{H^m(\Sone)}^2 
=\sum_{\xi\in2\pi\Z}(1+\lvert\xi\rvert^2)^m\lvert\hat \psi (\xi)\rvert^2<\infty.
\]
The Fourier coefficients of  $\psi \in L^2(\Sone)$ are defined by 
\[
\hat{\psi }(\xi)=\langle \psi ,\phi_{-\xi}\rangle,\quad \xi\in2\pi\Z,
\]
where the functions $\phi_\xi\in {\Ci}^\infty(\Sone)$ are given by $\phi_\xi(x)= e^{ix\xi}$
for $\xi\in2\pi\Z$ and $x\in \Sone$.

We also introduce an alternative $\nu$-dependent norm $\Vert\cdot \Vert_{H_\nu^m}$
on $H^m(\Sone)$ defined
by
\begin{equation}\label{eq:nu_dependentnorm}
\lVert \psi \rVert_{H_\nu^m}^2=\sum_{\xi\in2\pi\Z} (1 + ( \xi/\nu) ^2)^m \lvert \hat \psi (\xi)\rvert^2,\quad \psi \in L^2(\Sone),\quad 0\ne\nu\in\Z.
\end{equation}
Let $H_\nu^m(\Sone)$ denote the space $H^m(\Sone)$ equipped with the norm $\lVert\phantom{i}\rVert_{H_\nu^m}$.

Finally, we let $H^{r}(\mathbb T^2)$ denote the Sobolev space on $\mathbb T^2$ with regularity $r$, defined similarly as the set of measurable functions $\varphi $ satisfying
\[
\lVert \varphi \rVert_{H^r(\T^2)}^2 
=\sum_{\xi,\nu\in2\pi\Z}(1+\lvert\xi\rvert^2+\lvert\nu\rvert^2)^r\lvert\hat \varphi (\xi,\nu)\rvert^2<\infty,
\]
where 
\[
\hat{\varphi }(\xi , \nu ) = \int _{\mathbb T^2} e^{-i(x\xi+s\nu )}\varphi (x,s) dxds
,\quad \xi , \nu \in2\pi\Z .
\]
We shall work with the space $\mathscr H^m(\T ^2)\subset L^2(\T ^2 )$ defined as the set of measurable functions $\varphi $ satisfying
\begin{equation}\label{eq:Bnorm}
\lVert \varphi \rVert_{\mathscr H^m}^2=\lVert \varphi _0 \rVert_{H^m (\Sone )}^2 + \sum _{\nu \in 2\pi \Z , \; \nu \neq 0}(1+\nu^2)^m\lVert \varphi _\nu \rVert_{H^m_\nu (\Sone )}^2 <\infty
\end{equation}
where $\varphi _\nu (x) = \int _{\mathbb S^1} e^{-is\nu }\varphi (x,s) \,ds$ as before.
(This functional space did not appear in \cite{NW2015}.)
Observe that computing $\lVert \varphi _\nu \rVert_{H^m_\nu (\Sone )}$ involves the Fourier coefficients $\hat\varphi_\nu(\xi)$ of $x\mapsto \varphi_\nu(x)$. Naturally, these are precisely the Fourier coefficients $\hat\varphi(\xi,\nu)$ of $(x,s)\mapsto\varphi(x,s)$, since $\hat\varphi_\nu(\xi)=\int_{\Sone} e^{-ix\xi}\varphi_\nu(x)\,dx=\int_{\Sone} e^{-ix\xi}\int_{\Sone}e^{-is\nu}\varphi(x,s)\,dxds$.
In fact, the somewhat odd-looking norm in \eqref{eq:Bnorm} is equivalent to the usual Sobolev norm:
\begin{lem}\label{lem:eqnorm}
$\Vert \cdot \Vert_{\mathscr H^m}$ is equivalent to $\Vert \cdot \Vert_{H^{m}(\T^2)}$.
\end{lem}
\begin{proof}
Note that for any $\xi, \nu \in 2\pi \Z$ with $\nu \neq 0$, 
\[
\frac{(1+ \xi ^2 + \nu ^2)}{ (1 + \nu ^2)(1+ (\xi /\nu )^2) } 
=\frac{1+ \xi ^2 + \nu ^2}{ \xi ^2 + \nu ^2} \frac{\nu ^2}{ 1 + \nu ^2} 
= \frac{1+ \frac{1}{\xi^2 + \nu ^2}}{1+ \frac{1}{\nu ^2}}
 \]
is in an interval $(\frac12, 1]$. Indeed, the right-hand side is is clearly bounded by 1, and by inserting $\nu=1$ and letting $\lvert\xi\rvert\to\infty$ we see it is bounded from below by $\frac12$. (Note that since $\nu\in2\pi\Z$, the choice $\nu=1$ is actually illegal but the estimate is clearly valid.) Hence the lemma immediately follows in view of the definitions of the two norms.
\end{proof}

\subsection{Conditions \ref{(A1)} and \ref{(A2)}}\label{subsection:a}

 We begin by verifying that conditions \ref{(A1)} and \ref{(A2)} are satisfied when $\mathscr D=L^2 (\mathbb T^2)$, $\mathscr B=\mathscr H^m (\mathbb T^2)$ and $\mathscr E=H^{r}(\mathbb T^2)$ for some sufficiently large $m$ depending only on the unperturbed system $T_0$ given by \eqref{eq:unperturbedsystem}, see Remark \ref{rmk:rdependsonT0}. The dual pairing in \eqref{eq:2} can thus be expressed via the standard inner product $(\cdot,\cdot)_{L^2}$ on $L^2(\T^2)$ as
$$
\langle \varphi,u\rangle=(\varphi,\bar u)_{L^2}.
$$

\begin{remark}\label{rmk:rdependsonT0}
Assume that the map $E_0$ in \eqref{eq:unperturbedsystem} is expanding with an expansion rate $\lambda_0=\min_x E_0'(x)$.
For sufficiently small noise levels $\epsilon$ we have in view of \eqref{convinC1} that $E_\omega$ is an expanding map $\mathbb P$-almost surely with an expansion rate strictly larger than $\lambda=\frac12(\lambda_0+1)>1$. Let $\lambda^{-\frac12}<\rho<1$. By combining \cite[Theorem 1.4]{NW2015} and \cite[Theorem 1.5]{NW2015} we can find a positive integer $m_0>1$ together with numbers $\epsilon_0=\epsilon_0(m)$ and $\nu_0=\nu_0(m)$ depending also on an integer $m\ge m_0$ for which those results are in force. For the rest of this section we assume this to be the case and only consider noise levels $\epsilon$ such that $0\le\epsilon\le\epsilon_0$. We make sure to pick $m$ large enough so that 
\begin{equation}\label{eq:0419}
\log(\lambda^{-m-\frac12} k^\frac12)<0
\end{equation}
where $k\ge2$ is the angle multiplication factor of $E_0$. Then $m$ only depends on the unperturbed system $T_0$.
Fixing $r\ge m>1$, the same is true for the Sobolev regularity index $r$ appearing in $\mathscr E=H^r(\T^2)$.
\end{remark}

\begin{prop}\label{prop:admissible}
$\mathscr H^m (\mathbb T^2)$ is $T$-admissible.
\end{prop}

\begin{proof}
In view of Definition \ref{def:admissible} we need to prove \ref{(A1)}, i.e., that $\Vert \mathcal L_\omega \varphi \Vert_{\mathscr H^m}\le  C\Vert \varphi \Vert_{\mathscr H^m}$ for all $\varphi \in \mathscr H^m$ and some $C>0$.
By \eqref{eq:0811c} and the definition of the norm in $\mathscr H^m$ we have
\begin{equation}\label{eq:normtransferop}
\Vert \mathcal L_\omega \varphi \Vert_{\mathscr H^m}^2 
=
\lVert \mathcal L_{0,\omega}\varphi _0 \rVert_{H^m (\Sone )}^2 + \sum _{\nu \in 2\pi \Z , \; \nu \neq 0}(1+\nu^2)^m\lVert \mathcal L_{\nu,\omega}\varphi _\nu \rVert_{H^m_\nu (\Sone )}^2. 
\end{equation}
It follows from \cite[Theorem 1.5]{NW2015} that  there is a $\rho_0\in(0,1)$ as in Remark \ref{rmk:rdependsonT0} together with a $\nu _0 \in 2\pi\mathbb N$ and a constant $c_0$ such that
\begin{equation}\label{eq:Thm15withtime}
 \Vert \mathcal L_{\nu , \omega}^{(n)} \varphi _\nu\Vert _{H^m_\nu } \leq c_0\rho_0^n\Vert  \varphi _\nu\Vert _{H^m_\nu },\quad n\ge1,
\end{equation}
for any $\nu$ with $\vert \nu \vert > \nu _0$ and  $\mathbb P$-almost every $\omega \in \Omega$. On the other hand, for any $\nu \neq 0$ with $\vert \nu \vert \leq \nu _0$, we have
\[
\nu _0^{-2} (1+ \xi ^2) \leq \nu ^{-2} (\nu ^2 + \xi ^2) = 1+(\xi /\nu )^2 \leq 1 + \xi ^2 ,
\]
so that
\begin{equation}\label{eq:0811d}
\nu _0^{-m} \Vert  \varphi _\nu \Vert _{H^m } \leq  \Vert  \varphi _\nu \Vert _{H^m _\nu} \leq  \Vert  \varphi _\nu \Vert _{H^m } .
\end{equation}
Hence, combining \eqref{eq:Thm15withtime} for $n=1$ with \cite[Theorem 1.4]{NW2015} it follows that there is a constant $C$ such that
\begin{equation*}
 \Vert \mathcal L_{\nu , \omega} \varphi _\nu\Vert _{H^m_\nu } \leq C\Vert  \varphi _\nu\Vert _{H^m_\nu }
\end{equation*}
for all $\nu\ne0$, and that
\[
 \Vert \mathcal L_{0 , \omega} \varphi _0\Vert _{H^m } \leq C\Vert  \varphi _0\Vert _{H^m } .
\]
Combining these estimates with \eqref{eq:normtransferop} we conclude that \ref{(A1)} holds.
\end{proof}

We shall need the following well-known structure result for Sobolev spaces. We include a short direct proof for convenience.

\begin{lem}\label{lem:multiplication}
Let $r>1$. Then there is a constant $C_r$ such that for all $u,\varphi\in H^r(\T^2)$ we have $u\varphi\in H^r(\T^2)$ and
$$
\lVert u\varphi\rVert_{H^r(\T^2)}\le C_r \lVert u\rVert_{H^r(\T^2)}\lVert \varphi\rVert_{H^r(\T^2)}.
$$
\end{lem}

\begin{proof}
To shorten notation we shall denote $(\xi,\nu)$ by $\xi=(\xi_1,\xi_2)$, and write $\langle x\rangle=(1+\lvert x\rvert^2)^{\frac{1}{2}}$ for the Japanese bracket. Following Sj{\"o}strand \cite[Proposition 2.1]{sjostrand2009eigenvalue}, we pass to the Fourier side and see that the result follows if we show that
\begin{equation}\label{eq:suffest0}
\sum_{\xi\in (2\pi\Z)^2 }\langle\xi\rangle^rw(\xi) (\langle\cdot\rangle^{-r}\widetilde \varphi\ast \langle\cdot\rangle^{-r}\widetilde u)(\xi)\le C_{r}\lVert w\rVert\lVert \widetilde \varphi\rVert\lVert \widetilde u\rVert
\end{equation}
for all non-negative $\widetilde u,\widetilde \varphi,w\in \ell^2$, where $\ast$ denotes convolution and the norms are the ones in $\ell^2$. Indeed, 
$$
\lVert u\varphi \rVert_{H^r(\T^2)}=\lVert \langle\cdot\rangle^{r} \widehat{u\varphi }\rVert=\sup_{0\ne w\in \ell^2}\frac{\langle \langle\cdot\rangle^{r} \widehat{u\varphi },w\rangle_{\ell^2}}{\lVert w\rVert},
$$
so writing $\widetilde \varphi=\langle\cdot\rangle^{r} \widehat{\varphi}$ and $\widetilde u=\langle\cdot\rangle^{r} \widehat{u}$ the claim easily follows.

To establish \eqref{eq:suffest0} we see that the left-hand side is equal to
$$
\sum_{\xi\in (2\pi\Z)^2 }\sum_{\eta+\zeta=\xi}\langle\xi\rangle^rw(\xi) \langle\eta\rangle^{-r}\widetilde \varphi(\eta) \langle\zeta\rangle^{-r}\widetilde u(\zeta)\le \mathrm I+\mathrm{II},
$$
where $\mathrm I,\mathrm{II}$ are the sums over $\{\lvert\eta\rvert\ge\lvert\xi\rvert/2\}$ and $\{\lvert\zeta\rvert\ge\lvert\xi\rvert/2\}$, respectively. In $\mathrm I$ we have $\langle\xi\rangle^r\le 2^r\langle\eta\rangle^r$, so
$$
\mathrm I\le 2^r\sum_{\zeta}\bigg(\sum_{\xi}w(\xi)\widetilde \varphi(\xi-\zeta)\bigg)\langle\zeta\rangle^{-r}\widetilde u(\zeta)\le 2^r\lVert w\rVert\lVert\widetilde \varphi\rVert\lVert \langle\cdot\rangle^{-r}\widetilde u\rVert_{\ell^1},
$$
where $\lVert \langle\cdot\rangle^{-r}\widetilde u\rVert_{\ell^1}\le C_{r}\lVert \widetilde u\rVert$
by the Cauchy-Schwartz inequality since $r>1$. Hence $\mathrm{I}$ is bounded by a constant times $\lVert w\rVert\lVert \widetilde \varphi\rVert\lVert \widetilde u\rVert$. In $\mathrm{II}$ we have $\langle\xi\rangle^r\le 2^r\langle\zeta\rangle^r$ so by symmetry the same estimate holds for $\mathrm{II}$, and \eqref{eq:suffest0} follows.
\end{proof}

\begin{prop}\label{prop:associated}
$H^r(\T^2)$ is associated with $\mathscr H^m (\mathbb T^2)$ for all $r\ge m>1$.
\end{prop}

\begin{proof}
By Lemma \ref{lem:multiplication} the multiplication map $H^r(\T^2)\times H^r(\T^2)\to H^r(\T^2)$ is continuous. Hence, it suffices to prove that the first inequality in \ref{(A2)} is satisfied when $\Vert\cdot\Vert_\mathscr{B}=\Vert\cdot\Vert_{\mathscr{H}^m}$ and $\Vert\cdot\Vert_\mathscr{E}=\Vert\cdot\Vert_{H^r}$. It is even enough to show that
\begin{equation}\label{eq:proofA2}
\vertiii{u\varphi }\le C_{m,r}\vertiii{\varphi}\lVert u\rVert_{H^{r}}.
\end{equation}
Indeed, if $\psi\in H^m_\nu(\Sone)$ and $\nu\ne0$ then $\lVert\psi\rVert_{L^2(\Sone)}\le \langle\nu\rangle^m\lVert\psi\rVert_{H_\nu^m(\Sone)}$, which implies that 
$\lVert\varphi\rVert_{L^2(\T^2)}\le\lVert\varphi\rVert_{\mathscr H^m(\T^2)}$.
Hence
$$
\langle \varphi,u\rangle=\langle u\varphi,1_{\T^2}\rangle\le \lVert u\varphi\rVert_{L^2(\T^2)}\le\lVert u\varphi\rVert_{\mathscr H^m(\T^2)}
$$
so we always have $\max\{\langle \varphi,u\rangle,\vertiii{u\varphi }\}=\vertiii{u\varphi }$ in this case, which proves the claim.

But by Lemma \ref{lem:eqnorm}, $\mathscr H^m(\T^2)$ coincides with $H^m(\T^2)$, so if $r\ge m>1$ then
$$
\vertiii{u\varphi }\le C_{m}\vertiii{\varphi}\vertiii{u}\le C_m\vertiii{\varphi}\lVert u\rVert_{H^{r}}
$$
by Lemma \ref{lem:multiplication} which establishes \eqref{eq:proofA2} and the proof is complete.
\end{proof}

\subsection{Condition (UG)}\label{subsection:s1}

Having established that $H^r(\T^2)$ is associated with the $T$-admissible Banach space $\mathscr H^m(\T^2)$ we proceed to prove that $T$ 
satisfies the uniform spectral gap condition with respect to $\mathscr H^m(\T^2)$ and observables in $H^r(\T^2)$.
We begin by verifying the existence of invariant measures.

\begin{prop}\label{prop:GUmeasures}
For any sufficiently small noise level $\epsilon>0$ there exists a unique $h\in L^\infty(\Omega , \mathscr H^m(\T^2))$ and a  family of probability measures $\{ \mu_\omega \} _{\omega \in \Omega}$ 
such that 
\[
\mathcal L_{\omega } h_{\omega} =h_{\sigma \omega}, \quad \langle h_{\omega }, u \rangle = \int u \,d\mu _\omega , \quad \essinf _{\omega \in \Omega } \Vert h_{\omega}\Vert_{\mathscr H^m} \geq 1
\]
 for $\mathbb P$-almost every $\omega \in \Omega$ and any observable $u \in H^r(\T^2)$.
\end{prop}

\begin{proof}
For sufficiently small $\epsilon>0$ there is by \cite[Theorem 1.3]{NW2015} a uniquely defined function $h=h(\epsilon):\omega\mapsto h_\omega$ such that $\mathcal L_\omega h_\omega=h_{\sigma\omega}$ and $\langle h_\omega,1_{\T^2}\rangle=1$, together with a family of probability measures $\{\mu_\omega\}_{\omega\in\Omega}$ such that $d\mu_\omega= h_\omega d\mathrm{Leb}_{\T^2}$. Hence, if $u\in H^r(\T^2)\subset L^2(\T^2)$ we have $\langle h_\omega,u\rangle=\int u\,d\mu_\omega$. Note that by \cite[Theorem 1.3]{NW2015}, $h_\omega(x,s)$ for $(x,s)\in\T^2$ has the form of a tensor product $h_\omega(x,s)=h_\omega(x)\otimes 1_{\Sone}(s)$ with $x\mapsto h_\omega(x)\in \mathscr C^\infty(\Sone)$ satisfying $\int_{\Sone} h_\omega(x)\,dx=1$ and $\esssup_\omega\lVert h_\omega\rVert_{H^m(\Sone)}\le C_m$ for any $m\in\N$. In particular, decomposing $h_\omega(x,s)$ into Fourier modes we find that the mode corresponding to $\nu=0$ satisfies $(h_\omega)_0(x)=\int h_\omega(x,s)\,ds=h_\omega(x)$, while $(h_\omega)_\nu(x)=0$ for all $\nu\in 2\pi\Z$ such that $\nu\ne0$. Inspecting the definition of the norm $\Vert \cdot\Vert_{\mathscr H^m(\T^2)}$ this implies that
$$
\lVert h_\omega\rVert_{\mathscr H^m(\T^2)}^2= \lVert (h_\omega)_0\rVert_{H^m(\Sone)}^2\le C_m^2
$$
$\mathbb P$-almost surely, which shows that $h\in L^\infty(\Omega,\mathscr H^m(\T^2))$.
Moreover, by the Cauchy-Schwartz inequality we have
$$
1=\int_{\Sone}h_\omega(x)\,dx\le \lVert h_\omega\rVert_{L^2(\Sone)}=\lVert (h_\omega)_0\rVert_{L^2(\Sone)}
$$
and since $m\ge0$ we thus find that
$$
\lVert h_\omega\rVert_{\mathscr H^m(\T^2)}^2= \lVert (h_\omega)_0\rVert_{H^m(\Sone)}^2\ge \lVert (h_\omega)_0\rVert_{L^2(\Sone)}^2\ge 1,
$$
so $\essinf _{\omega \in \Omega } \Vert h_{\omega}\Vert_{\mathscr H^m} \geq 1$.
\end{proof}

We next record a perturbation lemma.

\begin{lem}\label{lem:perturbation}
There is a $\rho\in(0,1)$, an arbitrarily large integer $n_0$ and a number $\epsilon(n_0)$ such that for all $0\le \epsilon<\epsilon(n_0)$ we have
\begin{equation*}
\Vert \mathcal L_\omega^{(n_0)} \varphi \Vert_{\mathscr H^m} 
\le \rho^{n_0}\lVert \varphi \rVert_{\mathscr H^m }+C_{n_0}\lVert \varphi_0\rVert_{H^{m-1}(\Sone)}
\end{equation*}
where $\varphi_0(x)=\int_{\Sone}\varphi(x,s)\,ds$. If $\langle \varphi,1_{\T^2}\rangle=0$ we may take $C_{n_0}=0$.
\end{lem}

\begin{proof}
By \eqref{eq:0811c} together with \eqref{eq:normtransferop}--\eqref{eq:Thm15withtime} we have
\begin{equation}\label{eq:UGfirstestimate}
\Vert \mathcal L_\omega^{(n)} \varphi \Vert_{\mathscr H^m}^2 
\le
\lVert \mathcal L_{0,\omega}^{(n)}\varphi _0 \rVert_{H^m }^2 + \sum _{0<\lvert\nu\rvert\le\nu_0}\langle\nu\rangle^{2m}\lVert \mathcal L_{\nu,\omega}^{(n)}\varphi _\nu \rVert_{H^m_\nu }^2
+\sum _{\lvert\nu\rvert>\nu_0}c_0\rho_0^{2n}\langle\nu\rangle^{2m}\lVert \varphi _\nu \rVert_{H^m_\nu }^2 
\end{equation}
where $\rho_0\in(0,1)$.
Let $\mathcal L_{T_0}$ denote the transfer operator induced by the unperturbed map $T_0$ and $(\mathcal L_{T_0})_\nu$ its restriction corresponding to  $\mathcal L_{\nu,\omega}$. By applying \cite[Proposition 4.5]{NW2015} to each $\mathcal L_{\nu,\omega}$ with $\lvert\nu\rvert\le\nu_0$ we can find a $\rho_1\in(0,1)$ together with an integer $n_0$ (which can be chosen arbitrarily large) and a number $\epsilon(n_0)$ such that for all $\lvert\nu\rvert\le\nu_0$ and $0\le\epsilon<\epsilon(n_0)$ we have
\begin{equation}\label{eq:perturbation}
\esssup_{\omega}\lVert (\mathcal L_{\nu,\omega}^{(n_0)}-(\mathcal L_{T_0}^{n_0})_\nu)\varphi_\nu\rVert_{H^m} \le \rho_1^{n_0}\lVert\varphi_\nu\rVert_{H^m}.
\end{equation}
By \cite[Theorem 4.4]{NW2015} the spectral radius of each $(\mathcal L_{T_0})_\nu:H^m\to H^m$ is strictly less than 1 when $\nu\ne0$. Hence, by increasing $n_0$ if necessary, we have for some $\rho_2\in(0,1)$ and all $0<\lvert\nu\rvert\le\nu_0$ that
\begin{equation}\label{eq:boundbasedonspectralradius}
\lVert (\mathcal L_{T_0}^{n_0})_\nu\varphi_\nu\rVert_{H^m} \le \rho_2^{n_0}\lVert\varphi_\nu\rVert_{H^m}.
\end{equation}
When $\nu=0$, the transfer operator $(\mathcal L_{T_0})_{\nu=0}$ is the standard transfer operator induced by the uniformly expanding map $E_0$ on the circle. Thus, by \cite[Lemma 15]{BaillifBaladi} we can (by taking $n_0$ sufficiently large and increasing $\rho_2$ if necessary) find a constant $C_{n_0}$ such that 
\begin{equation}\label{eq:boundbasedonBaillifBaladi}
\lVert (\mathcal L_{T_0}^{n_0})_0\varphi_0\rVert_{H^m} \le \rho_2^{n_0}\lVert\varphi_0\rVert_{H^m}+C_{n_0}\lVert\varphi_0\rVert_{H^{m-1}}
\end{equation}
for all $\varphi_0\in H^m$. On the other hand, if $\langle\varphi,1_{\T^2}\rangle=0$ then the Fourier mode $\varphi_0(x)$ corresponding to $\nu=0$ satisfies
\begin{equation}\label{eq:inthekernel}
\int_{\Sone}\varphi_0(x)\,dx=\int_{\Sone}\bigg(\int_{\Sone}\varphi(x,s)\,ds\bigg)\,dx=\langle\varphi,1_{\T^2}\rangle=0.
\end{equation}
This implies that \eqref{eq:boundbasedonBaillifBaladi} holds with $C_{n_0}=0$ for such $\varphi$. 
Indeed, it is well-known that the only eigenvalue of $(\mathcal L_{T_0})_{\nu=0}$ on the unit circle is the simple eigenvalue 1, and the Lyapunov subspace associated to the eigenvalues of modulus 1 is one-dimensional and spanned by the function $h(\epsilon=0)$ given in the first paragraph of the proof, see \cite[Chapter 3]{Mane} or \cite[Theorem 4.4]{NW2015}.
Let us write $h_0$ for this function, indicating that $h_0$ does not depend on $\omega$ in contrast to $h_\omega=h_\omega(\epsilon)$ for $\epsilon>0$.
If $\pi_0:H^m(\Sone)\to H^m(\Sone)$ denotes the projection $\pi_0(\psi)=h_0(x)\int_{\Sone} \psi(x)\,dx$ then $\varphi_0$ is in the kernel of $\pi_0$ by \eqref{eq:inthekernel}. In the notation of Definition \ref{dfn:OD} we thus have a splitting
$H^m=\mathscr F _{0}\oplus \mathscr R_{0}
$ with $\mathscr F _{0}=\C h_0$ and if $\psi \in \mathscr R _{ 0}$ then
\[
\lim _{n\rightarrow \infty} \frac{1}{n} \log \lVert (\mathcal L_{T_0} ^{n})_0 \psi \rVert_{H^m} \leq \widetilde{ \Lambda } ((\mathcal L_{T_0})_0)
\]
for some number $ \widetilde{\Lambda} ((\mathcal L_{T_0})_0)
 <0$. After increasing $\rho_2$ if necessary we may assume that $ \widetilde{\Lambda} ((\mathcal L_{T_0})_0)
 <\log\rho_2$ which proves the claim 
 since $\varphi_0\in\mathscr R_0$.

By combining \eqref{eq:perturbation}--\eqref{eq:boundbasedonBaillifBaladi} we see that
$$
\lVert \mathcal L_{0,\omega}^{(n_0)}\varphi _0 \rVert_{H^m }^2 
\le K_1(n_0)\lVert\varphi_0\rVert_{H^m}^2+C_{n_0}^2\lVert\varphi_0\rVert_{H^{m-1}}^2
$$
where $K_1(n_0)=(\rho_1^{n_0}+\rho_2^{n_0})(\rho_1^{n_0}+\rho_2^{n_0}+2C_{n_0})$ and by using \eqref{eq:0811d} twice we also get
$$
\sum _{0<\lvert\nu\rvert\le\nu_0}\langle\nu\rangle^{2m}\lVert \mathcal L_{\nu,\omega}^{(n_0)}\varphi _\nu \rVert_{H^m_\nu }^2
\le K_2(n_0) \sum _{0<\lvert\nu\rvert\le\nu_0}\langle\nu\rangle^{2m}\lVert \varphi _\nu \rVert_{H^m_\nu }^2
$$
where $K_2(n_0)=\nu_0^{2m}(\rho_1^{n_0}+\rho_2^{n_0})^2$. Since $0<\rho_1,\rho_2<1$ we can choose $n_0$ so large that $K_1(n_0),K_2(n_0)<1$, which means that there is a $\rho_3\in(0,1)$ such that both $K_1(n_0)$ and $K_1(n_0)$ are bounded by $\rho_3^{2n_0}$. Inserting this into \eqref{eq:UGfirstestimate} we find that 
\begin{align*}
\lVert \mathcal L_\omega^{(n_0)} \varphi \rVert_{\mathscr H^m}^2&\le \rho_3^{2n_0}\bigg(\lVert \varphi _0 \rVert_{H^m }^2 + \sum _{0<\lvert\nu\rvert\le\nu_0}\langle\nu\rangle^{2m}\lVert \varphi _\nu \rVert_{H^m_\nu }^2\bigg)
\\&\quad+\sum _{\lvert\nu\rvert>\nu_0}c_0\rho_0^{2n_0}\langle\nu\rangle^{2m}\lVert \varphi _\nu \rVert_{H^m_\nu }^2+C_{n_0}^2\lVert\varphi_0\rVert_{H^{m-1}}^2. 
\end{align*}
Since $\rho_0<1$ we may assume that $c_0\rho_0^{2n_0}=\rho_4^{2n_0}$ for some $\rho_4\in(0,1)$ by increasing $n_0$ if necessary. The result now follows by picking $\rho=\max(\rho_3,\rho_4)$.
\end{proof}

\begin{prop}\label{prop:UG}
There is an $\epsilon_0$ such that for all $0\le\epsilon<\epsilon_0$, $T$
satisfies the uniform spectral gap condition with respect to $\mathscr H^m$ and $H^r(\T^2)$.
\end{prop}

\begin{proof}
In view of Proposition \ref{prop:GUmeasures} we only need to verify \eqref{eq:0212c2}, namely, that  there are constants $C>0$ and $\rho \in (0,1)$ such that if $\varphi \in \mathscr H^m(\T^2)$ satisfies $\langle \varphi ,1_{\T^2} \rangle =0$, then
\begin{equation}\label{eq:GU}
\esssup _{\omega \in \Omega} \Vert \mathcal L_\omega^{(n)} \varphi \Vert_{\mathscr H^m} \leq C\rho ^n \Vert \varphi \Vert_{\mathscr H^m} \quad \text{for all $n\geq 1$.}
\end{equation}
By Lemma \ref{lem:perturbation} there is a $\rho\in(0,1)$ together with an integer $n_0$ and a number $\epsilon(n_0)$ such that for all $0\le \epsilon<\epsilon(n_0)$ we have
\begin{equation}\label{eq:UGfinalestimate}
\Vert \mathcal L_\omega^{(n_0)} \varphi \Vert_{\mathscr H^m} 
\le \rho^{n_0}\lVert \varphi \rVert_{\mathscr H^m }
\end{equation}
for all $\varphi\in\mathscr H^m$ such that $\langle \varphi ,1_{\T^2} \rangle =0$.

Now let $n\ge1$ be arbitrary and write $n=kn_0+r$ where $0\le r<n_0$. Using \ref{(A1)} and \eqref{eq:UGfinalestimate} this immediately gives
$$
\Vert \mathcal L_\omega^{(n)} \varphi \Vert_{\mathscr H^m} 
\le C_0^r\rho^{kn_0}\lVert \varphi \rVert_{\mathscr H^m }=
\bigg(\frac{C_0}{\rho}\bigg)^r\rho^{kn_0+r}\lVert \varphi \rVert_{\mathscr H^m }\le C\rho^n\lVert \varphi \rVert_{\mathscr H^m }
$$
where $C$ is independent of $0\le r<n_0$. This establishes \eqref{eq:GU} and the proof is complete.
\end{proof}

\subsection{Condition (LY)}\label{ss:s2}

Here we show that $T$ satisfies the Lasota-Yorke inequality condition. {Recall that $(\Omega,\mathcal F,\mathbb P)$ is assumed to be a Lebesgue-Rokhlin probability space. Since $\mathscr H^m(\T^2)$ is equivalent to $H^m(\T^2)$ which is separable, condition (ii) in Definition \ref{dfn:LY} holds as long as $\mathcal L$ is $(\mathcal F, \mathcal S_{L(H^m(\T^2))})$-measurable, i.e., if $\mathcal L:\Omega\to L(H^m(\T^2))$ is strongly measurable. The fact that this holds can easily be shown by adapting the proof on page 981 in \cite{NW2015} of strong measurability of the reduced transfer operator $\mathcal L_{\nu,\omega}$. }
Thus it remains to verify that there is a Banach space $(\mathscr B_+,\vert\cdot\vert)$ such that $\mathscr H^m(\T^2)$ is compactly embedded in $\mathscr B_+$ and $\mathscr B_+$ is continuously embedded in $L^2(\T^2)$, that $\mathcal L$ is $\mathbb P$-almost surely bounded on $\mathscr B_+$ and that \eqref{eq:forProp35} holds, namely, 
\begin{equation}\label{eq:LYcond}
\lVert\mathcal L_\omega^{(n_0)}\varphi\rVert_{\mathscr H^m}\le\alpha_\omega\lVert\varphi\rVert_{\mathscr H^m}+\beta_\omega\lvert\varphi\rvert,\quad\varphi\in\mathscr H^m,
\end{equation}
for some positive integer $n_0$, where $\alpha$ and $\beta$ are random variables with values in $\mathbb R_+$ with
$
\mathbb E_{\mathbb P}\left[ \alpha\right] < 1$ 
and $ \mathbb E_{\mathbb P}\left[ \beta \right] <\infty .
$
By the following result we can take $\mathscr B_+=H^{m-1}(\T^2)$.

\begin{prop}\label{prop:LY}
There is an $\epsilon_0$ such that for all $0\le\epsilon<\epsilon_0$, $T$
satisfies the Lasota-Yorke inequality condition with respect to $\mathscr H^m\subset H^{m-1}(\T^2)$.
\end{prop}

\begin{proof}
By Lemma \ref{lem:perturbation} we have
\begin{equation}\label{eq:LYwithournorms}
\Vert \mathcal L_\omega^{(n_0)} \varphi \Vert_{\mathscr H^m} 
\le \rho^{n_0}\lVert \varphi \rVert_{\mathscr H^m }+C_{n_0}\lVert \varphi_0\rVert_{H^{m-1}(\Sone)}
\end{equation}
where $\rho\in (0,1)$. 
Let $H^{j,k}(\T^2)$ be the anisotropic Sobolev space with norm
$$
\Vert\varphi\Vert_{H^{j,k}}^2=\sum_{\xi,\nu\in2\pi\Z}((1+\xi^2)^j+(1+\nu^2)^k)\lvert\hat\varphi(\xi,\nu)\rvert^2.
$$
If $\varphi_0$ is the Fourier mode $\varphi_0(x)=\int_{\Sone}e^{-i\nu s}\varphi(x,s)\,ds$ then clearly
$$
\lVert\varphi_0\rVert_{H^{m-1}(\Sone)}\le \Vert\varphi\Vert_{H^{m-1,0}}\le \Vert\varphi\Vert_{H^{m-1,m-1}}
$$
where the right-hand side is equivalent to the usual norm in $H^{m-1}(\T^2)$.
Hence \eqref{eq:LYwithournorms} gives
$$
\Vert \mathcal L_\omega^{(n_0)} \varphi \Vert_{\mathscr H^m} 
\le \rho^{n_0}\lVert \varphi \rVert_{\mathscr H^m }+C_{n_0}\lVert \varphi\rVert_{H^{m-1}(\T^2)}.
$$
It is well-known that $H^{m}(\T^2)$ is compactly embedded in $H^{m-1}(\T^2)$, and since $\Vert \cdot \Vert_{\mathscr H^m }$ is equivalent to $\Vert \cdot \Vert_{H^m(\T^2) }$ we can thus take $\mathscr B_+=H^{m-1}(\T^2)$. (We also need to check that $\mathcal L_\omega$ is $\mathbb P$-almost surely bounded on $H^{m-1}(\T^2)$, but this is immediate from \ref{(A1)} by the same reasoning.) In particular, \eqref{eq:LYcond} holds with $\alpha_\omega\equiv \rho^{n_0}$ and $\beta_\omega\equiv C_{n_0}$.
\end{proof}

\begin{proof}[Proof of Theorem \ref{thm:v}]
The result follows by virtue of Propositions \ref{prop:admissible}, \ref{prop:associated}, \ref{prop:UG} and \ref{prop:LY}.
\end{proof}

\begin{proof}[Proof of Theorem \ref{thm:v2}]
We first prove item (i). 
Since $g_\omega (x,s)$ does not depend on $s$, a straightforward calculation for Sobolev norms 
 shows that there exists a constant $B'\geq 1$ such that 
\[
\Vert \mathcal L_{it, \omega }^{(n)} \varphi \Vert _{H^m} \leq \max\{1, \vert t\vert ^m\} 
 B' \Vert \varphi \Vert _{H^m}
\] 
for any $m\in \mathbb N_0$, $\varphi \in H^m(\mathbb T^2)$, $t\in \mathbb R$, $n\in \mathbb N_0$ and $\mathbb P$-almost every $\omega \in \Omega$
(see the proof of \cite[Lemma 4.3]{NW2015} for the case of $\mathbb S^1$ instead of $\mathbb T^2$). 
Hence, the second part of (HK2) immediately follows.

We next prove item (ii). 
By the assumption in item (i), there exists $\widetilde g\in L^\infty (\Omega , \mathscr C^m(\mathbb S^2))$ such that $ g_\omega (x, s) =\widetilde g_\omega (x)$ for all $(\omega , x, s) \in \Omega \times \mathbb T^2$. 
As in \eqref{eq:1115}, 
given $\theta \in \mathbb C$, $\nu \in \mathbb Z$ and $n\in \mathbb N$, we introduce the Fourier-decomposed twisted  transfer operator cocycle $\mathcal L_{\theta , \nu , \omega } : L^2(\mathbb S^1) \to L^2(\mathbb S^1)$ by
\[
\mathcal L_{\theta , \nu , \omega } ^{(n)}\psi (x) = \sum _{ E_\omega ^{(n)} (y)=x} \frac{e^{\theta  (S_n\widetilde g)_\omega - i\nu \tau _\omega ^{(n)} } \: \psi }{\frac{d}{dy}E_\omega ^{(n)} } (y) .
\]
Since $ g_\omega $ does not depend on $s$, we have for each  $\varphi (x,s) =\sum _{\nu \in \mathbb Z} e^{i\nu s} \psi _\nu (x) $ that
\[
\mathcal L_{\theta  , \omega } ^{(n)}\varphi (x,s) = 
\sum _{\nu \in \mathbb Z} e^{i\nu s}  \mathcal L_{\theta  , \nu , \omega } ^{(n)} \psi _\nu (x) .
\]
Fix $t\in \mathbb R$. Since  $\mathcal L_{it , \nu  , \omega _0} ^{(\ell _0)}$  can be considered as the transfer operator    of the  expanding map $E_{\omega _0} ^{(\ell _0)}$ with weight $e^{i(t (S_{\ell _0} \widetilde g)_{\omega _0} -\nu \tau _{\omega _0}^{(\ell _0)}) } $ whose real part is bounded by $1$, due to \cite{Baladibook2} together with the argument in the proof of \cite[Proposition 4.1]{NW2015}, one can conclude that the spectral radius  of $\mathcal L_{it , \nu  , \omega _0} ^{(\ell _0)}$
 is bounded by $1$ and   the essential spectral radius of $\mathcal L_{it , \nu  , \omega _0} ^{(\ell _0)}$ 
  is strictly less than $1$  on $H^m_\nu (\mathbb S^1)$. 
  Moreover, the uniform exponential decay in the semiclassical limit $\nu\to\infty$ in \cite[Theorem 1.5]{NW2015} that holds for $(\mathcal L_{0 , \nu ,\omega } ^{(n)} )_{n\geq 0}$ with $\mathbb P$-almost every $\omega$ also holds for $(\mathcal L_{it , \nu ,\omega _0} ^{(\ell _0 N)} )_{N\geq 0}$: 
for any $\rho >\lambda ^{-1/2}$ and sufficiently large $m$, there are $\nu _0=\nu _0(t)$ and $c_0=c_0(t)$ such that (if the noise level $\epsilon$ is sufficiently small)
\[
\Vert ( \mathcal L_{it , \nu , \omega _0}^{(\ell _0 )} ) ^N \Vert = \Vert \mathcal L_{it , \nu , \omega _0}^{(\ell _0 N )} \Vert  \leq c_0 \rho ^{\ell _0 N},\quad N\geq 1, \; \vert \nu \vert \geq \nu_0. 
\]
This fact is obtained by repeating the proof of \cite[Theorem 1.5]{NW2015}; the only difference is that the function $(d E_\omega^{(n)}/dx)^{-1}$ is replaced by $e^{it(S_{\ell _0N} \widetilde g)_{\omega_0}}(d E_{\omega_0}^{(\ell _0N)}/dx)^{-1}$ which does not lead to any complications in the proof.
Therefore, 
 the spectral radius  of $\mathcal L_{it ,   \omega _0} ^{(\ell _0)}$
 is bounded by $1$ and   the essential spectral radius of $\mathcal L_{it ,   \omega _0} ^{(\ell _0)}$ 
  is strictly less than $1$ on $H^m (\mathbb T^2)$  (see the proof of Lemma \ref{lem:perturbation}).

Furthermore, given  a compact interval $J$ of $\mathbb R\setminus \{0\}$,
$\mathcal L_{ it, \omega _0}^{(\ell_0)}$ has no eigenvalue with radius $1$  for each $t\in J$.
Indeed, arguing by contradiction, assume that there exist $c\in \mathbb R$  and $\gamma \in H^m(\mathbb T^2)$ such that
\[
\sum _{z=T_{\omega _0} ^{(\ell_0)} (\zeta ) } \frac{e^{i t(S_{\ell_0}g)_{\omega _0}(\zeta )}\gamma (\zeta )}{\vert \det DT_{\omega _0} ^{(\ell_0)}\vert }= e^{ic} \gamma (z) .
\]
In fact, according to the remark on p.~1479 in \cite{Faure}
we even have $\gamma \in \mathscr C ^\infty (\mathbb T^2)$.
Then, repeating the argument in \cite[Theorem 4.4]{NW2015}, 
we obtain
\[
t(S_{\ell_0}g)_{\omega _0} (\zeta ) = \psi (\zeta ) - \psi (T_{\omega _0} ^{(\ell_0)} (\zeta ) ) -c  +2\pi n(\zeta )
\]
with $\psi (\zeta ) = \Arg (\gamma (\zeta))$  where $\Arg$ is the principal value, for some integer valued function $\zeta \mapsto n(\zeta )$.
However, $n(\zeta )$ must be a constant function of $\zeta$ since it can be written as a linear combination of smooth functions.
So, we conclude that $(S_{\ell_0}g)_{\omega _0} $ is cohomologous to a constant, which is a contradiction to the statement about $g$ in item (ii) of Theorem \ref{thm:v2}.
In conclusion, the spectral radius of $\mathcal L_{it, \omega _0}^{(\ell _0)}$ is strictly smaller than $1$ for each $t\in J$.
Therefore, since $J$ is compact and $t\mapsto \mathcal L_{it, \omega _0}^{(\ell _0)}$ is a continuous map from $J$ to the space of bounded operators on $H^m(\mathbb T^2)$ (see \eqref{eq:1104c}), we get (HK3) by applying \cite[Corollary III.13]{HH2000}.
This completes the proof.
\end{proof}

\section{Proofs of Theorem \ref{thm:clt2}, \ref{thm:ldp2} and \ref{thm:lclt2}}\label{section:proofofabstract}

\subsection{Regularity of the top Oseledets space}\label{subsection:A1}

As in \cite{DFGV2018}, we start the proof of  Theorems \ref{thm:clt2}, \ref{thm:ldp2} and \ref{thm:lclt2}
by establishing the regularity of the top Oseledets space of the twisted transfer operator cocycles.

Due to the observation in Subsection \ref{ss:se}, 
it follows from (LY) that   $(T, g)$ satisfies the Oseledets decomposition condition (OD):  there is a real number $K $ such that for any $\theta$ with sufficiently small absolute value, one can find a real number  $ \Lambda _\theta >K$ and 
   a  measurable splitting of $\mathscr B$ into closed subspaces
$
\mathscr B=\mathscr F _{\theta, \omega}\oplus \mathscr R_{\theta , \omega}
$ with $
\dim{\mathscr F _{\theta ,\omega}}<\infty$ such that 
$\mathbb P$-almost surely
the following holds:
\begin{enumerate}
\item[$\mathrm{(1)}$]  
$
\mathcal L_{\theta ,\omega } \mathscr F _{\theta ,\omega } =\mathscr F_{\theta  , \sigma \omega }$ and $\mathcal L_{\theta , \omega } \mathscr R _{\theta , \omega } \subset \mathscr R _{\theta , \sigma \omega },
$
\item[$\mathrm{(2)}$] If $\varphi \in \mathscr  F_{\theta , \omega }\backslash \{ 0\}$, then
$
\lim _{n\rightarrow \infty} \frac{1}{n} \log \lVert \mathcal L_{\theta , \omega}^{(n)} \varphi \rVert  =\Lambda _\theta ,
$
\item[$\mathrm{(3)}$] If $\varphi \in \mathscr R _{\theta , \omega}$, then
$
\limsup _{n\rightarrow \infty} \frac{1}{n} \log \lVert \mathcal L_{\theta , \omega }^{(n)} \varphi \rVert \leq  K .
$
\end{enumerate}
For $\theta \in \mathbb C$ and $\omega \in \Omega$, let  $\mathcal L_{\theta , \omega}^* : \mathscr B^\prime \to \mathscr B^\prime$ be the adjoint operator of $\mathcal L_{\theta , \omega}$ given by
 \[
 (\mathcal L_{\theta , \omega}^* u) ( \varphi ) = u(\mathcal L_{\theta , \omega} \varphi ), \quad u \in \mathscr B^\prime, \quad \varphi \in \mathscr B.
 \]
 Recall that given a map $\mathbf A$ from 
an open subset $U$ of 
a normed vector space $X$ to a  normed vector space $Y$, the Fr\'echet derivative $D\mathbf A (x )$ of $\mathbf A$ at $x \in U$ is a bounded operator from $X$ to $Y$ such that
\[
\lim _{\Vert y  \Vert  _X \to 0} \frac{\Vert \mathbf A (x +y  ) -\mathbf A( x) -  D\mathbf A (x )(y)  \Vert _Y }{\Vert y \Vert _X} =0 .
\]
The Fr\'echet derivative  of $D\mathbf A : U\to L(X, Y) $ at $x \in U$ is denoted by $D^2\mathbf A (x)$, 
 where $L(X,Y)$ is the space of bounded linear operators from $X$ to $Y$. 
$\mathbf A$ 
 is said to be of class $\mathscr C^1$ (resp.~$\mathscr C^2$) if the map $D\mathbf A $ 
 (resp.~$D^2\mathbf A$)
  is continuous. 
We also use the notation $\mathbf A^\prime$, $\mathbf A^{\prime \prime}$ for $D\mathbf A $, $D^2\mathbf A$.

\begin{remark}\label{rmk:1204}
In this section we 
use $\mathbb B$ as a sufficiently small ball in $\mathbb C$ centered at $0$ on which the  maps given below are well-defined, 
even if it may change between occurrences. 
Furthermore,  
for simple description, we identify each $\varphi \in X$  
 with a map $t\mapsto t\varphi$ from
$\mathbb C$ to $X$ 
  for $X= L^\infty (\Omega , \mathscr B)$
   or  $L^\infty (\Omega , \mathbb C)$, 
 if it makes no confusion. 
 Moreover,  for a $\mathscr C^1$ map $\varphi :\mathbb B\to X$, 
 we permit us to write  $\varphi ^\prime (\theta )  (\omega ) = \psi (\omega )$
  with some $\psi \in X$ to mean $\varphi ^\prime  (\theta ) (t) (\omega ) =t \psi (\omega )$ for every $t\in \mathbb B$. 
Finally,
in this section 
 we hope to use the notation $h$ as  the $\mathscr C^1$ map from $\mathbb B$ to $L^\infty (\Omega , \mathscr B)$ (in Theorem \ref{thm:1204}), 
  which conflicts to  another map  $h: \Omega \to \mathscr B$ given in (UG). 
 To avoid notational confusion, we instead use notation $h_0$ for  the map 
 $h: \Omega \to \mathscr B$
  of (UG). 
\end{remark}

The following theorem is crucial in this section.
\begin{thm}\label{thm:1204}
There are a ball $\mathbb B$  in $\mathbb C$ centered at $0$, $\mathscr C^1$ maps $h: \mathbb B\to L^\infty (\Omega , \mathscr B)$, $w : \mathbb B\to L^\infty (\Omega , \mathscr B^\prime )$, and a $\mathscr C^2$ map  $\lambda : \mathbb B\to L^\infty (\Omega , \mathbb C )$ such that the following hold:
\begin{enumerate}
\item[$\mathrm{(1)}$]  
 For every $\theta \in \mathbb B$ and 
 $\mathbb P$-almost every $\omega \in \Omega$, 
\[
\mathcal L_{\theta ,\omega } h _{\theta ,\omega } = \lambda _{\theta , \omega } h_{\theta  , \sigma \omega } ,
\quad \text{ and}
 \quad 
\langle h _{\theta ,\omega } ,1_M\rangle =1 .
\]
Furthermore, $h(0) =h_0$  and $\lambda (0) =1$ $\mathbb P$-almost surely.
\item[$\mathrm{(2)}$]  If we define  $\widehat{\Lambda } : \mathbb B\to L^\infty (\Omega , \mathbb R)$ by $\widehat{\Lambda }(\theta ) (\omega ) = \log \vert \lambda (\theta ) (\omega ) \vert$ for $\theta \in \mathbb B$ and $\omega \in \Omega$, then 
\[
\widehat{\Lambda }(0) = \widehat{\Lambda }^\prime (0) =  0 \quad \text{$\mathbb P$-almost surely} 
\]
and 
\[ 
\mathbb E_{\mathbb P}\left[ \widehat{ \Lambda } ^{\prime \prime} (0) \right] = V 
\] 
 (recall \eqref{eq:V} for the definition of $V$).
\item[$\mathrm{(3)}$] $\mathscr F_{\theta ,\omega}$ is spanned by $h_{\theta , \omega}$ for every $\theta \in \mathbb B$ and 
 $\mathbb P$-almost every $\omega \in \Omega$, and 
\[
\mathbb E_{\mathbb P} \left[ \widehat{\Lambda } (\theta ) 
 \right] = \Lambda _\theta .
\]
\item[$\mathrm{(4)}$]  
 For every $\theta \in \mathbb B$ and 
 $\mathbb P$-almost every $\omega \in \Omega$, 
\[
\mathcal L_{\theta , \omega } ^* w _{\theta , \sigma \omega } = \lambda _{\theta , \omega } w _{\theta ,  \omega },
\quad \text{ and}
 \quad 
w_{\theta , \omega} (h_{\theta ,\omega } )=1.
\]
\end{enumerate}
We used the notation 
$h_{\theta , \omega }, w_{\theta , \omega }, \lambda _{\theta ,\omega }$ for $h(\theta )(\omega ) , w(\theta )(\omega ) , \lambda (\theta )(\omega )$, respectively.
\end{thm}

The proof of the item (1) and the former statement of the item (2) in Theorem \ref{thm:1204} will be given in Subsection \ref{subsection:A2}.
Subsections \ref{subsection:A3}, \ref{subsection:A4} and \ref{subsection:A5} are, respectively,  dedicated for the proof of the later statement of the items (2), (3) and (4).

\begin{remark}
Note  that $\widehat \Lambda (\theta )(\omega )$ in Theorem \ref{thm:1204} corresponds to $Z(\theta , \omega )$ in \cite{DFGV2018} (see the proof of Lemma 3.9 of \cite{DFGV2018}), and $\mathbb E_{\mathbb P}[\widehat \Lambda (\theta ) ]$ in our context is written as  $\widehat \Lambda (\theta )$ in that of \cite{DFGV2018}. 
\end{remark}

Before starting the proof of Theorem \ref{thm:1204}, 
we  give necessary definitions as well as a brief explanation for the strategy of the proof of Theorem \ref{thm:1204}.
  Define $\mathbf I : L^\infty (\Omega , \mathscr B)\to L^\infty (\Omega , \mathbb C)$ and $\mathbf L_\theta :  L^\infty (\Omega , \mathscr B) \to L^\infty (\Omega , \mathscr B)$ by
\[
\mathbf I(\varphi ) (\omega ) = \langle \varphi _\omega , 1_M \rangle, \quad 
 \mathbf L _\theta  (\varphi ) (\omega  ) = \mathcal L_{\theta ,\sigma ^{-1} \omega }  \varphi _{\sigma ^{-1} \omega }  
\]
for   $\varphi \in L^\infty (\Omega , \mathscr B)$ and $\omega \in \Omega$. 
$\mathbf L_\theta$ is well-defined for any $\theta \in \C$ by virtue of  \eqref{eq:A2theta}.
 We simply write $\mathbf L$ for $\mathbf L_{0}$. 
Note that $(\mathbf I \circ \mathbf L(h_0) )_\omega =1$ for $\mathbb P$-almost every $\omega \in \Omega$. 
We will see that $\theta \mapsto \mathbf I \circ \mathbf L_\theta (h_0)$ is a continuous map 
 from $\mathbb C$ to $L^\infty (\Omega , \mathbb C)$    in the proof of  Lemma \ref{lem:2.4}.   
Hence, 
 one can find a small neighborhood $\mathbb B$ of $0$ in $\C$ such that 
$\mathbf I \circ \mathbf L_\theta (\varphi ) \neq 0$ $\mathbb P$-almost surely 
and  $\mathbf F(\theta ,\varphi ) :\Omega \to \mathscr B$ given by
\begin{equation}\label{eq:1205a}
 \mathbf F(\theta ,  \varphi ) =
 \frac{ \mathbf L_\theta (  \varphi )}{\mathbf I \circ \mathbf L_\theta ( \varphi ) }  -   \varphi   
\end{equation}
is well-defined for any $(\theta , \varphi )\in \mathbb B \times (h_0+ \mathrm{Ker} (\mathbf I ) )$.
Notice that 
$\mathbf F(\theta ,\varphi )\in \mathrm{Ker} (\mathbf I)$ 
for each $(\theta , \varphi )\in \mathbb B \times (h_0+ \mathrm{Ker} (\mathbf I ) )$.
It also holds that 
\begin{equation}
\label{eq:1205c}
\mathbf F( \theta , h _\theta ) =0 \quad \text{implies} \quad 
\mathcal L_{\theta ,\omega } h _{\theta , \omega } = \lambda _{\theta , \omega } h _{\theta , \sigma \omega } \; \text{$\mathbb P$-almost surely},
\end{equation}
 where $\lambda _{\theta ,\omega } = \langle\mathcal L_{\theta ,\omega }h _{\theta ,\omega}  , 1_M \rangle $.
Therefore, in Subsections \ref{subsection:A2} and   \ref{subsection:A3},  we will
apply the implicit function theorem to a map $\widetilde{\mathbf F }: \mathbb B \times  \mathrm{Ker} (\mathbf I) \to   \mathrm{Ker} (\mathbf I) $ given by
\begin{equation}\label{eq:1205b}
\widetilde{\mathbf F }(\theta , \varphi ) = \mathbf F(\theta , h_0 + \varphi ) ,  \quad (\theta , \varphi ) \in \mathbb B \times  \mathrm{Ker} (\mathbf I)  ,
\end{equation}
in order to prove the regularities of $h: \mathbb B\to L^\infty (\Omega , \mathscr B )$ and $\lambda : \mathbb B\to L^\infty (\Omega , \mathbb C )$ in Theorem \ref{thm:1204}. 
Furthermore,  in Subsection \ref{subsection:A4},  we will apply the implicit function theorem to another map $\mathbf G$ induced by the adjoint twisted operators $\mathcal L_{\theta , \omega }^*$  to obtain the regularity of $w: \mathbb B\to L^\infty (\Omega , \mathscr B^\prime )$ in Theorem \ref{thm:1204}.

\subsection{First order regularity}\label{subsection:A2}
In this subsection, we show the existence of $h : \mathbb B\to L^\infty (\Omega , \mathscr B)$ satisfying the item (1) of Theorem \ref{thm:1204}.
We first consider Fr\'echet derivatives of $\mathbf M: \mathbb C\times  L^\infty (\Omega , \mathscr B) \to L^\infty (\Omega , \mathscr B) : \mathbf (\theta , \varphi ) \mapsto  \mathbf L_\theta (\varphi )$ (both with respect to $\theta$ and $\varphi$) and $ \mathbf I$. 
For any $\theta \in \mathbb C$ and $\varphi \in L^\infty (\Omega , \mathscr B) $, since $\mathbf M(\theta ,\cdot ) = \mathbf L_\theta : L^\infty (\Omega , \mathscr B) \to L^\infty (\Omega , \mathscr B) $ is a linear operator,  
\begin{equation}\label{eq:p2M}
\text{the  Fr\'echet derivative  of $\mathbf M(\theta ,\cdot )$ at $\varphi $,  denoted by $\partial _2\mathbf M (\theta ,\varphi ) $, is $\mathbf L_\theta$.}
\end{equation}
Similarly, 
\begin{equation}\label{eq:pI} 
\text{the Fr\'echet derivative $D\mathbf I (\varphi )$ of $\mathbf I$ at $\varphi$ is $\mathbf I$}
\end{equation} 
for  any $\varphi \in L^\infty (\Omega , \mathscr B)$ because $\mathbf I$ is linear.

Given $(\theta , \varphi )\in \mathbb C\times L^\infty (\Omega , \mathscr B)$, we denote by $\partial _1 \mathbf M(\theta , \varphi )$ the  Fr\'echet derivative 
of $\mathbf M (\cdot , \varphi ) $
at $\theta \in \mathbb C$.
Define $\mathbf L _{j, \theta } :L^\infty (\Omega , \mathscr B)\to L^\infty (\Omega , \mathscr B)$ with   $j=1,2$ by 
$
\left( \mathbf L_{j, \theta } \varphi \right) _  \omega  =  \mathcal L  _{j, \theta ,\sigma ^{-1}\omega } \left(\varphi _{\sigma ^{-1} \omega } \right)$
for $\varphi \in L^\infty (\Omega , \mathscr B)$ and $\omega \in \Omega$, 
where $\mathcal L _{j, \theta , \omega } :\mathscr B\to \mathscr B$ is given by the duality 
\begin{equation}\label{eq:1205d}
\left\langle \mathcal L _{j,\theta ,\omega}  \varphi ,  u \right \rangle = \left\langle \varphi ,  (g_{ \omega } )^j  e^{\theta g_{ \omega }} \cdot u\circ T_\omega  \right\rangle.
\end{equation}
 This is well-defined 
  due to  (A2).  
 Note that $\mathcal L_{j,\theta , \omega } \varphi  = \mathcal L_{\omega } ( (g_\omega )^j e^{\theta g_\omega } \varphi )$.
\begin{lem}\label{lem:2.3old}
For each $\theta \in \mathbb C$ and $\varphi  \in L^\infty (\Omega , \mathscr B)$,  
\begin{equation}\label{eq:p1M}
\partial _1 \mathbf M(\theta , \varphi )
= \mathbf L_{1, \theta } (\varphi ) 
\end{equation}
under the identification in Remark \ref{rmk:1204}.
 \end{lem}
 \begin{proof}
 Let $\widetilde g_{t, \omega } =  (e^{t g_\omega } -1 )/ t  -g_\omega  $ with $t\in \C \setminus \{0\}$ and $\widetilde g_{0,\omega } =0$ for each $\omega \in \Omega$. 
 We first show that there is a positive constant $c_t$ converging to $0$ as $t\to 0$ such that
 \begin{equation}\label{eq:G2}
 \displaystyle  \esssup _{\omega } \left\Vert \widetilde g_{t, \omega }  \varphi \right\Vert \leq c_t \Vert \varphi \Vert \quad \text{for $\varphi \in \mathscr B$}.
 \end{equation}
By Taylor expansion of $e^{t g_\omega } $,
\[
\widetilde g_{t, \omega }  =\frac{1}{t} \left(\sum _{k=0}^\infty \frac{(tg_\omega )^k }{k!} -1 \right)  -g_{t, \omega } = t \sum _{k=0}^\infty \frac{t^kg_\omega ^{k+2}}{(k+2)!} .
\]
Therefore, it follows from  (A2)   that
 \[
\Vert \widetilde g_{t, \omega } \varphi \Vert \leq C_1 \vert  t\vert  G ^2 \Vert \varphi \Vert \sum _{k=0}^\infty \frac{\vert t \vert ^kG ^k }{k!} =C_1 \vert  t\vert  G^2 e^{tG }\Vert \varphi \Vert  
 \]
(recall that $G= \esssup _{\omega} \Vert g_\omega \Vert _{\mathscr E}$), which concludes \eqref{eq:G2}.

For each $\theta \in \mathbb C$ and $\varphi  \in L^\infty (\Omega , \mathscr B)$,  it follows from  (A1)   that
\[
\frac{\left\Vert \left(\mathbf M (\theta +t ,\varphi ) -\mathbf M(\theta , \varphi ) -t\mathbf L_{1,\theta} (\varphi )\right)_{\sigma \omega } \right\Vert }{\vert t \vert} \leq C_0 \left\Vert \widetilde g_{t, \omega } e^{\theta g_\omega } \varphi _\omega  \right\Vert .
\]
Hence, by   (A2)    and \eqref{eq:G2} we get that
\begin{equation}\label{eq:0212b5}
\frac{\Vert \mathbf M (\theta +t ,\varphi ) -\mathbf M(\theta , \varphi ) - t \mathbf L_{1, \theta } (\varphi )  \Vert _{L^\infty (\Omega , \mathscr B)}}{\vert t \vert } \\
\leq 
 C_0 C_1 e^{\vert \theta \vert G } c_t \Vert \varphi \Vert _{L^\infty (\Omega , \mathscr B) }  ,
\end{equation}
 which converges to $0$ as $t\to 0$. 
 This completes the proof.
 \end{proof}
 We  let $\widetilde{\mathbf L} _{n,\theta }( \varphi )= \mathbf L_{n, \theta }( h_0 +\varphi )$ and $\widetilde{\mathbf M} (\theta , \varphi )= \mathbf M(\theta , h_0 +\varphi )$, and define Fr\'echet derivatives  $\partial _j \widetilde{\mathbf F }(\theta, \varphi )$ and $\partial _j \widetilde{\mathbf M }(\theta, \varphi )$ ($j=1, 2$) in a similar manner to the definition of $\partial _j\mathbf M(\theta ,\varphi )$.
By \eqref{eq:p1M}, \eqref{eq:p2M} and the chain rule for Fr\'echet derivatives,
it is obvious that 
\begin{equation}\label{eq:tildeM}
\partial _1 \widetilde{\mathbf M } (\theta , \varphi ) = \widetilde{\mathbf L}_{1, \theta } (\varphi ) \quad \text{and} \quad \partial _2 \widetilde{\mathbf M } (\theta , \varphi ) = \mathbf L_\theta .
\end{equation} 

\begin{lem}\label{lem:2.4}  
For each $\theta \in \mathbb B$ and $\varphi \in  \mathrm{Ker} (\mathbf I)$,
\[
\partial _1  \widetilde{\mathbf F}(\theta , \varphi )  = \frac{  \mathbf I \circ \widetilde{ \mathbf L}_\theta (\varphi ) \cdot \widetilde{\mathbf L }_{1, \theta } (\varphi ) - \widetilde{\mathbf L}_\theta (\varphi ) \cdot  \mathbf I \circ  \widetilde{\mathbf L}_{1, \theta } (\varphi ) }{( \mathbf I \circ \widetilde{\mathbf L} _\theta (\varphi ) )^2}  
\]
and 
\[
\partial _2  \widetilde{\mathbf F}(\theta , \varphi ) =   \frac{ \mathbf I \circ \widetilde{\mathbf L}_\theta (\varphi )  \cdot \mathbf L_\theta - \widetilde{\mathbf L}_\theta ( \varphi ) \cdot  \mathbf I \circ \mathbf L_\theta }{( \mathbf I \circ \widetilde{\mathbf L} _\theta (\varphi ) )^2} - \mathbf{id} ,
\]
under the identification in Remark \ref{rmk:1204}. 
In particular, 
for each $t\in \mathbb B$ and $\psi \in  \mathrm{Ker} (\mathbf I)$, 
\[
\partial _1  \widetilde{\mathbf F} ( 0 , 0 )(t)= t \mathbf L _{1, 0 } (h_0 )   \quad  \text{$\mathbb P$-almost surely} 
\]
and
\[
 \partial _2  \widetilde{\mathbf F} ( 0 , 0 )(\psi )
 =   \mathbf L(\psi )  - \psi .
\]
 \end{lem}
 \begin{proof}
  By the chain rule for Fr\'echet derivatives, 
  the Fr\'echet derivatives $D  \left[\mathbf I\circ \widetilde{\mathbf L}_\theta \right] (\varphi )$ of $\mathbf I\circ \widetilde{ \mathbf L}_\theta = \mathbf I\circ \widetilde{ \mathbf M}(\theta ,\cdot ) $ at $\varphi$ is calculated as
  \begin{equation}\label{eq:1204a}
D  \left[\mathbf I\circ \widetilde{\mathbf L}_\theta \right] (\varphi ) =  D \mathbf I (\widetilde{\mathbf L}_\theta (\varphi )) \circ \partial _2 \widetilde{\mathbf M} (\theta ,\varphi ) = \mathbf I \circ \mathbf L_\theta .
  \end{equation}
In the last equality, we used \eqref{eq:p2M}, \eqref{eq:pI} and \eqref{eq:tildeM}.
Similarly, for each $\varphi \in  \mathrm{Ker} (\mathbf I)$, the Fr\'echet derivatives $D\left[ \mathbf I\circ \widetilde{\mathbf L}_{(\cdot )}  (\varphi ) \right] (\theta )$ of 
$ \mathbf I\circ \widetilde{\mathbf L}_{(\cdot )}  (\varphi ) = \mathbf I \circ \widetilde{\mathbf M}(\cdot , \varphi )$
at $\theta \in \mathbb B$ is
    \begin{equation}\label{eq:1204b}
D\left[ \mathbf I\circ\widetilde{ \mathbf L}_{(\cdot )}  (\varphi ) \right] (\theta ) =  D \mathbf I (\widetilde{\mathbf L}_\theta (\varphi )) \circ \partial _1\widetilde{ \mathbf M } (\theta , \varphi )  = \mathbf I \circ \widetilde{\mathbf L}_{1, \theta }(\varphi ).
  \end{equation}
Here the right-hand side means a map $\mathbb B\ni t\mapsto \mathbf I (t \widetilde{\mathbf L}_{1, \theta }(\varphi )) $.
  We used  \eqref{eq:pI}, \eqref{eq:p1M} and \eqref{eq:tildeM}.
  
  By quotient rule for Fr\'echet derivatives, 
 \[
\partial _2 \widetilde{\mathbf F}(\theta , \varphi ) = \frac{\mathbf I\circ \widetilde{\mathbf L}_\theta (\varphi ) \cdot D  \widetilde{\mathbf L}_\theta  (\varphi )  -\widetilde{ \mathbf L}_\theta  (\varphi ) \cdot D  \left[\mathbf I\circ \widetilde{\mathbf L}_\theta \right] (\varphi )  }{\left(\mathbf I\circ \widetilde{\mathbf L}_\theta  (\varphi ) \right)^2 } -D\mathbf{id} (\varphi ) .
\]
This implies the second claim of this lemma by \eqref{eq:p2M} and \eqref{eq:1204a}.
  The first claim also can be proven in a similar manner, 
   together with \eqref{eq:p1M} and \eqref{eq:1204b}.

    On the other hand, 
  $\widetilde{ \mathbf L} _0 (0) = \mathbf L(h_0) = h_0$, and thus 
    $\mathbf I\circ \widetilde{\mathbf L}_0 (0) 
    =\mathbf I( h_0  ) =1$.
    Similarly, $\mathbf I\circ \mathbf L(  \psi )=\mathbf I(  \psi )=0$ for each $\psi \in \mathrm{Ker} (\mathbf I)$.
Moreover, by the assumption on $g$, we have that 
\begin{equation}\label{eq:0212}
\mathbf I\circ \widetilde{ \mathbf L} _{1, 0} (0) (\sigma \omega ) = \langle \mathcal L_\omega (g_\omega h_{0,\omega } ), 1_M \rangle = \mathbb E_{\mu _\omega}[g_\omega ] =0  \quad \text{$\mathbb P$-almost surely} .
\end{equation}
 The last claim of Lemma \ref{lem:2.4} immediately follows from these observations.
\end{proof}

As a final preparation to apply the implicit function theorem to $\widetilde{\mathbf F}$ around $(\theta ,\varphi )=(0,0)$, we check that $\partial _2\widetilde{\mathbf F} (0, 0): \mathrm{Ker} (\mathbf I) \to \mathrm{Ker} (\mathbf I)$ is bijective. 
We first show  injectivity.
Arguing by contradiction, we assume that  
\[
\partial _2\widetilde{\mathbf F} (0, 0) (\psi _1) =\partial _2\widetilde{\mathbf F} (0, 0) (\psi _2)
\]
 for some $\psi _1, \psi _2 \in \mathrm{Ker} (\mathbf I)$ with $\psi _1\neq \psi _2$. 
 Let $\psi =\psi _1-\psi _2$. 
Then, $\psi$ is a nonzero mapping in $\mathrm{Ker} (\mathbf I)$, and it follows from Lemma \ref{lem:2.4} that $ \mathbf L (\psi )  =\psi$. That is, $\psi _\omega \in \mathscr F_{0,\omega }$ $\mathbb P$-almost surely.
On the other hand,
 $\psi _\omega \not\in \C h_{0, \omega }$ 
  because $\langle \psi _\omega , 1_M \rangle =0$ while $\langle h _{0,\omega} , 1_M \rangle =1$. 
This contradicts to (UG), and we get that $\partial _2\widetilde{\mathbf F} (0, 0) (\psi ) $ is injective.

Next we show surjectivity. We note that  a linear opeartor $\mathbf N: \mathrm{Ker} (\mathbf I) \to \mathrm{Ker } (\mathbf I)$ given by
\[
\mathbf N \varphi = -\sum _{n=0}^\infty \mathbf L ^n \varphi \quad \text{for $\varphi \in \mathrm{Ker } $}
\]
is well-defined due to the assumption (UG) and the fact that $\mathbf I \circ \mathbf L = \mathbf I$ on $L^\infty (\Omega , \mathscr B)$. 
Therefore, since it is easy to see that 
\begin{equation}\label{eq:0207f2}
\left(\partial _2\widetilde{\mathbf F} (0, 0) \right)^{-1 } = \mathbf N \quad \text{ on $\mathrm{Ker} (\mathbf I)$}
\end{equation}
 by Lemma \ref{lem:2.4}, we obtain the surjectivity.

Note that $\widetilde{\mathbf F}(0,0)=0$.
It follows from the implicit function theorem for Banach spaces 
together with the estimates in this subsection that
there is a small ball  $\mathbb B$ in $\mathbb C $ centered at $0$ and a $\mathscr C^1$ map $\eta : \mathbb B\to \mathrm{Ker} (\mathbf I)$ such that $\widetilde{\mathbf F}(\theta , \eta (\theta )) =0$.
Hence, by the observation in \eqref{eq:1205c}, if we let
\begin{equation}\label{eq:lambda}
h (\theta ) = h_0 + \eta (\theta )  \quad \text{and} \quad \lambda (\theta ) (\omega ) = \langle \mathbf M(\theta ,  h (\theta ) )  (\sigma \omega ),1_M\rangle 
\end{equation}
for $ \theta \in \mathbb B$ and  $\omega \in \Omega$,
then $h: \mathbb B\to L^\infty (\Omega , \mathscr B)$ and $\lambda : \mathbb B\to L^\infty (\Omega , \mathbb C)$ are $\mathscr C^1$ maps and satisfy the first equality in the item (1)  of Theorem \ref{thm:1204}, and it $\mathbb P$-almost surely holds that 
\begin{equation}\label{eq:0207e}
h(0) =h_0 \quad \text{and} \quad \lambda (0) (\omega ) = \langle \mathcal L_{0,\omega } h_{0,\omega } , 1_M\rangle  =   \langle h_{0,\omega } , 1_M\rangle  = 1 
\end{equation}
by (UG). 
This completes the proof of the item (1) of Theorem \ref{thm:1204}.

Next we show the first assertion of the item (2) of Theorem \ref{thm:1204}. 
By applying the implicit function theorem for Banach spaces to $\widetilde{\mathbf F}$, we have
\begin{align}\label{eq:1205e}
h^\prime (0 ) (t) (\omega ) = \eta ^\prime (0) (t) (\omega ) &= - \left(\partial _2\widetilde{\mathbf F} (0, 0) \right )^{-1}\left( \partial _1\widetilde{\mathbf F} (0, 0) (t) \right) (\omega ) \\
&= -t \mathbf N \circ  \mathbf L_{1,0} (h_0) (\omega ) \notag\\
&= t\sum _{n=1}^\infty \mathcal L^{(n)}_{\sigma ^{-n} \omega} (h_{0, \sigma ^{-n}\omega} g_{\sigma ^{-n} \omega })\notag
\notag
\end{align} 
for each $t\in \mathbb B$ and $\mathbb P$-almost every $\omega \in \Omega$ by Lemma \ref{lem:2.4} and \eqref{eq:0207f2}. 
Since $\mathbf N$ is a bounded operator on $ \mathrm{Ker} (\mathbf I)$ and  $\mathbf L_{1,0} (h_0) \in \mathrm{Ker} (\mathbf I)$ by the assumption on $g$, we have
\begin{equation}\label{eq:0207f3}
h^\prime (0) (t) \in \mathrm{Ker} (\mathbf I) \quad \text{for all $t\in \mathbb B$}.
\end{equation}

Using these estimates, we can show the following.
\begin{lem}\label{lem:0207}
For each $\theta ,t\in \mathbb B$ and $\mathbb P$-almost every $\omega \in \Omega$, 
\begin{align}\label{eq:0207e3}
\lambda ^\prime (\theta ) (t) (\omega ) 
= \langle  t \mathbf L _{1,\theta  } (h (\theta )) (\sigma \omega ) ,1_M\rangle 
+  \langle \mathbf L _{\theta  } (h ^\prime (\theta ) (t) ) (\sigma \omega ) ,1_M\rangle  .
\end{align}
 In particular, we have
\begin{equation}\label{eq:0207e2}
\lambda ^\prime (0) (t) (\omega ) 
=0.
\end{equation}
\end{lem}
\begin{proof}
For each $\theta \in \mathbb B$ and $\varphi \in \mathrm{Ker} (\mathbf I)$, we define $\lambda (\theta , \varphi ) \in L^\infty (\Omega , \mathbb C )$ by
\[
\lambda (\theta , \varphi ) (\omega ) =\langle \mathbf M(\theta , \varphi )  (\sigma \omega ),1_M\rangle \quad (\omega \in \Omega ).
\]
Then, for any $\theta , t\in \mathbb B$, $\varphi \in \mathrm{Ker} (\mathbf I)$ and $\mathbb P$-almost every $\omega \in \Omega$, it follows from the assumption (A2) and  $1_M\in \mathscr E$ (given in Subsection \ref{section:abstract}) that 
\begin{multline*}
 \vert \lambda (\theta+t , \varphi ) (\omega ) -\lambda (\theta, \varphi ) (\omega ) -\langle t \mathbf L_{1,\theta } ( \varphi )  (\sigma \omega ),1_M\rangle \vert \\
 \leq  C _M \Vert \mathbf M(\theta +t , \varphi )  (\sigma \omega ) - \mathbf M(\theta , \varphi )  (\sigma \omega ) - t \mathbf L_{1,\theta } ( \varphi )  (\sigma \omega ) \Vert 
\end{multline*}
with $C _M= \Vert 1_M\Vert _{\mathscr E} >0$. 
Therefore, due to  \eqref{eq:p1M}, we get 
\begin{equation}\label{eq:0207c}
\partial _1 \lambda (\theta , \varphi ) (t) (\omega )= \langle t \mathbf L_{1,\theta } ( \varphi )  (\sigma \omega ),1_M\rangle  .
\end{equation}
Similarly, by virtue of  \eqref{eq:p2M} and the assumption (A2) and  $1_M\in \mathscr E$, we can calculate $\partial _2 \lambda (\theta , \varphi )$ as
\begin{equation}\label{eq:0207d}
\partial _2 \lambda (\theta , \varphi ) (\psi ) (\omega )= \langle  \mathbf L_{\theta } ( \psi )  (\sigma \omega ),1_M\rangle  .
\end{equation}

The first assertion of Lemma \ref{lem:0207} immediately follows from \eqref{eq:0207c} and \eqref{eq:0207d} together with
the chain rule for Fr\'echet derivatives.
The second assertion also immediately follows from the assumption on $g$, 
\eqref{eq:1205d} and \eqref{eq:0207f3}.
\end{proof}

Let $\widehat{\Lambda }  (\theta ) (\omega ) = \log \vert \lambda (\theta )(\omega )\vert $ for each $\theta \in \mathbb C$ and $\omega \in \Omega$. Then, since $\widehat{\Lambda }  (\theta )= \frac{1}{2} \log ( \lambda (\theta ) \overline{\lambda (\theta )})$ one can check that (by the chain rule for Fr\'echet derivatives)  
\begin{align}\label{eq:0207e6}
\widehat{\Lambda } ^\prime (\theta ) = \frac{\lambda (\theta ) \overline{\lambda ^\prime (\theta )} + \lambda ^\prime (\theta ) \overline{\lambda (\theta )} }{2\vert \lambda (\theta ) \vert ^2 } 
 =\Re \left(\frac{  \lambda ^\prime (\theta ) \overline{\lambda (\theta )} }{\vert \lambda (\theta ) \vert ^2 }\right) 
\end{align}
under the identification in Remark \ref{rmk:1204}, 
where $\Re (z) = (z + \overline z ) / 2$ is the real part of a complex number $z$. 
Therefore, by virtue of \eqref{eq:0207e} and  \eqref{eq:0207e2} we have that 
\[
\widehat{\Lambda } (0)=0 \quad \text{ and} \quad  \widehat{\Lambda } ^\prime (0) =0 ,
\]
which completes the proof of the first assertion in the item (2) of Theorem \ref{thm:1204}.

\subsection{Second order regularity}\label{subsection:A3} 
For  $(\theta , \varphi )\in \mathbb C\times L^\infty (\Omega , \mathscr B)$, we denote by $\partial _1 ^2\widetilde{\mathbf M}(\theta , \varphi )$ the  Fr\'echet derivative 
of $\partial _1 \widetilde{\mathbf M }(\cdot , \varphi ) : \mathbb C \to L(\mathbb C, \mathrm{Ker} (\mathbf I))$ 
at $\theta \in \mathbb C$.
Note that $\partial _1 ^2\widetilde{\mathbf M}(\cdot , \varphi )$ is a map from $\mathbb C$ to $L(\mathbb C, L(\mathbb C, \mathrm{Ker} (\mathbf I) ))$.
In a similar manner, we define  $\partial _2 \partial _1\widetilde{\mathbf M}(\theta , \varphi )$,  $\partial _1 \partial _2\widetilde{\mathbf M}(\theta , \varphi )$ and  $\partial _2 ^2\widetilde{\mathbf M}(\theta , \varphi )$.
\begin{lem}\label{lem:2.3}
For each $\theta \in \mathbb C$ and $\varphi  \in L^\infty (\Omega , \mathscr B)$,  
\begin{equation}\label{eq:p1M1204}
\partial _1 ^2\widetilde{\mathbf M}(\theta , \varphi ) (s)(t) =
 ts \mathbf L _{2, \theta } (\varphi ) \quad (s, t\in \mathbb C) ,
\end{equation}
\begin{equation}\label{eq:p1M1204b}
\partial _2 \partial _1 \widetilde{\mathbf M}(\theta , \varphi ) (\psi ) (t) =
 t \mathbf L _{1, \theta } (\psi ) \quad ( \psi \in \mathrm{Ker} (\mathbf I), \; t\in \mathbb C) .
\end{equation}
Furthermore, for each $\theta \in \mathbb C$ and $\varphi  \in L^\infty (\Omega , \mathscr B)$,  
\begin{equation*}
\partial _1 \partial _2 \widetilde{\mathbf M}(\theta , \varphi ) (t) (\psi )=
 t \mathbf L _{1, \theta } (\psi ) \quad (t\in \mathbb C, \;  \psi  \in \mathrm{Ker} (\mathbf I)) ,
\end{equation*}
\begin{equation*}
\partial _2 ^2 \widetilde{\mathbf M}(\theta , \varphi )  (\phi ) (\psi )=0 \quad (\phi ,  \psi  \in \mathrm{Ker} (\mathbf I)).
\end{equation*}
 \end{lem}
 \begin{proof}
 One can prove \eqref{eq:p1M1204} (resp.~\eqref{eq:p1M1204b}) from the formula \eqref{eq:p1M}  as the proof of \eqref{eq:p1M} (resp.~\eqref{eq:p2M}),
 see also \cite[Lemma B.11]{DFGV2018}.
The proof of the later equalities also follows from the formula \eqref{eq:p2M} by a similar argument (cf.~\cite[Lemma B.10]{DFGV2018}). 
 \end{proof}

It follows from \eqref{eq:pI} that $\mathbf I$ is of class $\mathscr C^2$, and so is $\widetilde{\mathbf F}: \mathbb B\times \mathrm{Ker} (\mathbf I)\to \mathrm{Ker} (\mathbf I)$ due to the form of $\widetilde{\mathbf F}$ given in \eqref{eq:1205a}, \eqref{eq:1205b}.
Hence, by the implicit function theorem, $h : \mathbb B$ given in \eqref{eq:lambda} is of class $\mathscr C^2$.

Furthermore, 
in a manner similar to one in the proof of Lemma \ref{lem:0207}, 
it follows from  \eqref{eq:0207e3} and Lemma \ref{lem:2.3} together with assumption (A2) and $1_M\in \mathscr E$,
we have
\begin{multline*}
\lambda ^{\prime \prime} (\theta ) (t) (s) (\omega ) =
 \langle ts \mathbf L_{2, \theta } (h (\theta )) (\sigma \omega ) , 1_M \rangle 
+
\langle t \mathbf L_{1, \theta } (h ^\prime (\theta )(s)) (\sigma \omega ) , 1_M \rangle \\
+
\langle s \mathbf L_{1, \theta } (h ^\prime (\theta )(t)) (\sigma \omega ) , 1_M \rangle 
+
\langle  \mathbf L_{ \theta } (h ^{\prime \prime} (\theta )(t)(s)) (\sigma \omega ) , 1_M \rangle 
\end{multline*}
 all $\theta , t, s \in \mathbb C$ and $\mathbb P$-almost every $\omega \in \Omega$.
On the other hand, since  $h $ is a $\mathscr C^2$ map from $\mathbb B$ to $\mathrm{Ker} (\mathbf I)$, 
$ h ^{\prime \prime} (0 )(t)(s) \in \mathrm{Ker} (\mathbf I)$ for all $t, s\in \mathbb B$ and
\[
\langle  \mathbf L_{ \theta } (h ^{\prime \prime} (\theta )(t)(s)) (\sigma \omega ) , 1_M \rangle =\langle  h ^{\prime \prime} (0 )(t)(s) ( \omega ) , 1_M \rangle =0.
\]
So, by \eqref{eq:1205d} and \eqref{eq:1205e},
 \begin{align}\label{eq:0211b}
 \lambda ^{\prime \prime} (0 )  (t) (s)(\omega )
&=
ts \langle  h _{0, \omega }, (g_{ \omega }) ^2 \cdot 1_M \circ T_\omega \rangle
 +
t\langle  h ^{\prime } (\theta )(s) ( \omega ) , g_\omega \cdot 1_M  \circ T_\omega \rangle \\
\notag & \quad \quad \quad \quad +
s\langle  h ^{\prime } (\theta )(t) ( \omega ) , g_\omega \cdot 1_M  \circ T_\omega \rangle \\
\notag &=
ts \mathbb E_{\mu _\omega }\left[ g_{ \omega } ^2  \right] + 2ts \left\langle \sum _{n=1}^\infty \mathcal L^{(n)}_{\sigma ^{-n} \omega} (h_{0, \sigma ^{-n}\omega} g_{\sigma ^{-n} \omega }) , g_\omega \right\rangle \\
\notag &=
ts \mathbb E_{\mu _\omega }\left[ g_{ \omega } ^2  \right] + 2ts \sum _{n=1}^\infty \mathbb E_{\mu _{\sigma ^{-n} \omega}} \left[ g_{\sigma ^{-n} \omega } \cdot g_\omega \circ T_{\sigma ^{-n}\omega } ^{(n)} \right] .
 \end{align}
 Thus it follows from the $\sigma$-invariance of $\mathbb P$ that 
$ \mathbb E_{\mathbb P}\left[ \lambda ^{\prime \prime} (0 ) \right] = 
   V$
 under the identification of Remark \ref{rmk:1204} (recall \eqref{eq:V}).

On the other hand, it follows from \eqref{eq:0207e6} that (by chain and quotient rules for Fr\'echet derivatives)
\[
\widehat{\Lambda } ^{\prime\prime} (\theta ) 
 =\Re \left(\frac{ \lambda ^{\prime \prime} (\theta )}{\lambda (\theta ) } - \left(\frac{\lambda ^\prime (\theta )}{\lambda (\theta )} \right)^2 \right) .
\]
Since $g_\omega$ is a real-valued function,  \eqref{eq:0207e}, \eqref{eq:0207e2} and \eqref{eq:0211b} leads to that 
\[
 \widehat{\Lambda } ^{\prime\prime} (0 ) = \Re \left( \lambda ^{\prime \prime} (0 )  \right) =
 \lambda ^{\prime \prime} (0 ) \quad \text{$\mathbb P$-almost surely}.
\]
These observations immediately imply the later assertion of the item (3) of Theorem \ref{thm:1204}.

\subsection{Relation between $\Lambda _\theta$ and $\widehat{\Lambda} (\theta )$}\label{subsection:A4}
We first show 
\begin{equation}\label{eq:02011c}
\mathbb E_{\mathbb P} \left[ \widehat{\Lambda } (\theta ) \right] \leq \Lambda _\theta
\end{equation}
for every $\theta$ with sufficiently small absolute value.
It follows from the item (1) of Theorem \ref{thm:1204} and (UG) that 
 for every $\theta$ with sufficiently small absolute value and $\mathbb P$-almost every $\omega\in\Omega$, 
\[
 \Lambda _\theta \geq \lim _{n\to \infty} \frac{1}{n} \log \Vert \mathcal L_{\theta , \omega }^{(n)} h_{\theta , \omega }\Vert   \geq  \lim _{n\to \infty} \frac{1}{n}\sum _{j=0}^{n-1} \log \vert \lambda_{\theta , \sigma ^j \omega }\vert   
\]
and the claim immediately follows from the definition of $\widehat{\Lambda}$ and the ergodicity of $(\theta , \mathbb P)$.

Let $\mathscr B = \mathscr F_{\theta ,\omega }\oplus \mathscr R_{\theta ,\omega }$ be the closed splitting  for $\mathcal L_\theta$ given in Subsection \ref{subsection:A1}. 
Note that $\mathrm{dim } (\mathscr F_{0, \omega }) =1$ due to (UG).
On the other hand,  it immediately  follows from   \eqref{eq:A2theta} that  for each $\omega \in\Omega$,  
$\theta \mapsto \mathcal L_{\theta ,\omega} : \mathbb C\to L(\mathscr B)$ is continuous.
Hence, by the semi-continuity argument of Lyapunov exponents of operator cocycles restricted on finite dimensional spaces (cf.~\cite[Theorem 3.12]{DFGV2018}), 
we get that 
\begin{equation}\label{eq:0211d}
\mathrm{dim } (\mathscr F_{\theta , \omega }) =1
\end{equation}
 for every $\theta$ with sufficiently small absolute value and $\mathbb P$-almost every $\omega \in \Omega$.

Finally, the item (3) of Theorem \ref{thm:1204} immediately follows from \eqref{eq:02011c}, \eqref{eq:0211d} and the item (1) of Theorem \ref{thm:1204}.

\subsection{Regularity of adjoint cocycles}\label{subsection:A5}
Note that the map from $ \mathbb N \times \Omega \times \mathscr B^\prime$ to $\mathscr B^\prime$ given by 
\[
  (n, \omega ,u) \mapsto  \mathcal L_{\theta ,\sigma ^{-n} \omega}^* \circ \mathcal L_{\theta ,\sigma ^{-n+1} \omega}^* \circ \cdots  \circ  \mathcal L_{\theta ,\sigma ^{-1} \omega} ^*u 
\]
satisfies the cocycle property over the driving system $\sigma ^{-1} : \Omega \to \Omega$.
Hence, as  the cocycle $(\mathcal L_\theta , \sigma )$ in Subsection \ref{subsection:A1}, we  define $\mathbf J : L^\infty (\Omega , \mathscr B^\prime)\to L^\infty (\Omega , \mathbb C)$ and $\mathbf L_\theta^* :  L^\infty (\Omega , \mathscr B^\prime) \to L^\infty (\Omega , \mathscr B^\prime)$ by
\[
\mathbf J(u ) (\omega ) = u_\omega (h_{0, \omega } ), \quad 
 \mathbf L _\theta ^* (u) (\omega )  = \mathcal L_{\theta ,  \omega } ^* u _{\sigma  \omega }  
\]
for   $u \in L^\infty (\Omega , \mathscr B^\prime )$ and $\omega \in \Omega$. 
By (UG), it holds that
\[
\mathbf J \circ \mathbf L^*_0(u)  (\omega ) =\mathbf J(u) (\sigma \omega ) \quad \text{for $\mathbb P$-almost every $\omega \in \Omega$},
\] 
so $\mathbf L^*_0$ preserves $\mathrm{Ker} (\mathbf J)$.
Let $w_0\in  L^\infty (\Omega , \mathscr B^\prime )$ such that 
\begin{equation}\label{eq:0211e}
\mathbf L^*_0 (w_0 ) = w_0 \quad \text{and} \quad \mathbf J(w_0) =1 \quad \text{$\mathbb P$-almost surely}
\end{equation}
(its existence will be proven in Lemma \ref{lem:0211} below). 
In a similar manner to Subsection \ref{subsection:A2} (by using the argument in  the proof of  Lemma \ref{lem:2.4v2}  below), 
 one can find a small neighborhood $\mathbb B$ of $0$ in $\C$ such that 
 $\mathbf G(\theta , u ) :\Omega \to \mathscr B^\prime $ given by
\begin{equation*}
 \mathbf G(\theta ,  u ) =
 \frac{ \mathbf L_\theta ^* (  u)}{\mathbf J \circ \mathbf L_\theta ^*( u ) }  -   u 
\end{equation*}
is well-defined 
and in $\mathrm{Ker} (\mathbf J )$
 for any $(\theta , u )\in \mathbb B \times (w_0+ \mathrm{Ker} (\mathbf J ) )$.
Define $\widetilde{ \mathbf G} : \mathbb B \times  \mathrm{Ker} (\mathbf J ) \to \mathrm{Ker} (\mathbf J )$ by 
\[
\widetilde{ \mathbf G} (\theta , u ) = \mathbf G(\theta ,  w_0 + u ) \quad  \text{for $(\theta ,u) \in \mathbb B \times  \mathrm{Ker} (\mathbf J ) $} ,
\]
and its derivatives $\partial _1  \widetilde{\mathbf G} , \partial _2  \widetilde{\mathbf G}$ as in Subsection \ref{subsection:A2}.

As in Subsection \ref{subsection:A2}, we will apply the implicit function theorem to $ \widetilde{\mathbf G}$.
Necessary estimates for it is the following.
\begin{lem}\label{lem:0211}
There is a constant $C_2^\prime >0$ and $\rho \in (0,1)$ such that for any $u \in \mathrm{Ker} (\mathbf J)$ and $n\geq 0$, 
\begin{equation}\label{eq:0212c}
 \left\Vert (\mathbf L^* _0) ^{n} (u)  \right\Vert  _{L^\infty(\Omega , \mathscr B^\prime )}\leq C_2^\prime \rho ^n  \Vert u \Vert  _{L^\infty (\Omega , \mathscr B^\prime )} .
\end{equation}
Furthermore, there exists $w_0 \in L^\infty (\Omega , \mathscr B^\prime )$  satisfying \eqref{eq:0211e}.
\end{lem}
\begin{proof}
Let $\Pi _\omega :\mathscr B\to \mathrm{Ker} (\mathbf I)$ denote the projection onto  $\mathrm{Ker} (\mathbf I)$ 
along the subspace spanned by $h_{0, \omega }$.
By (UG), we can  take a positive number $K >0$ such that
\[
1 \leq \Vert h_{0, \omega } \Vert \leq K \quad \text{$\mathbb P$-almost surely} .
\]
Let  $n_1$  be an integer such that $K^{-1} -C_2\rho ^{n_1} >0$, and
 $c$  a positive number given by
\[
c = \frac{K^{-1} -C_2\rho ^{n_1} }{C_0^{n_1}} .
\]
We will show that
\begin{equation}\label{eq:0211f}
\esssup _{\omega \in \Omega} \Vert \Pi _\omega \Vert \leq 2c^{-1} .
\end{equation}

Note  that the invariant splitting $\mathscr B= \mathscr F_{0, \omega } \oplus \mathscr R_{0, \omega}$ with $\mathrm{dim} (\mathscr F_{0, \omega }) =1$
with respect to the cocycle $(\mathcal L_{0 }, \sigma )$  (given in Subsection \ref{subsection:A1}) is written as $\mathscr F_{0, \omega } = \mathbb Ch_{0, \omega} $ and  $\mathscr R_{0, \omega} = \{ \varphi \in \mathscr B \mid \langle \varphi , 1_M \rangle =0\}$.
Let $\gamma _\omega$ be a real number given by
\[
\gamma _\omega = \inf \left\{ \Vert a h_{0, \omega }+ \varphi  \Vert \; :  \; a \in \mathbb C , \; \varphi  \in \mathscr R_{0, \omega}, \; \Vert a h_{0, \omega } \Vert = \Vert \varphi  \Vert  =1\right\} .
\]
Then it is straightforward to see that $\Vert \Pi _\omega \Vert \leq 2\gamma _\omega ^{-1}$  (cf.~\cite[Lemma 1]{DF2018}). 
On the other hand, by (A1), for any 
 $a \in \mathbb C$ and $ \varphi  \in \mathscr R_{0, \omega}$,
\[
\Vert a h_{0, \omega }+ \varphi \Vert \geq \frac{ \Vert \mathcal L_{0,\omega } ^{(n_1)} (a h_{0, \omega }+ \varphi )\Vert }{C_0^{n_1}}
\geq \frac{ \Vert \mathcal L_{0,\omega } ^{(n_1)} (a h_{0, \omega } ) \Vert - \Vert \mathcal L_{0,\omega } ^{(n_1)}  \varphi  \Vert}{C_0^{n_1}} .
\]
So, if $\Vert a h_{0, \omega } \Vert = \Vert \varphi \Vert  =1$, then it follows from (UG) that
\[
\Vert a h_{0, \omega }+ \varphi \Vert 
\geq \frac{1}{C_0^{n_1}}  \left(\frac{\Vert h_{0, \sigma ^{n_1}\omega }  \Vert }{\Vert h_{0, \omega }\Vert}   - C_2 \rho ^{n_1} \right)\geq c ,
\]
and we get \eqref{eq:0211f}.

It follows from (UG) and \eqref{eq:0211f} that for any $u \in \mathrm{Ker} (\mathbf J)$,
\begin{align*}
\left\Vert (\mathbf L^* _0) ^{n} (u) (\omega ) \right\Vert _{\mathscr B^\prime } & = \sup _{\varphi \in \mathscr B, \; \Vert \varphi \Vert \leq 1} \left\vert u (\sigma ^n\omega ) (\mathcal L_{0,\omega } ^{(n)} \varphi ) \right\vert \\
& = \sup _{\varphi \in \mathscr B, \; \Vert \varphi \Vert \leq 1}\left \vert u (\sigma ^n\omega ) (\mathcal L_{0,\omega } ^{(n)}  \Pi _\omega \varphi ) \right\vert \\
&\leq C_2 \rho ^n \Vert \Pi _\omega \Vert  \Vert u  \Vert _{L^\infty (\Omega , \mathscr B^\prime) }  \\ 
& \leq 2c^{-1}C_2 \rho ^n \Vert u  \Vert _{L^\infty (\Omega , \mathscr B^\prime) } ,
\end{align*}
and the first assertion of Lemma \ref{lem:0211} is proven with $C_2^\prime =2c^{-1}C_2 $.

Let $\mathscr F^* _{0, \omega }$ be the subspace of  $\mathscr B^\prime$  given by 
\[
\mathscr F^* _{0, \omega } = \left\{ u\in  \mathscr B^\prime \mid u( \varphi )=0 \; \text{for all $\varphi \in \mathscr R_{0, \omega}$} \right\} .
\]
Then, 
$\mathrm{dim} (\mathscr F^* _{0, \omega } )=1$ 
due to \cite[Lemma 2.6]{DFGV2018}.
Let $w_{0, \omega }\in \mathscr F^* _{0, \omega }$  such that $\mathbf J(w_{0 } )
=1$. 
If we denote by $\widetilde \lambda _{\omega}$  the complex number such that $\mathcal L_{0,\omega }^* w_{0, \sigma \omega } = \widetilde \lambda _\omega w_{0, \omega}$, then
\begin{equation}\label{eq:0212d}
\widetilde \lambda _\omega = \widetilde \lambda _\omega w_{0,  \omega} (h_{0,  \omega }) =
 \mathcal L_{0, \omega }^* w_{0,  \sigma\omega} (h_{0,  \omega }) 
 =  w_{0,  \sigma \omega} ( \mathcal L_{0, \omega } h_{0,  \omega }) 
 =  w_{0,  \sigma \omega} (  h_{0,  \omega }) =1 .
\end{equation}
Furthermore,  by (UG), for each $a \in \mathbb C$ and $ \varphi  \in \mathscr R_{0, \omega}$,
\begin{equation}\label{eq:0212b}
w_{0, \omega } (a h_{0, \omega } + \varphi ) = a = \langle ah_{0, \omega } + \varphi , 1_M\rangle .
\end{equation}
Hence, due to (A1) we $\mathbb P$-almost surely have
\[
\Vert w_{0, \omega } \Vert _{\mathscr B^\prime } \leq \Vert 1_M \Vert _{\mathscr E}.
\]
By the assumption $1_M \in \mathscr E$, we obtain the second assertion of Lemma \ref{lem:0211}.
\end{proof}

We need the following lemma to apply the implicit function theorem to $\widetilde G$.
For each $\theta \in \mathbb C$, define $\widetilde{\mathbf L }_{1, \theta } ^*: L^\infty (\Omega , \mathscr B^\prime ) \to L^\infty (\Omega , \mathscr B^\prime )$ by
\[
\widetilde{\mathbf L }_{1, \theta } ^* (u) = \mathbf L _{1, \theta } ^* (w_0 +u) ,\quad
 \mathbf L _{1, \theta } ^* (v) (\omega ) ( \varphi ) = v(\sigma \omega ) ( \mathcal L_{1, \theta , \omega } \varphi )
\]
for $u, v\in L^\infty (\Omega , \mathscr B^\prime ) $, $\omega \in \Omega$ and $\varphi \in \mathscr B$. 
\begin{lem}\label{lem:2.4v2}  
For each $\theta \in \mathbb B$ and $u \in  \mathrm{Ker} (\mathbf J)$,
\[
\partial _1  \widetilde{\mathbf G}(\theta , u )  = \frac{  \mathbf J \circ \widetilde{ \mathbf L}_\theta ^* (u ) \cdot \widetilde{\mathbf L }_{1, \theta } ^*(u ) - \widetilde{\mathbf L}_\theta ^* (u ) \cdot  \mathbf J \circ  \widetilde{\mathbf L}_{1, \theta } ^*(u ) }{( \mathbf J \circ \widetilde{\mathbf L} _\theta ^*(u ) )^2}  
\]
and 
\[
\partial _2  \widetilde{\mathbf G}(\theta , u ) =   \frac{ \mathbf J \circ \widetilde{\mathbf L}_\theta ^* (u )  \cdot \mathbf L_\theta ^* - \widetilde{\mathbf L}_\theta ^*( u ) \cdot  \mathbf J \circ \mathbf L_\theta ^*}{( \mathbf J \circ \widetilde{\mathbf L} _\theta ^*(u ) )^2} - \mathbf{id} .
\]
In particular, 
\[
 \partial _2  \widetilde{\mathbf G} ( 0 , 0 ) 
 =   \widetilde{\mathbf L}^*  - \mathbf{id}.
\]
 \end{lem}
 \begin{proof}
 Define $\mathbf M^* : \mathbb C \times  L^\infty (\Omega , \mathscr B^\prime ) \to  L^\infty (\Omega , \mathscr B^\prime )$ by $\mathbf M^*(\theta , u) = \mathbf L^* _\theta (u)$, and $\widetilde{\mathbf M}^* : \mathbb C \times  L^\infty (\Omega , \mathscr B^\prime ) \to  L^\infty (\Omega , \mathscr B^\prime )$ by $\widetilde{\mathbf M} ^* (\theta , u) = \mathbf M^*(\theta , w_0 +u)$.
Since $\mathbf J$ and $\mathbf L^* _\theta$ are linear with $\theta \in \mathbb C$, 
\begin{equation}\label{eq:0212b1}
D\mathbf J (u)= \mathbf J 
\quad \text{and} \quad 
\partial _2 \widetilde{\mathbf M} ^* (\theta , u) = \widetilde{\mathbf L}^* _\theta
\end{equation}
 for all $\theta \in \mathbb C$ and $u\in L^\infty (\Omega , \mathscr B^\prime )$. 
 
 We will show that
 \begin{equation}\label{eq:0212b2}
 \partial _1 \widetilde{\mathbf M} ^* (\theta , u) (t)= t \widetilde{\mathbf L}^* _{1, \theta }
 \end{equation}
  for all $\theta ,t \in \mathbb C$ and $u\in L^\infty (\Omega , \mathscr B^\prime )$. 
  For each $\omega \in \Omega$ and $\varphi \in \mathscr B$ with $\Vert \varphi \Vert =1$,
 \begin{multline*}
\left\vert \left(\mathbf M^* (\theta +t ,u ) -\mathbf M^*(\theta , u ) -t\mathbf L_{1,\theta}^* (u )\right) ( \omega ) (\varphi ) \right\vert \\
=
\left\vert u (\sigma \omega )\left(\mathcal L_{\theta +t , \omega } -\mathcal L_{\theta ,\omega } -t\mathcal L_{1,\theta , \omega } \right) (\varphi ) \right\vert \\
\leq \Vert u \Vert _{L^\infty (\Omega , \mathscr B^\prime )}
\left\Vert \left(\mathcal L_{\theta +t , \omega } -\mathcal L_{\theta ,\omega } -t\mathcal L_{1,\theta , \omega } \right) (\varphi ) \right\Vert .
\end{multline*}
Hence, it follows from \eqref{eq:0212b5} that 
\[
\lim _{t\to 0} \frac{\left\Vert \mathbf M^* (\theta +t ,u ) -\mathbf M^*(\theta , u ) -t\mathbf L_{1,\theta}^* (u )\right\Vert _{L^\infty (\Omega , \mathscr B^\prime )}}{\vert t\vert } =0 ,
\]
and we get \eqref{eq:0212b2}.

    On the other hand, 
by \eqref{eq:0212} and \eqref{eq:0212b}, we have that 
\[
\mathbf J\circ \widetilde{ \mathbf L}^* _{1, 0} (0) ( \omega ) = w_{0, \sigma \omega} \left(\mathcal L_{1, 0, \omega}   h_{0,\omega } \right) = \mathbf I \circ \widetilde{\mathbf L}_{1, 0} (0) (\sigma \omega ) =0 
\]
for $\mathbb P$-almost every $\omega \in \Omega$.
 Therefore, Lemma \ref{lem:2.4v2}  follows from \eqref{eq:0212b1}, \eqref{eq:0212b2} in a similar manner to the proof of Lemma \ref{lem:2.4}.
\end{proof}

Now we can check that $\partial _2\widetilde{\mathbf G} (0, 0): \mathrm{Ker} (\mathbf J) \to \mathrm{Ker} (\mathbf J)$ is bijective by virtue of Lemma \ref{lem:2.4v2} in a manner similar to the proof of the bijectivity of  $\partial _2\widetilde{\mathbf F} (0, 0) : \mathrm{Ker} (\mathbf I) \to \mathrm{Ker} (\mathbf I)$
   (by using $\mathrm{dim} (\mathscr F^*_{0, \omega } ) =1$ shown in the proof of Lemma \ref{lem:0211} instead of  $\mathrm{dim} (\mathscr F_{0, \omega } ) =1$ from (UG), and \eqref{eq:0212c} in Lemma \ref{lem:0211} instead of \eqref{eq:0212c2} in (UG)).

Note that $\widetilde{\mathbf G}(0,0)=0$.
It follows from the implicit function theorem for Banach spaces 
together with the estimates in this subsection that
there is a small ball  $\mathbb B$ in $\mathbb C $ centered at $0$ and a $\mathscr C^1$ map $\xi : \mathbb B\to \mathrm{Ker} (\mathbf J)$ such that $\widetilde{\mathbf G}(\theta , \xi (\theta )) =0$.
Hence, by the definition of $\widetilde{\mathbf G}$ and the fact that $h : \mathbb B\to L^\infty (\Omega , \mathscr B )$ is a $\mathscr C^1$ map, one can check that if we let
\begin{equation*}
\widetilde w _{\theta  } =  w_{0} + \xi (\theta ), \quad w_{\theta ,\omega } = \frac{\widetilde w_{\theta ,\omega }}{\widetilde w_{\theta ,\omega } (h_{\theta ,\omega })} \quad \text{and} \quad \widetilde{\lambda } _{\theta ,\omega } =  \frac{(\mathbf J \circ \mathbf L_\theta ^* (\widetilde w_{\theta } ) )_\omega \widetilde w_{\theta ,\omega } (h_{\theta ,\omega })}{\widetilde w_{\theta ,\sigma \omega } (h_{\theta ,\sigma \omega })}
\end{equation*}
for $ \theta \in \mathbb B$ and $\omega \in \Omega$,
then $w: \mathbb B\to L^\infty (\Omega , \mathscr B^\prime )$ and $\widetilde{\lambda }: \mathbb B \to L^\infty (\Omega , \mathbb C)$ are $\mathscr C^1$ maps and satisfy
\[
\mathcal L_{\theta , \omega } ^* \widetilde{w}_{\theta , \sigma \omega } =  (\mathbf J \circ \mathbf L_\theta ^* (\widetilde w_{\theta } ) )_{\omega} \widetilde w_{\theta , \omega } \quad \text{$\mathbb P$-almost surely}
\]
and so, 
\[
\mathcal L_{\theta , \omega } ^* w_{\theta , \sigma \omega } =\frac{ \mathcal L_{\theta , \omega } ^* \widetilde w_{\theta , \sigma \omega } }{\widetilde w_{\theta ,\sigma \omega } (h_{\theta ,\sigma \omega })}
=\frac{  (\mathbf J \circ \mathbf L_\theta ^* (\widetilde w_{\theta } ) )_{\omega} \widetilde w_{\theta , \omega } }{\widetilde w_{\theta ,\sigma \omega } (h_{\theta ,\sigma \omega })}
=\widetilde{\lambda} _{\theta , \omega } w_{\theta , \omega } \quad \text{$\mathbb P$-almost surely}.
\] 
It also holds by construction that $w_{\theta , \omega } (h_{\theta ,\omega } ) =1$ $\mathbb P$-almost surely.
Therefore, repeating the argument in \eqref{eq:0212d}, we can see that $\widetilde{\lambda} _{\theta , \omega } = \lambda _{\theta , \omega }$.
This completes the proof of the item (4) of Theorem \ref{thm:1204}.

\subsection{Central limit theorem}

Now we can prove Theorem \ref{thm:clt2} by using Theorem \ref{thm:1204}. 
By L\'evy's continuity theorem, it suffices to show that 
\begin{equation}\label{eq:0212v4}
\lim _{n\to \infty} \mathbb E_{\mu _\omega }\left[e^{it \frac{(S_n g)_\omega }{\sqrt n} }\right] = e^{-\frac{
 Vt^2}{2}}  
\end{equation}
for all $t\in \R$ and $\mathbb P$-almost every $\omega \in \Omega$. 
 On the other hand, if we define $I_{n, \omega }:\mathbb C\to \mathbb C$  by
\begin{equation}\label{eq:0212v3b}
 I_{n,\omega }(\theta ) = \mathbb E_{\mu _\omega }\left[e^{\theta (S_n g)_\omega  }\right]   , \quad \theta \in \mathbb C ,
\end{equation}
then
 \begin{equation}\label{eq:0212v3}
\mathbb E_{\mu _\omega }\left[e^{it \frac{(S_n g)_\omega }{\sqrt n} }\right] =I _{n, \omega } \left( \frac{it}{\sqrt n}\right ),
\end{equation}
and it is straightforward to see that 
\begin{equation}\label{eq:0212v2}
I _{n, \omega }  (\theta ) = \langle \mathcal L_{\theta ,\omega }^{(n)}h _{0, \omega },  1_M  \rangle 
\end{equation}
 (recall \eqref{eq:2}, \eqref{eq:1102f} and (UG); cf.~\cite[Lemma 3.3]{DFGV2018}). 
Let $\mathscr B= \mathscr F_{\theta , \omega } \oplus \mathscr R_{\theta , \omega }$ be the invariant splitting for the cocycle $(\mathcal L_\theta ,\sigma)$ given in Subsection \ref{subsection:A1}. 
Then, it follows from \cite[Lemma 2.6]{DFGV2018} that $\mathscr F_{\theta ,\omega } = \Pi _{\theta ,\omega }\mathscr B$ and $\mathscr R_{\theta , \omega } =\widetilde{\Pi } _{\theta , \omega }\mathscr B$, where $ \Pi  _{\theta , \omega } $ and $\widetilde{ \Pi } _{\theta , \omega } $ are bounded operators on $\mathscr B$ given by
\[
 \Pi _{\theta , \omega } \varphi = w_{\theta , \omega} (\varphi ) h_{\theta , \omega }, \quad \widetilde{ \Pi } _{\theta , \omega } \varphi  =  \varphi - \Pi _{\theta , \omega } \varphi    \quad \text{for $\varphi \in \mathscr B$}
\]
(their boundedness are ensured by Theorem \ref{thm:1204}).
\begin{lem}\label{lem:0212b}
There are constants $K>0$ and $0\leq r\leq 1$ such that for $\mathbb P$-almost every $\omega\in \Omega$ and every $\varphi \in \mathscr R_{\theta , \omega }$ and $\psi \in \mathscr B$, 
\begin{equation*}
\left\Vert \mathcal L_{\theta ,\omega }^{(n)} \varphi \right\Vert \leq K r^n  \Vert \varphi \Vert 
\quad \text{and } \quad \left\Vert \widetilde{ \Pi } _{\theta , \omega } \psi \right\Vert \leq K \Vert \psi \Vert .
\end{equation*}
\end{lem}
\begin{proof}
The second inequality Lemma \ref{lem:0212b}  immediately follows from the continuity of $h: \mathbb B\to L^\infty (\Omega , \mathscr B)$ and $w : \mathbb B\to L^\infty (\Omega , \mathscr B^\prime )$  in Theorem \ref{thm:1204}, so we only prove the first inequality. 
Recall  $C_2 >0$ and $\rho \in (0,1)$ given in (UG).
Take $r\in (\rho ,1)$ and $n_0\geq 1$ such that $C_2\rho ^{n_0}<r^{n_0}$, so that 
\[
\esssup _{\omega\in \Omega} 
\left\Vert \mathcal L_{0 ,\omega }^{(n_0)} \vert _{\mathscr R_{0 , \omega }} \right\Vert < r^{n_0}
\]
due to (UG).
Since the map $\theta \mapsto \mathcal L_\theta$ from $\mathbb B$ to the space of bounded operators on $L^\infty (\Omega , \mathscr B)$ with the norm topology is continuous  by  the argument in \eqref{eq:A2theta}, together with the continuity of $h$ and $w$, 
 for any $\theta $ with sufficiently small absolute value, we have
 \[
\esssup _{\omega\in \Omega} 
\left\Vert \mathcal L_{\theta ,\omega }^{(n_0)} \vert _{\mathscr R_{\theta , \omega }} \right\Vert < r^{n_0} .
\]
Hence, for each $n\geq 1$, if we write  $n = kn_0 +\ell$ with $0\leq \ell \leq n_0-1$, then for $\mathbb P$-almost every $\omega \in \Omega$,
\[
\left\Vert \mathcal L_{\theta ,\omega }^{(n)} \vert _{\mathscr R_{\theta , \omega }} \right\Vert \leq  
 \left\Vert \mathcal L_{\theta ,\sigma ^{kn_0} \omega }^{(\ell)} \vert _{\mathscr R_{\theta , \sigma ^{kn_0}\omega }} \right\Vert \prod _{j=0}^{k-1} \left\Vert \mathcal L_{\theta ,\sigma ^{jn_0}\omega }^{(n_0)} \vert _{\mathscr R_{\theta , \sigma ^{jn_0} \omega }} \right\Vert \leq \left( \frac{K_1}{r}\right)^\ell  r^{n} ,
\]
where $K_1= \Vert \mathcal L_{\theta , \omega } \Vert $ which is bounded by (A1) and \eqref{eq:A2theta}.
Therefore, the first inequality of Lemma \ref{lem:0212b} holds with $K= \max \{ 1, (K_1/r) ^{n_0} \}$.
\end{proof}

By (A2), (UG), \eqref{eq:0212v2} and Lemma \ref{lem:0212b},
\begin{multline}\label{eq:14}
\left\vert I _{n,\omega  } (\theta ) - \left(\prod _{j=0}^{n-1}\lambda _{\theta ,\sigma ^j\omega } \right) w_{\theta , \omega} (h_{0, \omega} ) \langle h_{\theta , \omega } ,1_M\rangle \right\vert\\
 =\left\vert  \langle \mathcal L_{\theta , \omega } ^{(n)} \widetilde{ \Pi } _{\theta , \omega } h_{0, \omega }  ,1_M\rangle \right\vert 
 \leq K^2 \Vert h_0 \Vert _{L^\infty(\Omega , \mathscr B)} \Vert 1_M \Vert _{\mathscr E}  r^n .
\end{multline}
Note that 
\begin{equation}\label{eq:15}
 w_{\theta , \omega} (h_{0, \omega} ) \langle h_{\theta , \omega } ,1_M\rangle\to 1 \quad (n\to \infty)
\end{equation}
  by Theorem \ref{thm:1204} and (UG). 
Furthermore, it follows from Theorem \ref{thm:1204} that 
\begin{align*}
\log \left(\prod _{j=0}^{n-1}\lambda _{\theta ,\sigma ^j\omega } \right)
& = \sum _{j=0}^{n-1} \Lambda (\theta )( \sigma ^j \omega) \\
\notag &=  \sum _{j=0}^{n-1} \Lambda ^{\prime \prime} (0) ( \sigma ^j \omega) \theta ^2 +n\cdot o_\omega (\theta ^2) ,
\end{align*}
where $o_\omega (\theta ^2)$ satisfies that $\esssup _{\omega } (\vert o_\omega (\theta ^2)\vert / \theta ^2 ) \to 0$ as $\theta \to 0$. 
On the other hand,
\[
\sum _{j=0}^{n-1} \Lambda ^{\prime \prime} (0) ( \sigma ^j \omega) \left(\frac{it}{\sqrt n}\right) ^2 = -\frac{t^2}{n}\sum _{j=0}^{n-1} \Lambda ^{\prime \prime} (0) ( \sigma ^j \omega) .
\]
Therefore, by Birkhoff's ergodic theorem together with Theorem \ref{thm:1204}, 
\begin{equation}\label{eq:0422a7}
\lim _{n\to \infty } \prod _{j=0}^{n-1}\lambda _{\frac{it}{\sqrt n} ,\sigma ^j\omega } = e^{-\frac{V t^2}{2}} ,
\end{equation}
This together with \eqref{eq:0212v3}, \eqref{eq:14} and \eqref{eq:15} leads to 
\eqref{eq:0212v4},
which completes the proof of Theorem \ref{thm:clt2}.

\subsection{Large deviation principle}

We give the proof of Theorem \ref{thm:ldp2}.
Recall \eqref{eq:0212v3b} for $I_{n,\omega}$ with $n\geq 1$ and $\omega \in \Omega$.
 By the G\"artner-Ellis theorem \cite{HH2000},
 it suffices to show that for  $\mathbb P$-almost every $\omega \in \Omega$, 
\[
\mathbb R \ni \theta \mapsto \lim _{n\to \infty} \frac{1}{n} \log  I _{n, \omega } (\theta )  
\]
is a strictly convex $\Ci ^1$ function with vanishing derivative at $\theta =0$. 
As in \eqref{eq:14}, we get
\[
I _{n,\omega  } (\theta ) = \left(\prod _{j=0}^{n-1}\lambda _{\theta ,\sigma ^j\omega } \right) w_{\theta , \omega} (h_{0, \omega} ) \langle h_{\theta , \omega } ,1_M\rangle 
+  \langle \mathcal L_{\theta , \omega } ^{(n)} \widetilde{ \Pi } _{\theta , \omega } h_{0, \omega }  ,1_M\rangle .
\]
On the other hand, due to Lemma \ref{lem:0212b} and \eqref{eq:15}, one can find a constant $K_0 >0$ such that
\[
K^{-1}_0 \leq \left\vert w_{\theta , \omega} (h_{0, \omega} ) \langle h_{\theta , \omega } ,1_M\rangle  + \frac{ \langle \mathcal L_{\theta , \omega } ^{(n)} \widetilde{ \Pi } _{\theta , \omega } h_{0, \omega }  ,1_M\rangle}{\prod _{j=0}^{n-1}\lambda _{\theta ,\sigma ^j\omega } } \right\vert \leq K_0 ,
\]
so that
\[
 \lim _{n\to \infty} \frac{1}{n} \log  I _{n, \omega } (\theta )  = \lim _{n\to \infty} \frac{1}{n} \log \left(\prod _{j=0}^{n-1}\lambda _{\theta ,\sigma ^j\omega } \right) =\mathbb E_{\mathbb P} \left[ \widehat{\Lambda} (\theta )\right] \quad \text{$\mathbb P$-almost surely}
 \]
by Birkhoff's ergodic theorem. 
Therefore, it follows from Theorem \ref{thm:1204} that $\mathbb R\ni \theta \mapsto \lim _{n\to \infty} \frac{1}{n} \log  I _{n, \omega } (\theta )$ is $\mathbb P$-almost surely a strictly convex $\Ci ^1$ function with vanishing derivative at $\theta =0$, and this completes the proof of Theorem \ref{thm:ldp2}.

\subsection{Local central limit theorem}

We give the proof of Theorem \ref{thm:lclt2}. 
By a standard density argument (cf.~\cite{Morita1989}), one can see that it suffices to show that 
for $\mathbb P$-almost every $\omega \in \Omega$ and every $u \in L^1(\mathbb R)$, if the Fourier transform $\hat u$ of $u$ has compact support, then
\begin{equation}\label{eq:0411}
\lim _{n\to\infty} \sup _{z\in\mathbb R}\left\vert \sqrt{nV } \int u\left(z+ (S_n g)_\omega (\cdot )\right )d\mu _\omega  - \frac{1}{\sqrt{2\pi}} e^{-\frac{z^2}{2nV}} \int _{\mathbb R}u (\zeta )d\zeta \right\vert =0 .
\end{equation}

Fix $u \in L^1(\mathbb R)$ for which  the support of $\hat u$ is included in $[-\delta _0, \delta _0]$ with some 
$\delta _0>0$. 
Given  $\delta \in (0, \delta _0)$, define $A_1$ and $A_2$ by
\[
A_1 
 = \frac{\sqrt{nV }}{2\pi } \int _{  \vert t \vert <\delta } e^{it z}  \hat u\left( t \right)  I_{n, \omega } \left(it \right)   dt, 
\quad 
A_2 
 = \frac{\sqrt{nV }}{2\pi } \int _{ \delta \leq \vert t \vert \leq \delta _0} e^{it z}  \hat u\left( t \right)  I_{n, \omega } \left(it \right)   dt .
\]
Then, by the Fourier inverse formula and Fubini's theorem, it holds 
\begin{align*}
\sqrt{ nV}\int u\left(z+ (S_n g)_\omega (\cdot )\right)d\mu _\omega  & = \sqrt{nV} \int \left( \frac{1}{2\pi } \int  _{\mathbb R}\hat u (t) e^{it (z + (S_ng )_\omega (x) )} dt \right) \mu _\omega (dx)\\
&= A_1 +A_2.
\end{align*}
On the other hand, it is straightforward to see that
\[
\frac{1}{\sqrt{2\pi}} e^{-\frac{z^2}{2nV}} \int _{\mathbb R} u (\zeta )d\zeta 
= \frac{\hat u(0) }{\sqrt{2\pi } } e^{-\frac{z^2}{2nV}} 
 = B_1 +B_2 ,
\]
 where
\[
B_1  = \frac{\sqrt{nV}}{2\pi } \int _{\vert t\vert <\delta } e^{it z}  \hat u(0)  e^{-\frac{nV t^2 }{2}} dt ,
\quad 
B_2  = \frac{\sqrt{nV}}{2\pi } \int _{\vert t\vert \geq \delta } e^{it z}  \hat u(0)  e^{-\frac{nV t^2 }{2}} dt . 
\]
Finally, we decompose $A_1 -B_1$ into three parts $s_1$, $s_2 $ and $s_3$ given by
\[
s_1  =  \frac{\sqrt{nV}}{2\pi } \int _{\vert t \vert <\delta } e^{it z}  \hat u\left( t \right) \left( I_{n, \omega } \left(it \right)  - \left(\prod _{j=0} ^{n-1} \lambda _{it , \sigma ^j \omega}\right)w _{it , \omega} (h_{0, \omega }) \langle h _{it , \omega} 1_M\rangle   \right) dt ,
\]
\[
s_2  = \frac{\sqrt{nV}}{2\pi } \int _{\vert t \vert <\delta }e^{itz}  \hat u\left( t \right)\left(\prod _{j=0} ^{n-1} \lambda _{it , \sigma ^j \omega}\right) \left( w _{it , \omega} (h_{0, \omega }) \langle h _{it , \omega} 1_M\rangle  -1 \right) dt ,
\]
\[
s_3 = \frac{\sqrt{nV}}{2\pi } \int _{\vert t \vert <\delta } e^{itz} \left( \hat u\left( t \right)\left(\prod _{j=0} ^{n-1} \lambda _{it, \sigma ^j \omega}\right) - \hat u(0) e^{- \frac{nV t^2}{2}} \right) dt ,
\]
so that \eqref{eq:0411} immediately follows from that
\begin{equation}\label{eq:0411a}
\lim _{n\to\infty} \sup _{z\in\mathbb R}  \left( \vert A_2  \vert  +  \vert B_2\vert + \vert s_1 \vert + \vert s_2 \vert + \vert s_3 \vert \right) =0
\end{equation}
for some $\delta \in (0, \delta _0)$.

\bf Estimate of $A_2 , B_2, s_1$. \rm
By using condition (L) together with (A2) and (UG),  we can easily get the desired estimate of $\vert A_2  \vert  $:
\[
\sup _{z\in \mathbb R} \vert A_2\vert \leq  \frac{\sqrt{nV}}{2\pi } 2(\delta _0 - \delta ) \Vert \hat u \Vert _{L^\infty}C_1 C _\omega \rho ^n \Vert h_0 \Vert _{L^\infty (\Omega ,\mathscr B)}\Vert 1_M \Vert _{\mathscr E} ,
\]
which goes to $0$ as $n\to \infty $.

Furthermore, it holds that
\[
\sup _{z\in \mathbb R} \vert B_2 \vert 
 \leq  \frac{\vert \hat u(0) \vert  \sqrt V}{2\pi} \int _{\vert t\vert \geq \delta \sqrt n}  e^{- \frac{V t^2}{2}} dt ,
\]
which, due to the dominated convergence theorem and the integrability
of the map $t\mapsto e^{- \frac{Vt^2}{2}}$,  converges to zero as $n\to \infty$.

Finally it follows from \eqref{eq:14} that, if $\delta $ is sufficiently small, then
\[
\lim _{n\to \infty } \sup _{z\in \mathbb R}\vert s_1 \vert 
\leq \lim _{n\to \infty }  \frac{\sqrt{nV}}{2\pi} 2\delta \Vert \hat u\Vert _{L^\infty } K^2 \Vert h_0 \Vert _{L^\infty (\Omega ,\mathscr B)} \Vert 1_M \Vert _{\mathscr E} r^n =0 .
\]

\bf Estimate of $s_2, s_3$. \rm
The following lemma is important to get the desired estimates of $s_2$ and $s_3$.

\begin{lem}\label{lem:0422c}
 For any sufficiently small $\delta >0$ and $\mathbb P$-almost every $\omega \in \Omega$, there exists $n_\omega \in \N$ such that for all $n\ge n_\omega $ and $t\in \mathbb R$ with $\lvert t\rvert <  \delta \sqrt n$,
 \[
  \bigg{\lvert} \prod_{j=0}^{n-1} \lambda_{  \frac{it}{\sqrt n} , \sigma^j \omega} \bigg{\rvert} \le e^{-\frac{t^2 V}{8}}.
 \]
\end{lem}
\begin{proof}
 Recall   the proof of \eqref{eq:0211b} for $\Lambda$.
 Let $R$ denote the reminder of $\Lambda$ of order $2$ in the sense that
 $\Lambda (\theta ) (\omega ) = \frac{\Lambda ^{\prime \prime } (0) (\omega )}{2} \theta ^2 + R(\theta )(\omega )$ holds, so that we have
 \begin{equation}\label{eq:0422b1}
\log  \bigg{\lvert} \prod_{j=0}^{n-1} \lambda_{ \frac{it}{\sqrt n} , \sigma^j \omega} \bigg{\rvert}=-\frac{t^2}{2} \Re \left(\frac 1n \sum_{j=0}^{n-1} \Lambda ^{\prime \prime}(0)(\sigma^j \omega)\right)
  +\Re \left(\sum_{j=0}^{n-1} R(it/\sqrt n)(\sigma^j \omega)\right) .
 \end{equation}
Let $\delta$  be a positive number such that $\lVert  R(\theta)\rVert_{L^\infty (\Omega )} \le  \frac{V}{8} \vert\theta \vert ^2$ whenever $\lvert \theta \rvert \le \delta$.

On the one hand, by \eqref{eq:0422a7}, we have that 
\[
\lim _{n\to \infty} \Re \bigg{(}\frac 1n \sum_{j=0}^{n-1} \Lambda ^{\prime \prime } (0)(\sigma^j \omega) \bigg{)} = V \quad \text{ $\mathbb P$-almost surely},
  \]
and thus for $\mathbb P$-almost every $\omega$ there exists $n_\omega \in \N$ such that
\[
-\frac{t^2}{2}  \Re \bigg{(}\frac 1n \sum_{j=0}^{n-1} \Lambda ^{\prime \prime }(0)(\sigma^j \omega) \bigg{)} \leq -\frac{t^2V}{4} \quad \text{for $n\ge n_\omega$.}
\]
On the other hand, for $t\in \mathbb R$ with $\lvert t\rvert<  \delta \sqrt n$,
 it holds that
\[
 \bigg{\lvert} \sum_{j=0}^{n-1} R(it/\sqrt n)(\sigma^j \omega ) \bigg{\rvert} \le    \sum_{j=0}^{n-1}\frac{V}{8} \left\vert \frac{it}{\sqrt n}\right\vert ^2= \frac{t^2 V}{8} ,
\]
and thus
\[
 \Re \left(\sum_{j=0}^{n-1} R(it/\sqrt n)(\sigma^j \omega )\right) \le \frac{t^2 V}{8}  .
\]
These estimates together with \eqref{eq:0422b1} completes the proof. 
\end{proof}

Now the desired estimate of $\vert s_2  \vert  $ can be shown.
Recall  that $w_{0,\omega } (h_{0,\omega }) = 1$ and $w_\theta$ is differentiable by Theorem \ref{thm:1204}, so that there exists a constant $C > 0$ such that $\vert w_{\theta ,\omega } ( h_{0, \omega }) -1\vert \leq C \vert \theta \vert $
 for every $\theta $ in a neighborhood of $0$ in $\mathbb C$.
 Similar reasoning from the facts about $h_\theta$ in Theorem \ref{thm:1204} leads to that 
 $\vert w_{\theta ,\omega } ( h_{0, \omega }) -1\vert \leq C \vert \theta \vert $
 for every $\theta $ in a neighborhood of $0$ in $\mathbb C$, with some constant $C>0$.
Therefore, it follows from Lemma \ref{lem:0422c} that when $n$ is sufficiently large and $\delta $ is sufficiently small, $\sup _{z\in \mathbb R} \vert s_2\vert $ is bounded by
\begin{multline*}
 \frac{\sqrt V}{2\pi} \int _{\vert t \vert <\delta \sqrt n}\left\vert e^{\frac{itz}{\sqrt n}}  \hat u\left( \frac{t}{\sqrt n} \right)\left(\prod _{j=0} ^{n-1} \lambda _{\frac{it}{\sqrt n} , \sigma ^j \omega}\right) \left( w _{\frac{it}{\sqrt n} , \omega} (h_{0, \omega }) \langle h _{\frac{it}{\sqrt n} , \omega} 1_M\rangle  -1 \right)\right\vert dt \\
\leq  \frac{\sqrt V}{2\pi } \Vert \hat u\Vert _{L^\infty}\int _{\vert t \vert <\delta \sqrt n }  e^{-\frac{t^2V}{8}} C\left\vert \frac{it }{\sqrt n}\right\vert dt 
\end{multline*}
with some constant $C>0$, so we conclude that $\lim _{n\to\infty} \sup _{z\in \mathbb R} \vert s_2\vert =0$.

Furthermore,
\[
\sup _{z\in \mathbb R} \vert s_3\vert 
  \leq \frac{\sqrt V}{2\pi}
 \int _{\vert t \vert <\delta \sqrt n}  \left\vert \hat u\left( \frac{t}{\sqrt n} \right)\left(\prod _{j=0} ^{n-1} \lambda _{\frac{it}{\sqrt n}, \sigma ^j \omega}\right) - \hat u(0) e^{- \frac{V t^2}{2}} \right\vert dt , 
\]
which is bounded by $2\Vert \hat u\Vert _{L^\infty} e^{-\frac{Vt^2}{8}}$ for any sufficiently large $n\in \mathbb N$ and small $\delta >0$ due to Lemma \ref{lem:0422c}.
Hence, the dominated convergence theorem  is applicable: since $\left\vert \hat u\left( \frac{t}{\sqrt n} \right)\left(\prod _{j=0} ^{n-1} \lambda _{\frac{it}{\sqrt n}, \sigma ^j \omega}\right) - \hat u(0) e^{- \frac{V t^2}{2}} \right\vert \to 0$ as $n\to \infty$  by virtue of the continuity of $\hat u$ and \eqref{eq:0422a7}, we get 
$\lim _{n\to \infty} \sup _{z\in \mathbb R} \vert s_3\vert =0$.
This completes the proof of Theorem \ref{thm:lclt2}.

\appendix

\section{The partial captivity condition}\label{app:B}
In this appendix, we briefly recall the partial captivity condition.
Let $T=T_{(E,\tau)}$ be the U(1) extension of an expanding map $E:\Sone \to \Sone$ over $\tau\in\mathscr C^r(\Sone )$ with $r\geq 2$ (given by \eqref{eq:unperturbedsystem} with $T, E, \tau$ in place of $T_0, E_0, \tau _0$). 
Fix some $R>\|\tau'\|_\infty$.
Then the corresponding cone 
$
\mathscr K_R=\{(\xi,\eta)\in\R^2 \mid \lvert\eta\rvert\le \frac{R}{\inf _xE' (x)-1} \lvert\xi\rvert\}
$ 
is (forward) invariant under the Jacobian matrix
\[
D T(z)=\left( \begin{array}{cc} E  '(x) & 0\\ \tau'(x) & 1\end{array} \right)
\quad z=(x,s)\in\T^2.
\]
 For  $z\in\T^2$ and  $n\ge1$, let us consider  
the images of $\mathscr K_R$ by $DT^n$ in $T_z\mathbb T^2$, i.e. 
\[
DT^n(\zeta) \mathscr K_R \quad \text{for $\zeta\in T^{-n}(z)$}.
\]
It is not difficult to see that $\tau$ is  cohomologous to a constant if and only if all the above cones have a line in common at every point $z\in \mathbb T^2$ and $n\ge 1$. Thus we naturally come to the idea of considering  transversality between the above cones.  
As a way to quantify this notion, we set
\begin{equation*}
\ncal (\tau )=\lim_{n\to\infty}\left(\sup_{z\in \mathbb T^2} \sup_{v\in \mathbb R^2, \, \vert v\vert =1}
\# \{ \zeta\in T^{-n}(z)\mid v\in DT^n(\zeta)\mathscr K_R\}\right)^{1/n}
\end{equation*}
 (see \cite{NTW2016} for the existence of the limit). 
Note that $\ncal (\tau)\leq k$ with equality when $\tau$ is cohomologous to a constant. 
We say that $T=T_{(E,\tau)}$ is \emph{partially captive} if  $\ncal (\tau)=1$. 
\begin{thm}[\cite{NTW2016}]\label{conj:maingoal}
Let $r\ge 2$ and suppose that the expanding map $E:\Sone\to\Sone$ is fixed. For every $\varrho>1$, there is an open dense subset 
$\mathscr V_\varrho\subset \Ci^r(\Sone)$ such that 
if $\tau\in \mathscr V_\varrho$ then 
\[
\ncal (\tau)<\varrho.
\]
Consequently, there is a residual subset 
$\mathscr R\subset \Ci^r(\Sone)$ such that 
$T=T_{(E,\tau)}$ is partially captive  for every $\tau\in \mathscr R$. 
\end{thm}

\section*{Acknowledgments}

We would like to express our profound gratitude to Sandro Vaienti for many fruitful discussions
and valuable comments. 
We also thank Yeor Hafouta for insightful comments.
This work was partially supported by JSPS KAKENHI
Grant Numbers 19K14575, 19K21834, 20K03631. Research of Jens Wittsten was partially supported by The Swedish Research Council grant 2019-04878.

\bibliography{CNW}

\end{document}